\documentclass[english,reqno, 11pt]{amsart}
\usepackage{amsaddr}
\usepackage[T1]{fontenc}
\usepackage[latin9]{inputenc}
\usepackage{amscd}
\usepackage{amsthm,amsmath,amsfonts,amssymb}
\usepackage{verbatim}
\usepackage{lmodern}
\usepackage{graphicx}
\usepackage{enumerate}
\usepackage[percent]{overpic}
\usepackage[numbers,sort&compress]{natbib}

\newcommand{\res}{\text{\upshape Res\,}}

\newcommand{\re}{\text{\upshape Re\,}}
\newcommand{\Arg}{\text{\upshape Arg\,}}
\newcommand{\im}{\text{\upshape Im\,}}

\allowdisplaybreaks

\newtheorem{figuretext}{Figure}

\usepackage[colorlinks=true]{hyperref}
\hypersetup{urlcolor=blue, citecolor=red, linkcolor=blue}

\makeatletter
\numberwithin{equation}{section}
\numberwithin{figure}{section}
\theoremstyle{plain}
\newtheorem{thm}{\protect\theoremname}
  \theoremstyle{plain}
  \newtheorem{lemma}[thm]{\protect\lemmaname}
   \newtheorem{prop}[thm]{\protect\propname}

\newtheorem{remark}[thm]{Remark}

\numberwithin{thm}{section}

\newtheorem{cor}[thm]{Corollary}

\theoremstyle{remark}
\newtheorem*{rem}{Remark}
\makeatother

\usepackage{babel}
\providecommand{\propname}{Proposition}

\providecommand{\lemmaname}{Lemma}
\providecommand{\theoremname}{Theorem}

\newcommand{\imag}{\operatorname{Im} }
\newcommand{\real}{\operatorname{Re} }
\renewcommand{\Im}{\imag}
\renewcommand{\Re}{\real}

\newcommand{\ee}{\epsilon}

\newcommand{\HH}{\mathbb{H}}
\newcommand{\DD}{\mathbb{D}}

\newcommand{\RR}{\mathbb{R}}

\newcommand{\E}{\mathbf{E}}
\newcommand{\PP}{\mathbf{P}}
\newcommand{\Prob}{\mathbf{P}}

\newcommand{\hcap}{\operatorname{hcap}}
\newcommand{\diam}{\operatorname{diam}}

\newcommand{\dist}{\operatorname{dist}}
\newcommand{\hm}{\omega}

\newcommand{\ball}{\mathcal{B}}

\newcommand{\Z}{{\mathbb Z}}

\newcommand{\R}{{\mathbb{R}}}
\newcommand{\C}{{\mathbb C}}
\newcommand{\rad}{{\rm rad}}

\newcommand{\GG}{{\mathcal G}} 

\let \le \leqslant
\let \leq \leqslant
\let \ge \geqslant
\let \geq \geqslant
\input epsf

\title[Schramm's formula and Green's function]{Schramm's formula and the Green's function \\for multiple SLE}
\author{Jonatan Lenells and Fredrik Viklund}
\begin{document}
\address{Department of Mathematics, KTH Royal Institute of Technology, \\ 100 44 Stockholm, Sweden.}
\email{jlenells@kth.se, fredrik.viklund@math.kth.se}

\begin{abstract}
We construct martingale observables for systems of multiple SLE curves by applying screening techniques within the CFT framework recently developed by Kang and Makarov, extended to admit multiple SLEs.
We illustrate this approach by rigorously establishing explicit expressions for the Green's function and Schramm's formula in the case of two curves growing towards infinity. 
In the special case when the curves are ``fused'' and start at the same point, some of the formulas we prove were predicted in the physics literature.
\end{abstract}

\maketitle




\setcounter{tocdepth}{1}
\tableofcontents

\section{Introduction}
Schramm-Loewner evolution (SLE) processes are universal lattice size scaling limits of interfaces in critical planar lattice models. SLE with parameter $\kappa > 0$ is a random continuous curve constructed using Loewner's differential equation driven by Brownian motion with speed $\kappa$. Solving the Loewner equation gives a continuous family of conformal maps and the SLE curve is obtained by tracking the image of the singularity of the equation. Various geometric observables are useful and important in SLE theory. To name a few examples, the SLE Green's function, i.e., the renormalized probability that the interface passes near a given point, is important in connection with the Minkowski content parametrization \cite{LR2015}; Smirnov proved Cardy's formula for the probability of a crossing event in critical percolation which then entails conformal invariance \cite{Sm2001}; left or right passage probabilities known as Schramm formulae \cite{S2001} were recently used in connection with finite Loewner energy curves \cite{YW2018}; and observables involving derivative moments of the SLE conformal maps are important in the study of fractal and regularity properties, see, e.g., \cite{BS2009, LV2012}. By the Markovian property of SLE, such observables give rise to martingales with respect to the natural SLE filtration and, conversely, it is sometimes possible to construct martingales carrying some specific geometric information about the SLE. 

Assuming sufficient regularity, SLE martingale observables satisfy differential equations which can be derived using It\^o calculus. If a solution of these differential equations with the correct boundary values can be found, it is sometimes possible to apply a probabilistic argument using the solution's boundary behavior to show that the solution actually represents the desired quantity. In the simplest case, the differential equation is an ODE, but generically, in the case of multipoint correlations, it is a semi-elliptic PDE in several variables and it is difficult to construct solutions with the desired boundary data. (But see \cite{D2007, BS2009, KP2016a}.) Seeking new ways to construct explicit solutions and methods for extracting information from them therefore seems to be worthwhile. 
 
 It was observed early on that the differential equations that arise in this way in SLE theory also arise in conformal field theory (CFT) as so-called level two degeneracy equations satisfied by certain correlation functions, see \cite{BPZ1984, C1984, BB2002, BB2003, D2015b}. As a consequence, a clear probabilistic and geometric interpretation of the degeneracy equations is obtained via SLE theory. On the other hand, CFT is a source of ideas and methods for how to systematically construct solutions of these equations, cf. \cite{BB2003, BBK2005, KM2013}. Thus CFT provides a natural setting for the construction of martingale observables for SLE processes. 

In \cite{KM2013}, a rigorous Coulomb gas framework was developed in which CFTs are modeled using families of fields built from the Gaussian free field (GFF). SLE martingales for any $\kappa$ can then be represented as GFF correlation functions involving special fields inserted in the bulk or on the boundary. By making additional, carefully chosen, field insertions, the scaling behavior at the insertion points can in some cases be prescribed. In this way many chordal SLE martingale observables were recovered in \cite{KM2013} as CFT correlation functions. 

Multiple SLE arises, e.g., when considering scaling limits of models with alternating boundary conditions forcing the existence of several interacting interfaces. See \cite{D2006} for several examples and results closely related to those of the present paper. Many single-path observables generalize to this setting but when considering several paths, additional boundary points need to be marked thus increasing the dimensionality of the problem. One purpose of this paper is to suggest and explore a method for the explicit construction of at least some martingale observables for multiple SLE starting from single-path observables and exploiting ideas based in CFT considerations. Boundary insertions are easier to handle than insertions in the bulk, so multiple SLE provides a natural first arena in which to consider the extension of one-point functions to multipoint correlations. 

The method involves three steps:
 
\begin{enumerate}[$(1)$]
\item{The first step uses screening techniques and ideas from CFT to generate a non-rigorous prediction for the observable \cite{DF1984, KM2013, AKM} (see also \cite{DMS1997, KP2016a}). The prediction is expressed in terms of a contour integral with an explicit integrand. 
We refer to these integrals as \emph{Dotsenko-Fateev integrals} (after \cite{DF1984}) or sometimes simply as screening integrals. The main difficulty is to choose the appropriate integration contour, but this choice can be simplified by considering appropriate limits. }

\item{The second step is to prove that the prediction from Step 1 satisfies the correct boundary conditions. This technical step involves the computation of rather complicated integral asymptotics. In a separate paper \cite{LV2018B}, we present an approach for computing such asymptotics and carry out the estimates required for this paper.}

\item{The last step is to use probabilistic methods together with the estimates of Step 2 to rigorously establish that the prediction from Step 1 gives the correct quantity.}
\end{enumerate}

\begin{rem}
Step 1 can be viewed as a way to ``add'' a commuting SLE curve to a known observable by first inserting an appropriate marked boundary point and then employing screening to readjust the boundary conditions.
\end{rem}
\begin{rem}We stress that we do not need use \emph{a-priori} information on the regularity of the considered observables as would be the case, e.g., if one would work directly with the differential equations.\end{rem}


\subsection{Two examples}
We illustrate the method by presenting two examples in detail. Both examples involve a system of two SLEs aiming for infinity with one marked point in the interior, but we will indicate how the arguments may be generalized to more complicated configurations. 

The first example concerns the probability that the system of SLEs passes to the right of a given interior point; that is, the analogue of Schramm's formula \cite{S2001}. This probability obviously depends only on the behavior of the leftmost curve. (So it is really an SLE$_{\kappa}(2)$ quantity.) The main difficulty in this case lies in implementing Steps 1 and 2; the latter step is discussed in detail in \cite{LV2018B}.

The second example concerns the limiting renormalized probability that the system of SLEs passes near a given point, that is, the Green's function. We first check that this Green's function actually exists as a limit. The main step is to verify existence in the case when only one of the two curves grows. We complete this step by establishing the existence of the SLE$_\kappa(\rho)$ Green's function when the force point lies on the boundary and $\rho$ belongs to a certain interval. The proof gives a representation formula in terms of an expectation for two-sided radial SLE stopped at its target point; the formula is similar to that obtained in the main result of \cite{AKL2012}. In Step 1, we make a prediction for the observable by choosing an appropriate linear combination of the screening integrals such that the leading order terms in the asymptotics cancel (thereby matching the asymptotics we expect). In Step 2, which is detailed in \cite{LV2018B}, we carefully analyze the candidate solution and estimate its boundary behavior. Lastly, given these estimates, we show that the candidate observable enjoys the same probabilistic representation as the Green's function defined as a limit -- so they must agree.    

For both examples, the  asymptotic analysis of the screening integrals in Step 2 is quite involved. 
The integrals are natural generalizations of hypergeometric functions and belong to a class sometimes referred to as Dotsenko-Fateev integrals. 
Even though Dotsenko-Fateev integrals and other generalized hypergeometric functions have been considered before in related contexts (see, e.g., \cite{DF1984, D2007, KP2016a, KP2016, GI2016}), we have not been able to find the required analytic estimates in the literature. We discuss these issues in a separate paper \cite{LV2018B} which also includes details of the precise estimates needed here.

\subsection{Fusion}
By letting the seed points of the SLEs collapse, we obtain rigorous proofs of \emph{fused} formulas as corollaries. One can verify by direct computation that the limiting one-point observables satisfy specific third-order ODEs which can be alternatively obtained from the non-rigorous fusion rules of CFT, cf. \cite{DMS1997}. 
In fact, in the case of the Schramm probability, the formulas we prove here were predicted using fusion in \cite{GC2006}. The formulas we derive for the fused Green's functions appear to be new.

The interpretation of fusion in SLE theory as the successive merging of seeds was described in \cite{D2015b}. In \cite{D2015b}, the difficult fact that fused one-point observables actually do satisfy higher order ODEs was also established.  
The ODEs for the Schramm formula for several fused paths were derived rigorously in \cite{D2015b} and the two-path formula in the special case $\kappa=8/3$ (also allowing for two interior points) was established in \cite{BJV2013}. We do not need to use any of these results in this paper.

Regarding the solution of the equation corresponding to Schramm's formula for two SLE curves started from two distinct points $x_1,x_2 \in \R$ it was noted in \cite{GC2006} that ``explicit analytic progress is only feasible in the limit when $\delta = x_2 - x_1 \to 0$'', that is, in the fusion limit. It is only by applying the screening techniques mentioned above that we are able to obtain explicit expressions for Schramm's formula in the case of two distinct point $x_1 \neq x_2$ in this paper.



\subsection{Outline of the paper} 
The main results of the paper are stated in Section \ref{mainsec}, while we review some aspects of  SLE$_\kappa$ and SLE$_\kappa(\rho)$ processes in Section \ref{sect:prelim}.

In Section \ref{martingalesec}, we review and use ideas from CFT to generate predictions for Schramm's formula and Green's function for multiple SLEs with two curves growing toward infinity in terms of screening integrals. 

In Section \ref{schrammsec}, we prove rigorously that the predicted Schramm's formula indeed gives the probability that a given point lies to the left of both curves. The proof relies on a number of technical asymptotic estimates; proofs of these estimates are given in \cite{LV2018B}.

In Section \ref{greensec}, we give a rigorous proof that the predicted Green's function equals the renormalized probability that the system passes near a given point. The proof relies both on pure SLE estimates (established in Sections~\ref{greensec}-\ref{lemmasec}) and on asymptotic estimates for contour integrals (established in \cite{LV2018B}). 

In Section \ref{lemmasec}, we prove a lemma which expresses the fact that it is very unlikely that both curves in a commuting system get near a given point.

In Section \ref{fusionsec}, we consider the special case of two fused curves, i.e., the case when both curves in the commuting system start at the same point. In the case of Schramm's formula, this provides rigorous proofs of some predictions for Schramm's formula due to Gamsa and Cardy \cite{GC2006}.


In Appendix \ref{alphaintegersec}, we consider the Green's function when $8/\kappa$ is an integer and derive explicit formulas in terms of elementary functions in a few cases.

\subsection{Acknowledgements} 
Lenells acknowledges support from the European Research Council, Consolidator Grant No. 682537, the Swedish Research Council, Grant No. 2015-05430, and the Gustafsson Foundation, Sweden. Viklund acknowledges support from the Knut and Alice Wallenberg Foundation, the
Swedish Research Council, the National Science Foundation, and the Gustafsson Foundation, Sweden. 

It is our pleasure to thank Julien Dub\'edat and Nam-Gyu Kang for interesting and useful discussions, Dapeng Zhan for a helpful comment on a previous version of the paper, and Tom Alberts, Nam-Gyu Kang, and Nikolai Makarov for sharing with us ideas from their preprint \cite{AKM}. 

We would also like to thank the referees for their careful reading of the manuscript and for providing many suggestions that helped improve the quality of our paper.

\section{Main results}\label{mainsec}
Before stating the main results, we briefly review some relevant definitions.  

Consider first a system of two SLE paths $\{\gamma_j\}_1^2$ in the upper half-plane $\mathbb{H} := \{\im z > 0\}$ growing toward infinity. 
Let $0 < \kappa \le 4$. Let $(\xi^1, \xi^2) \in \R^2$ with $\xi^1 \neq \xi^2$.  
The Loewner equation corresponding to two growing curves is
\begin{align}\label{multiplegtdef}
dg_t(z) = \frac{\lambda_1dt}{g_t(z) - \xi^{1}_t} + \frac{\lambda_2 dt}{g_t(z) - \xi^{2}_t}, \quad g_0(z) = z,
\end{align}
where $\xi_t^1$ and $\xi_t^2$, $t \ge 0$, are the driving terms for the two curves and the growth speeds $\lambda_j $ satisfy $\lambda_j \ge 0$.
The solution of (\ref{multiplegtdef}) is a family of conformal maps $(g_t(z))_{t \geq 0}$ called the Loewner chain of $(\xi_t^1, \xi_t^2)_{t \ge 0}$. 
The {\it  multiple SLE system started from $(\xi^1,\xi^2)$} is obtained by taking $\xi_t^1$ and $\xi_t^2$ as solutions of the system of SDEs
\begin{align}\label{multiplexidef}
\begin{cases}
 d\xi_t^{1} = \frac{\lambda_1 +\lambda_2 }{\xi_t^{1} - \xi_t^{2}} dt + \sqrt{\frac{\kappa}{2} \lambda_1 } dB_t^1,	& \xi_0^1 = \xi^1, 
	\\
 d\xi_t^{2} = \frac{\lambda_1 +\lambda_2 }{\xi_t^{2} - \xi_t^{1}} dt + \sqrt{\frac{\kappa}{2} \lambda_2 } dB_t^2,	& \xi_0^2 = \xi^2, 
\end{cases}
\end{align}
where $B_t^{1}$ and $B_t^2$ are independent standard Brownian motions with respect to some measure $\PP = \PP_{\xi^1, \xi^2}$. The  paths are defined by
\begin{align}\label{gammajtdef}
\gamma_j(t) = \lim_{y \downarrow 0} g^{-1}_t(\xi_t^j + iy), \quad \gamma_{j,t}: = \gamma_j[0,t], \qquad j = 1,2.
\end{align}
For $j = 1,2$, $\gamma_{j,\infty}$ is a continuous curve from $\xi^j$ to $\infty$ in $\HH$. It can be shown that the system (\ref{multiplegtdef}) is commuting in the sense that the order in which the two curves are grown does not matter \cite{D2007}. Since our theorems only concern the distribution of the fully grown curves $\gamma_{1,\infty}$ and $\gamma_{2,\infty}$, the choice of growth speeds is irrelevant. When growing one single curve we will often choose the growth rate to equal $a:=2/\kappa$.

Let us also recall the definition of (chordal) SLE$_\kappa(\rho)$ for a single path $\gamma_1$ in $\mathbb{H}$ growing toward infinity. Let $0 < \kappa < 8$,  $\rho \in \R$, and let $(\xi^1, \xi^2)\in\R^2$ with $\xi^1 \neq \xi^2$. Let $W_t$ be a standard Brownian motion with respect to some measure $\PP^{\rho}$. Then {\it SLE$_\kappa(\rho)$ started from $(\xi^1,\xi^2)$} is defined by the equations
\begin{subequations}
\begin{align}\label{SLEkapparhodefa}
& \partial_t g_t(z) = \frac{2/\kappa}{g_t(z) - \xi_t^1}, \quad g_0(z) = z,
	\\
&  d\xi^1_t = \frac{\rho/\kappa}{\xi^1_t-g_t(\xi^2)}dt + dW_t, \quad \xi_0^1 = \xi^1.
\end{align}
\end{subequations}
Depending on the choice of parameters, a solution may not exist for some range of $t$.
When referring to SLE$_\kappa(\rho)$ started from $(\xi^1,\xi^2)$, we always assume that the curve starts from the first point of the tuple $(\xi^1, \xi^2)$ while the second point (in this case $\xi^2$) is the force point. 
The SLE$_\kappa(\rho)$ path $\gamma_1$ is defined as in (\ref{gammajtdef}), assuming existence of the limit. In general, SLE$_\kappa(\rho)$ need not be generated by a curve, but it is in all the cases considered in this paper. 

The marginal law of either of the SLEs in a commuting system is that of an SLE$_\kappa(\rho), \, \rho =2,$ with the force point at the seed of the other curve. A similar statement also holds for stopped portions of the curve(s), see \cite{D2007}. 
%
%

\subsection{Schramm's formula}
Our first result provides an explicit expression for the probability that an SLE$_\kappa(2)$ path passes to the right of a given point. (See below for a precise definition of this event.) The probability is expressed in terms of the function $\mathcal{M}(z, \xi)$ defined for $z = x + iy \in \mathbb{H}$ and $\xi >0$ by
\begin{align}\nonumber
\mathcal{M}(z, \xi) = &\; y^{\alpha - 2} z^{-\frac{\alpha}{2}}(z- \xi)^{-\frac{\alpha}{2}}\bar{z}^{1-\frac{\alpha}{2}}(\bar{z} - \xi)^{1-\frac{\alpha}{2}}
	\\ \label{Mdef}
&\times \int_{\bar{z}}^z(u-z)^{\alpha}(u- \bar{z})^{\alpha - 2} u^{-\frac{\alpha}{2}}(u - \xi)^{-\frac{\alpha}{2}} du,
\end{align}
where $\alpha = 8/\kappa > 1$ and the integration contour from $\bar{z}$ to $z$ passes to the right of $\xi$, see Figure \ref{Jcontour.pdf}. (Unless otherwise stated, we always consider complex powers defined using the principal branch of the complex logarithm.)

\begin{figure}
\bigskip\medskip
\begin{center}
\begin{overpic}[width=.5\textwidth]{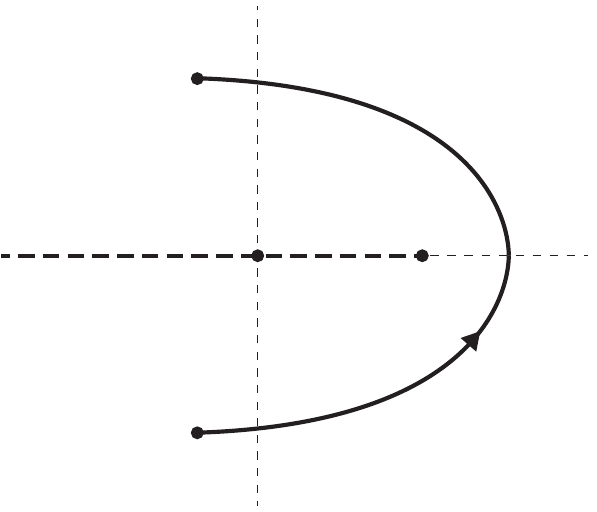}
      \put(28,71.5){\small $z$}
      \put(28,12){\small $\bar{z}$}
      \put(39.5,38){\small $0$}
      \put(69.5,37.5){\small $\xi$}
      \put(103,41.6){\small $\re z$}
      \put(39,89){\small $\im z$}
\end{overpic}
     \begin{figuretext}\label{Jcontour.pdf}
       The integration contour used in the definition (\ref{Mdef}) of $\mathcal{M}(z, \xi)$ is a path from $\bar{z}$ to $z$ which passes to the right of $\xi$. 
       \end{figuretext}
     \end{center}
\end{figure}

\begin{thm}[Schramm's formula for SLE$_\kappa(2)$]\label{thm:schramm}
Let $0 < \kappa \le 4$. Let $\xi > 0$ and consider chordal SLE$_\kappa(2)$ started from $(0,\xi)$. 
Then the probability $P(z, \xi)$ that a given point $z = x + iy \in \mathbb{H}$ lies to the left of the curve is given by 
\begin{align}\label{Pdef}
P(z, \xi) = \frac{1}{c_\alpha} \int_x^\infty \re \mathcal{M}(x' +i y, \xi) dx', \qquad z \in \mathbb{H}, \  \xi > 0,
\end{align}
where the normalization constant $c_\alpha \in \R$ is defined by
\begin{align}\label{calphadef}
c_\alpha & = -\frac{2 \pi ^{3/2} \Gamma \left(\frac{\alpha -1}{2}\right) \Gamma \left(\frac{3 \alpha
   }{2}-1\right)}{\Gamma \left(\frac{\alpha }{2}\right)^2 \Gamma (\alpha )}.
\end{align}
\end{thm}

The proof of Theorem \ref{thm:schramm} will be given in Section \ref{schrammsec}. 
The formula (\ref{Pdef}) for $P(z,\xi)$ is motivated by the CFT and screening considerations of Section \ref{martingalesec}.

A point $z \in \mathbb{H}$ lies to the left of both curves in a commuting system iff it lies to the left of the leftmost curve. Since each of the two curves of a commuting process  has the distribution of an SLE$_\kappa(2)$ (see Section \ref{relationshipsubsubsec}), Theorem \ref{thm:schramm} can be interpreted as the following result for multiple SLE.
%

\begin{cor}[Schramm's formula for multiple SLE]\label{schrammcor}
Let $0 < \kappa \le 4$. Let $\xi > 0$ and consider a multiple SLE$_\kappa$ system in $\mathbb{H}$ started from $(0,\xi)$ and growing toward infinity. 
Then the probability $P(z, \xi)$ that a given point $z = x + iy \in \mathbb{H}$ lies to the left of both curves is given by (\ref{Pdef}).
\end{cor}

Corollary \ref{schrammcor} together with translation invariance immediately yields an expression for the probability that a point $z$ lies to the left of a system of two SLEs started from two arbitrary points $(\xi^1, \xi^2)$ in $\R$. The probabilities that $z$ lies to the right of or between the two curves then follow by symmetry. For completeness, we formulate this as another corollary.

\begin{cor}
Let $0 < \kappa \le 4$. Suppose $-\infty < \xi^1 < \xi^2 < \infty$ and consider a multiple SLE$_\kappa$ system in $\mathbb{H}$ started from $(\xi^1, \xi^2)$ and growing toward infinity. Let $P(z, \xi)$ denote the function in (\ref{Pdef}).
Then the probability $P_{left}(z, \xi^1, \xi^2)$ that a given point $z = x + iy \in \mathbb{H}$ lies to the left of both curves is given by 
$$P_{left}(z, \xi^1, \xi^2) = P(z - \xi^1, \xi^2 - \xi^1);$$
the probability $P_{right}(z, \xi^1, \xi^2)$ that a point $z \in \mathbb{H}$ lies to the right of both curves is
$$P_{right}(z, \xi^1, \xi^2) 
= P(-\bar{z} + \xi^2, \xi^2 - \xi^1);$$
and the probability $P_{middle}(z, \xi^1, \xi^2)$ that $z$ lies between the two curves is given by
$$P_{middle}(z, \xi^1, \xi^2) = 1 -  P_{left}(z, \xi^1, \xi^2)- P_{right}(z, \xi^1, \xi^2).$$
\end{cor}

By letting $\xi \to 0+$ in (\ref{Pdef}), we obtain proofs of formulas for ``fused'' paths. See Section \ref{fusionsec} for a derivation of the following corollary.

\begin{cor}\label{fusioncor}
Let $0 < \kappa \le 4 $ and define $P_{fusion}(z) = \lim_{\xi \downarrow 0} P(z,\xi)$, where $P(z,\xi)$ is as in \eqref{Pdef}. Then
\begin{align}\label{Pfusiondef}
P_{fusion}(z) = \frac{\Gamma(\frac{\alpha}{2})\Gamma(\alpha)}{2^{2 - \alpha}\pi \Gamma(\frac{3\alpha}{2} - 1)} \int_{\frac{x}{y}}^\infty S(t') dt',
\end{align}
where the real-valued function $S(t)$ is defined by
\begin{align*}S(t) = &\; (1 + t^2)^{1- \alpha} 
\bigg\{{}_2F_1\bigg(\frac{1}{2} + \frac{\alpha}{2}, 1 - \frac{\alpha}{2}, \frac{1}{2}; - t^2\bigg)
	\\
& - \frac{2\Gamma(1 + \frac{\alpha}{2})\Gamma(\frac{\alpha}{2})t}{\Gamma(\frac{1}{2} + \frac{\alpha}{2})\Gamma(- \frac{1}{2} + \frac{\alpha}{2})}
 {}_2F_1\bigg(1 + \frac{\alpha}{2}, \frac{3}{2} - \frac{\alpha}{2}, \frac{3}{2}; -t^2\bigg)\bigg\}, \qquad t \in \R.
\end{align*}
\end{cor}

Corollary \ref{fusioncor} provides a proof of the predictions of \cite{GC2006} where the formula (\ref{Pfusiondef}) was derived by solving a third order ODE obtained from so-called fusion rules. (We prove the result for $\kappa \le 4$ but the formulas match those from \cite{GC2006} in general.) We note that even given the explicit predictions of \cite{GC2006}, it is not clear how to proceed to verify them rigorously without additional non-trivial information. Indeed, as soon the evolution starts, the tips of the curves are separated and the system leaves the fused state. However, \cite{D2015b} provides a different rigorous approach by exploiting the hypoellipticity of the PDEs to show that the fused observables satisfy the higher order ODEs. In the special case $\kappa=8/3$, the formula for $P_{fusion}(z)$ was proved in \cite{BJV2013} using Cardy and Simmons' prediction \cite{SC2009} for a two-point Schramm formula.

\subsection{The Green's function}
Our second main result provides an explicit expression for the Green's function for SLE$_\kappa(2)$. 

Let $\alpha = 8/\kappa$. Define the function $I(z,\xi^1, \xi^2)$ for $z \in \mathbb{H}$ and $-\infty < \xi^1 < \xi^2 < \infty$ by
\begin{align}\label{Idef}
I(z,\xi^1, \xi^2) = \int_A^{(z+,\xi^2+,z-,\xi^2-)} (u - z)^{\alpha -1} (u - \bar{z})^{\alpha -1} (u - \xi^1)^{-\frac{\alpha}{2}} (\xi^2 - u)^{-\frac{\alpha}{2}} du,
\end{align}
where $A = (z + \xi^2)/2$ is a basepoint and the Pochhammer integration contour is displayed in Figure \ref{Pochhammernew.pdf}. More precisely, the integration contour begins at the base point $A$, encircles the point $z$ once in the positive (counterclockwise) sense, returns to $A$, encircles $\xi^2$ once in the positive sense, returns to $A$, and so on. The points $\bar{z}$ and $\xi^1$ are exterior to all loops. The factors in the integrand take their principal values at the starting point and are then analytically continued along the contour. The use of a Pochhammer contour ensures that the integrand is analytic everywhere on the contour despite the fact that the integrand involves (multiple-valued) non-integer powers.

\begin{figure}
\begin{center}
\begin{overpic}[width=.55\textwidth]{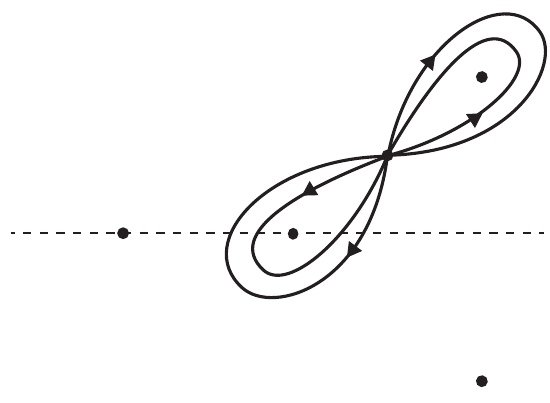}
      \put(71,40.5){\small $A$}
      \put(89.5,58.5){\small $z$}
      \put(89.5,3){\small $\bar{z}$}
      \put(21.2,25.8){\small $\xi^1$}
      \put(52,25.8){\small $\xi^2$}
      \put(103,30){\small $\re z$}
      \put(84.5,53.5){\small $1$}
      \put(58.5,36){\small $2$}
      \put(75.7,63.5){\small $3$}
      \put(66,25){\small $4$}
\end{overpic}
     \begin{figuretext}\label{Pochhammernew.pdf}
       The Pochhammer integration contour in (\ref{Idef}) is the composition of four loops based at the point $A = (z + \xi^2)/2$ halfway between $z$ and $\xi^2$. The loop denoted by $1$ is traversed first, then the loop denoted by $2$ is traversed, and so on.
       \end{figuretext}
     \end{center}
\end{figure}

For $\alpha \in (1, \infty) \smallsetminus \Z$, we define the function $\GG(z, \xi^1, \xi^2)$ for $z = x+iy \in \mathbb{H}$ and $\xi^1 < \xi^2$ by
\begin{align}\label{GGdef}
\GG(z, \xi^1, \xi^2) = \frac{1}{\hat{c}}y^{\alpha + \frac{1}{\alpha} - 2} |z - \xi^1|^{1 - \alpha} |z - \xi^2|^{1 - \alpha}  \Im\big(e^{-i\pi\alpha} I(z,\xi^1,\xi^2) \big), 
\end{align}
where the constant $\hat{c} = \hat{c}(\kappa)$ is given by
\begin{align}\label{hatcdef}
\hat{c} = \frac{4 \sin ^2\left(\frac{\pi  \alpha }{2}\right) \sin (\pi  \alpha ) \Gamma
   \left(1-\frac{\alpha }{2}\right) \Gamma \left(\frac{3 \alpha }{2}-1\right)}{\Gamma
   (\alpha )} \quad \text{with} \quad \alpha = \frac{8}{\kappa}.
   \end{align}
We extend this definition of $\GG(z,\xi^1,\xi^2)$ to all $\alpha > 1$ by continuity.
The following lemma shows that even though $\hat{c}$ vanishes as $\alpha$ approaches an integer, the function $\GG(z,\xi^1,\xi^2)$ has a continuous extension to integer values of $\alpha$. 

\begin{lemma}
For each $z \in \mathbb{H}$ and each $(\xi^1, \xi^2) \in \R^2$ with $\xi^1 < \xi^2$, $\GG(z,\xi^1,\xi^2)$ can be extended to a continuous function of $\alpha \in (1, \infty)$. 
\end{lemma}
\begin{proof}
See Appendix \ref{alphaintegersec}.
\end{proof}

The CFT and screening considerations described in Section \ref{martingalesec} suggest that $\GG$ is the Green's function for SLE$_\kappa(2)$ started from $(\xi^1, \xi^2)$; that is, that $\GG(z, \xi^1, \xi^2)$ provides the normalized probability that an SLE$_\kappa(2)$ path originating from $\xi^1$ with force point $\xi^2$ passes through an infinitesimal neighborhood of $z$. 
Our next theorem establishes this rigorously. 

In the following statements,  $\Upsilon_\infty(z)$ denotes $1/2$ times the conformal radius seen from $z$ of the complement of the curve(s) under consideration (as indicated by the probability measure) in $\HH$. For example, in the case of two paths the conformal radius is with respect to the component of $z$ of $\mathbb{H} \smallsetminus (\gamma_{1,\infty} \cup \gamma_{2,\infty})$.

\begin{thm}[Green's function for SLE$_\kappa(2)$]\label{thm:green}
Let $0< \kappa \le 4$. Let $-\infty < \xi^1 < \xi^2 < \infty$ and consider chordal SLE$_\kappa(2)$ started from $(\xi^1,\xi^2)$. Then, for each $z = x + iy \in \mathbb{H}$,
\begin{equation}\label{greens-formula}
\lim_{\ee \to 0} \ee^{d-2}\PP^2_{\xi^1, \xi^2} \left( \Upsilon_\infty(z) \le \ee \right) = c_* \GG(z,\xi^1, \xi^2),
\end{equation}
where $d=1+\kappa/8$, $\PP^2$ is the SLE$_\kappa(2)$ measure, the function $\GG$ is defined in (\ref{GGdef}), and the constant $c_* = c_*(\kappa)$ is defined by 
\begin{align}\label{cstardef}
c_* = \frac{2}{\int_0^\pi \sin^{4a} x dx} = \frac{2\Gamma\left(1+2a\right)}{\sqrt{\pi } \Gamma \left(\frac{1}{2}+2a\right)} \quad \text{with} \quad a = \frac{2}{\kappa}.
\end{align} 
\end{thm}

The proof of Theorem \ref{thm:green} will be presented in Section \ref{greensec}.

\begin{rem}\upshape
The function $\GG(z, \xi^1, \xi^2)$ can be written as
\begin{align}\label{Gh}
\GG(z, \xi^1, \xi^2) = (\Im z)^{d-2} h(\theta^1, \theta^2), \qquad z \in \mathbb{H}, \  -\infty < \xi^1 < \xi^2 < \infty,
\end{align} 
where $h$ is a function of $\theta^1 = \arg(z - \xi^1)$ and $\theta^2 = \arg(z - \xi^2)$.
This is consistent with the expected translation invariance and scale covariance of the Green's function.
\end{rem}

In the appendix, we derive formulas for $\GG(z, \xi^1, \xi^2)$ when $\alpha$ is an integer. In particular, we obtain a proof of the following proposition which provides explicit formulas for the SLE$_\kappa(2)$ Green's function in the case of $\kappa = 4$, $\kappa = 8/3$, and $\kappa = 2$.

\begin{prop}\label{hexplicit234prop}
For $\kappa = 4$, $\kappa = 8/3$, and $\kappa = 2$ (i.e. for $\alpha = 2,3,4$), the SLE$_\kappa(2)$ Green's function is given by equation (\ref{Gh}) where $h(\theta^1, \theta^2)$ is given explicitly by
\begin{align}\nonumber
h(\theta^1, \theta^2) = &\; \frac{1}{4 \pi\sin (\theta^1-\theta^2)}\big\{\sin (2 \theta^1-2\theta^2)+2 \theta^1(1-\cos{2 \theta^2})+2 \theta^2 (\cos{2 \theta^1} -1)	
\\ \label{halpha2}
&-\sin{2 \theta^1} +\sin{2 \theta^2}\big\}, \qquad \kappa = 4,
\end{align}
\begin{align}\nonumber
h(\theta^1, \theta^2) = &\; \frac{1}{30 \pi  (\cos (\theta^1-\theta^2)+1)}\bigg\{\sqrt{\sin\theta^1 \sin\theta^2}\bigg[-6 \cos\bigg(\frac{\theta^1-3\theta^2}{2}\bigg)
	\\\nonumber
&+\cos\bigg(\frac{3 \theta^1- 5 \theta^2}{2}\bigg)
 +\cos\bigg(\frac{5 \theta^1 - 3 \theta^2}{2}\bigg)
-6 \cos\bigg(\frac{3 \theta^1 - \theta^2}{2}\bigg)
-38 \cos\bigg(\frac{\theta^1+\theta^2}{2}\bigg)
   	\\\nonumber
& +20 \cos\bigg(\frac{3 \theta^1+ 3\theta^2}{2}\bigg)
+14 \cos\bigg(\frac{5 \theta^1+\theta^2}{2}\bigg)
 +14 \cos\bigg(\frac{\theta^1+5 \theta^2}{2}\bigg)\bigg]  
	\\\nonumber
& -2 \cos ^2\left(\frac{\theta^1-\theta^2}{2}\right) \big[-9 \sin{2 \theta^1} \sin{2 \theta^2} +(7 \cos{2\theta^2}+8)\cos{2 \theta^1}+8 \cos2 \theta^2
	\\\label{halpha3}
& -23\big]\arg \left(\cos \left(\frac{\theta^1+\theta^2}{2}\right)+i \sqrt{\sin\theta^1 \sin\theta^2}\right)\bigg\}, \qquad \kappa = 8/3,
\end{align}
and
\begin{align}\nonumber
h(\theta^1, \theta^2) = &\; \frac{1}{192 \pi}\bigg\{\frac{72 \sin ^5(\theta^1) \cos (\theta^2) \cos (\theta^1-3\theta^2)}{\sin^3(\theta^1-\theta^2)} +\frac{\sin^3(\theta^2)}{(\cot\theta^1 -\cot\theta^2)^3} 
	\\\nonumber
& \times \bigg[ 96\frac{(3 \theta^1 \cot \theta^2 +2)\cot \theta^1
+\theta^1(3-2 \csc^2\theta^1)-3 \cot\theta^2}{\sin \theta^1}
	\\\nonumber
& +\csc^6(\theta^2)\big[3 \big(16 \theta^2 (3 \sin (\theta^1-2 \theta^2)+\sin\theta^1) \sin \theta^1 + 5 \sin{2\theta^2} -4 \sin{4 \theta^2}\big) \sin\theta^1
	\\\nonumber
& +6 \cos (\theta^1-6\theta^2)-\cos (3\theta^1- 6\theta^2)
+(75 \cos{2 \theta^2} - 30 \cos{4 \theta^2} -33)\cos\theta^1
	\\\label{halpha4}
& -17 \cos{3 \theta^1}\big]\bigg]\bigg\}, \qquad \kappa = 2.
\end{align}
\end{prop}

It is possible to derive an explicit expression for the Green's function for a system of two multiple SLEs as a consequence of Theorem \ref{thm:green}. To this end, we need a correlation estimate which expresses the intuitive fact that it is very unlikely that both curves pass near a given point $z \in \mathbb{H}$. 
\begin{lemma}\label{2SLElemma} Let $0< \kappa \le 4$. Then,
$$\lim_{\ee \downarrow 0} \ee^{d-2}\PP_{\xi^1, \xi^2} \left( \Upsilon_\infty(z) \le \ee \right)
= \lim_{\ee \downarrow 0} \ee^{d-2}\Big[\PP_{\xi^1, \xi^2}^2 \left(\Upsilon_\infty(z) \le \ee \right) + \PP_{\xi^2, \xi^1}^2 \left( \Upsilon_\infty(z) \le \ee \right) \Big],$$
where $\PP_{\xi^1, \xi^2}$ denotes the law of a system of two multiple SLE$_\kappa$ in $\HH$ started from $(\xi^1, \xi^2)$ and aiming for $\infty$, and $\PP^2_{\xi^1, \xi^2}$ denotes the law of chordal SLE$_\kappa(2)$ in $\HH$ started from $(\xi^1, \xi^2)$. 
\end{lemma}

The proof of Lemma \ref{2SLElemma} will be given in Section \ref{lemmasec}.

Assuming Lemma \ref{2SLElemma}, it follows immediately from Theorem \ref{thm:green} that the Green's function for a system of multiple SLEs started from $(-\xi, \xi)$ is given by
$$G_\xi(z) = \GG(z,-\xi,\xi) + \GG(-\bar{z},-\xi,\xi), \qquad z \in \mathbb{H}, \quad \xi > 0.$$
In other words, given a system of  two multiple SLE$_\kappa$ paths started from $-\xi$ and $\xi$ respectively, $G_\xi(z)$ provides the normalized probability that at least one of the two curves passes through an infinitesimal neighborhood of $z$. 
We formulate this as a corollary.

\begin{cor}[Green's function for multiple SLE]
Let $0< \kappa \le 4$. Let $\xi > 0$ and consider a system of two multiple SLE$_\kappa$ paths in $\mathbb{H}$ started from $(-\xi, \xi)$ and growing towards $\infty$. Then, for each $z = x + iy \in \mathbb{H}$,
\begin{equation}\label{2SLEgreens-formula}
\lim_{\ee \to 0} \ee^{d-2}\PP_{-\xi, \xi} \left( \Upsilon_\infty(z) \le \ee \right) = c_* G_{\xi}(z),
\end{equation}
where $d=1+\kappa/8$, the constant $c_* = c_*(\kappa)$ is given by (\ref{cstardef}), and the function $G_\xi$ is defined for $z \in \mathbb{H}$ and $\xi > 0$ by 
\begin{align*}
G_{\xi}(z) = \frac{1}{\hat{c}} y^{\alpha + \frac{1}{\alpha} - 2} |z + \xi|^{1 - \alpha} |z - \xi|^{1 - \alpha} 
\Im\big[e^{-i\pi\alpha}(I(z,-\xi,\xi) + I(-\bar{z},-\xi,\xi))\big].
\end{align*}
\end{cor}

\begin{remark}\label{Ganyxi1xi2remark}\upshape
If the system is started from two arbitrary points $(\xi^1, \xi^2) \in \R$ with $\xi^1 < \xi^2$, then it follows immediately from (\ref{2SLEgreens-formula}) and translation invariance that 
\begin{align*}
\lim_{\ee \to 0} \ee^{d-2}\PP_{\xi^1, \xi^2} \left( \Upsilon_\infty \le \ee \right)= c_* G_{\frac{\xi^2 - \xi^1}{2}}\bigg(z - \frac{\xi^1 + \xi^2}{2}\bigg).
\end{align*}
\end{remark}

We will prove Theorem \ref{thm:green} by establishing two independent propositions, which when combined imply Theorem~\ref{thm:green}. 
The first of these propositions (Proposition \ref{prop:slekr}) establishes existence of a Green's function for SLE$_\kappa(\rho)$ and provides a representation for this Green's function in terms of an expectation with respect to two-sided radial SLE$_\kappa$. For the proof of Theorem \ref{thm:green}, we only need this proposition for $\rho = 2$. However, since it is no more difficult to state and prove it for a suitable range of positive $\rho$, we consider the general case.
\begin{prop}[Existence and representation of Green's function for SLE$_\kappa(\rho)$]\label{prop:slekr}
Let $0 < \kappa \le 4$ and $0 \le \rho < 8-\kappa$. Given two points $\xi^1, \xi^2 \in \R$ with $\xi^1 < \xi^2$, consider chordal SLE$_\kappa(\rho)$ started from $(\xi^1, \xi^2)$. Then, for each $z \in \mathbb{H}$,
\[
\lim_{\ee \downarrow 0}\ee^{d-2} \PP^{\rho}_{\xi^1, \xi^2} \left( \Upsilon_\infty \le \ee  \right) = c_* G^{\rho}(z, \xi^1, \xi^2),\]
where the SLE$_\kappa(\rho)$ Green's function $G^{\rho}$ is given by
\begin{align}\label{Grhorepresentation}
G^{\rho}(z, \xi^1, \xi^2) = G(z-\xi^1) \, \E^*_{\xi^1,z} \left[M_T^{(\rho)} \right].
\end{align}
Here $G(z) = (\Im z)^{d-2} \sin^{4a-1}( \arg \, z)$ is the Green's function for  chordal SLE$_\kappa$ in $\HH$ from $0$ to $\infty$, the martingale $M^{(\rho)}$ is defined in (\ref{SLEkr_mg}), 
$\E^*_{\xi^1,z}$ denotes expectation with respect to two-sided radial SLE$_\kappa$ from $\xi^1$ through $z$, stopped at $T$, the hitting time of $z$, and the constant $c_*$ is given by \eqref{cstardef}.
\end{prop}
We expect that the analogous result is true for $\kappa \in (0,8)$ and for a wider range of $\rho$. We restrict to the stated ranges for simplicity, as these assumptions simplify some of the arguments, e.g., due to the relationship between the boundary exponent and the dimension of the path.  

The next result (Proposition \ref{G2prop}) shows that the function $\GG(z, \xi^1, \xi^2)$ predicted by CFT and defined in (\ref{GGdef}) can be represented in terms of an expectation with respect to two-sided radial SLE$_\kappa$. Since this representation coincides with the representation in (\ref{Grhorepresentation}), Theorem \ref{thm:green} will follow immediately once we establish Propositions \ref{prop:slekr} and \ref{G2prop}.

\begin{prop}[Representation of $\GG$]\label{G2prop}
Let $0 < \kappa \le 4$ and let $\xi^1, \xi^2 \in \R$ with $\xi^1 < \xi^2$. The function $\GG(z,\xi^1,\xi^2)$ defined in (\ref{GGdef}) satisfies
\begin{align}\label{G2propeq}
\GG(z,\xi^1,\xi^2) = G(z-\xi^1) \, \E^*_{\xi^1,z} \left[M_T^{(2)} \right], \qquad z \in \mathbb{H}, \  0 < \xi < \infty,
\end{align}
where $G(z) = (\Im z)^{d-2} \sin^{4a-1}( \arg \, z)$ is the Green's function for chordal SLE$_\kappa$ in $\HH$ from $0$ to $\infty$ and $\E^*_{\xi^1,z}$ denotes expectation with respect to two-sided radial SLE$_\kappa$ from $\xi^1$ through $z$, stopped at $T$, the hitting time of $z$.
\end{prop}
\begin{rem}\upshape
Note that equation \eqref{G2propeq} gives a formula for the two-sided radial SLE observable,
\[
\E^*_{\xi^1,z} \left[M_T^{(2)} \right] = \frac{\GG(z,\xi^1,\xi^2)}{G(z-\xi^1)},
\]
and as a consequence we obtain smoothness and the fact that it satisfies the expected PDE.
\end{rem}

The proofs of Propositions \ref{prop:slekr} and \ref{G2prop} are presented in Sections \ref{greensubsec1} and  \ref{greensubsec2}, respectively.

In Section~\ref{sec:fused-green}, we obtain fusion formulas by letting $\xi \to 0+$. The formulas simplify for some values of $\kappa$. In particular, we will prove the following result. 

\begin{prop}[Fused Green's functions]
Suppose $\kappa = 4, 8/3$, or $2$. Consider a system of two fused multiple SLE$_\kappa$ paths in $\mathbb{H}$ started from $0$ and growing toward $\infty$. Then, for each $z = x + iy \in \mathbb{H}$,
\begin{align*}
\lim_{\ee \to 0} \ee^{d-2}\PP_{0,0+} \left( \Upsilon_\infty(z) \le \ee \right) = c_* (\GG_f(z) + \GG_f(-\bar{z})),
\end{align*}
where $d=1+\kappa/8$, the constant $c_* = c_*(\kappa)$ is given by (\ref{cstardef}), and the function $\GG_f$ is defined by
\[\GG_f(x+iy) = y^{d-2} h_{f}(\theta)\]
with $h_f(\theta)$ given explicitly by
\begin{align*}
h_f(\theta) = \begin{cases} \frac{2}{\pi}(\sin\theta - \theta \cos\theta) \sin \theta, & \kappa = 4, 
	\\  
\frac{8}{15 \pi} (4 \theta-3 \sin{2 \theta} 
+2 \theta \cos{2 \theta})\sin ^2\theta, & \kappa = 8/3,
	\\ 
\frac{1}{12 \pi}(27 \sin\theta + 11 \sin{3 \theta}	
-6 \theta (9 \cos\theta + \cos{3 \theta})) \sin^3\theta, & \kappa = 2,
\end{cases} \ 0 < \theta < \pi.
\end{align*}
\end{prop}
\subsection{Remarks}
We end this section by making a few remarks.
\begin{itemize}
\item{We believe the method used in this paper will generalize to produce analogous results for observables for $N \ge 3$ multiple SLE paths depending on one interior point. This would require $N-1$ screening insertions, and the integrals will then be $N-1$ iterated contour integrals.}
\item{In \cite{KP2016a, KP2016} screening integrals for SLE \emph{boundary} observables (such as the ordered multipoint boundary Green's function) are given and shown to be closely related to a particular \emph{quantum group}. In fact, this algebraic structure is used to systematically make the difficult choices of integration contours. It seems reasonable to expect that a similar connection exists in our setting as well, allowing for an efficient generalization to several commuting SLE curves, but we will not pursue this  here. }

\item{Another way of viewing the system of two multiple SLEs growing towards $\infty$ is as one SLE path conditioned to hit a boundary point, also known as two-sided chordal SLE. Indeed, the extra $\rho=2$ at the second seed point forces a $\rho = \kappa-8$ at $\infty$.}

\item{Suppose one has an SLE$_{\kappa}$ martingale and wants to construct a similar martingale for SLE$_{\kappa}(\rho)$. The first idea that comes to mind is to try to ``compensate''  the SLE$_\kappa$ martingale by multiplying by a differentiable process. In the cases we consider this method does not give the correct observables (the boundary behavior is not correct), but rather corresponds to a change of coordinates moving the target point at $\infty$.  }
\end{itemize}

\section{Preliminaries}\label{sect:prelim}
Unless specified otherwise, all complex powers are defined using the principal branch of the logarithm, that is, $z^\alpha = e^{\alpha (\ln |z| + i \Arg z)}$ where $\Arg z \in (-\pi, \pi]$.  We write $z = x + iy$ and let
$$\partial = \frac{1}{2}\bigg( \frac{\partial}{\partial x} - i \frac{\partial}{\partial y}\bigg), \qquad \bar{\partial} = \frac{1}{2}\bigg( \frac{\partial}{\partial x} + i \frac{\partial}{\partial y}\bigg).$$
We let $\mathbb{H} = \{z \in \C: \im z > 0\}$ and $\mathbb{D} = \{z \in \C: |z| < 1\}$ denote the open upper half-plane and the open unit disk, respectively. 
The open disk of radius $\epsilon > 0$ centered at $z \in \C$ will be denoted by $\ball_\epsilon(z) = \{w \in \C : |w-z| < \epsilon\}$.
Throughout the paper, $c > 0$ and $C > 0$ will denote generic constants which may change within a computation.

Let $D$ be a simply connected domain with two distinct boundary points $p,q$ (prime ends). There is a conformal transformation $f:D \to \HH$ taking $p$ to $0$ and $q$ to $\infty$; in fact, $f$ is determined only up to a final scaling. We choose one such $f$, but the quantities we define do not depend on the choice. 
Given $z \in D$, we define the conformal radius $r_D(z)$ of $D$ seen from $z$ by letting
\[
\Upsilon_D(z) =  \frac{\Im f(z)}{|f'(z)|}, \quad r_D(z) = 2 \Upsilon_D(z).
\]
Schwarz' lemma and Koebe's 1/4 theorem give the distortion estimates
\begin{align}\label{Koebe}
\dist(z,\partial D)/2 \le \Upsilon_D(z) \le 2 \dist(z,\partial D).
\end{align}
We define
\[
S_{D,p,q}(z) = \sin[\arg f(z) ], \quad S(z) = S_{\HH, 0, \infty}(z),  
\]
and note that this is a conformal invariant.
Suppose $D$ is a Jordan domain and that $J_-,J_+$ are the boundary arcs $f^{-1}(\mathbb{R}_-)$ and $f^{-1}(\mathbb{R}_+)$, respectively. Let $\hm_D(z, E)$ denote the harmonic measure of $E \subset \partial D$ in $D$ from $z$. Then it is easy to see that
\begin{align}\label{Somegaestimate}
S_{D,p,q}(z) \asymp \min\{\hm_D(z, J_-), \, \hm_D(z, J_+) \},
\end{align}
with the implicit constants universal. By conformal invariance an analogous statement holds for any simply connected domain different from $\mathbb{C}$. We will use this relation several times without explicitly saying so in order to estimate $S_{D,p,q}$.

In many places we will estimate harmonic measure using the following lemma often referred to as the \emph{Beurling estimate}. It is derived from Beurling's projection theorem, see for example Theorem~9.2 and Corollary~9.3 of \cite{GM2005}.
\begin{lemma}[Beurling estimate]\label{beurling1}
There is a constant $C < \infty$ such that the following holds.
Suppose $K$ is a connected set in $\overline{\mathbb{D}}$ such that $K \cap \partial \mathbb{D} \neq \emptyset$. Then
\[
\hm_{\mathbb{D} \smallsetminus K}(0, \partial \mathbb{D}) \le C \, (\dist(0, K))^{1/2}.
\]
\end{lemma}


\subsection{Schramm-Loewner evolution}
Let $0 < \kappa < 8$. Throughout the paper we will use the following parameters:
\[
a = \frac{2}{\kappa}, \quad r=r_{\kappa}(\rho)=\frac{\rho}{\kappa} = \frac{\rho a}{2}, \quad d=1+\frac{1}{4a}, \quad \beta = 4a-1.
\]
We will also sometimes write $\alpha = 4a$. The assumption $\kappa = 2/a< 8$ implies that $\alpha > 1$. 

We will work with the $\kappa$-dependent Loewner equation
\begin{align}\label{Loewnereq}
\partial_t g_t(z) = \frac{a}{g_t(z) - \xi_t}, \quad g_0(z) = z,
\end{align}
where $\xi_t$, $t \ge 0$, is the (continuous) Loewner driving term. The solution is a family of conformal maps $(g_t(z))_{t \ge 0}$ called the Loewner chain of $(\xi_t)_{t \ge 0}$. The SLE$_\kappa$ Loewner chain is obtained by taking the driving term to be a standard Brownian motion and $a=2/\kappa$. The chordal SLE$_\kappa$ path is the continuous curve connecting $0$ with $\infty$ in $\HH$ defined by
\[
\gamma(t) = \lim_{y \downarrow 0} g^{-1}_t(\xi_t + iy), \quad \gamma_t: = \gamma[0,t].
\]  
We write $H_t$ for the simply connected domain given by taking the unbounded component of $\HH \smallsetminus \gamma_t$.
 Given a simply connected domain $D$ with distinct boundary points $p,q$, we define chordal SLE$_\kappa$ in $D$ from $p$ to $q$ by conformal invariance. We write
\begin{align}\label{StUpsilontdef}
S_t(z) = S_{H_t, \gamma(t), \infty}(z), \quad \Upsilon_t(z) = \Upsilon_{H_t}(z) = \frac{\Im g_t(z)}{|g'_t(z)|}.
\end{align}
We will make use of the following sharp one-point estimate which also defines the Green's function for chordal SLE$_\kappa$, see Lemma~2.10 of \cite{LW2013}.

\begin{lemma}[Green's function for chordal SLE$_\kappa$]\label{lem:one-point}
Suppose $0 < \kappa < 8$. There exists a constant $c > 0$ such that the following holds. Let $\gamma$ be SLE$_\kappa$ in $D$ from $p$ to $q$, where $D$ is a simply connected domain with distinct boundary points (prime ends) $p,q$. 
As $\epsilon \to 0$ the following estimate holds uniformly with respect to all $z \in D$ satisfying $\dist(z,\partial D) \ge 2\ee$:
\[
\PP \left( \Upsilon_\infty(z) \le \ee \right) = c_{*}\ee^{2-d}G_D(z;p,q)\left[1+O(\ee^c)\right], 
\]
where, by definition, 
\[G_D(z;p,q) = \Upsilon_D(z)^{d-2} S_{D,p,q}(z)^\beta\] 
is the Green's function for SLE$_\kappa$ from $p$ to $q$ in $D$, and $c_*$ is the constant defined in \eqref{cstardef}.
\end{lemma}
\begin{rem}Using the relation between $\Upsilon_\infty(z)$ and Euclidean distance, Lemma~\ref{lem:one-point} shows that $\PP \left( \dist(x, \gamma) \le \ee \right) \leq C \ee^{2-d}G_D(z;p,q)$. In fact, the statement of Lemma~\ref{lem:one-point} holds with $\Upsilon_\infty(z)$ replaced by Euclidean distance, with another constant in place of $c_{*}$.
\end{rem}

We also need to use a boundary estimate for SLE which is convenient to express in terms of extremal distance, see Chapter~IV of \cite{GM2005} for the definition and basic properties we use here.
For a domain $D$ with $E,F \subset \overline{D}$, we write $d_D(E,F)$ for the conformally invariant extremal distance between $E$ and $F$ in $D$. Recall that if $D = \{z: \, r < |z| < R\}$ is the round annulus and $E,F$ are the two boundary components, then $d_D(E,F) = \ln(R/r)/(2\pi)$. A crosscut $\eta$ of $D$ is an open Jordan arc in $D$ with the property that the closure of $\eta$ equals $\eta \cup \{x,y\}$, where $x,y \in \partial D$, and $x,y$ may coincide.
\begin{lemma}\label{lem:boundary-excursion-ED}
Let $0 < \kappa < 8$. 
Suppose $D$ is a simply connected Jordan domain and let $p,q \in \partial D$ be two distinct boundary points. Write $J_{-}, J_{+}$ for the boundary arcs connecting $q$ with $p$ and $p$ with $q$ in the counterclockwise direction, respectively. Suppose $\eta$ is a crosscut of $D$ starting and ending on $J_{+}$, see Figure \ref{excursioncrosscut.pdf}. Then, if $\gamma$ is chordal SLE$_{\kappa}$ in $D$ from $p$ to $q$,
\begin{equation}\label{excursion-boundary}
\PP \left( \gamma \cap \eta \neq \emptyset \right) \le C \, e^{- \beta \pi  d_{D }(J_-, \eta)},
\end{equation}
where $\beta = 4a-1$ and the constant $C \in (0, \infty)$ is independent of $D, p, q$, and $\eta$.
\end{lemma}
\begin{proof}
By conformal invariance, we may assume $D = \mathbb{H}, p=0, q = \infty$ and that $\eta$ separates $1$ from $\infty$. Set $\delta = \max\{|z-1|: z \in \eta\}$ and let $w \in \eta$ be a point such that $|w-1| = \delta$. Then the unbounded annulus $A$ whose boundary components are $\R_-$ and $\eta \cup \bar{\eta}$ separates $w$ and $1$ from $\R_-$. Hence, first using the symmetry rule for extremal distance and then Teichm\"uller's module theorem (see Chapter~II.1.3 and Eq.\ II.2.10 of \cite{LV1970}) and the relation between extremal distance and module\[
d_{\mathbb{H} }(\R_-, \eta) = 2 d_{\mathbb{C}}(\R_-, \eta \cup \bar{\eta})  < \frac{2}{\pi}\ln\left(4/\sqrt{ \delta/(1+\delta) } \right) \le \frac{1}{\pi} \ln(1/(\delta \wedge 1)) + C. 
\]
(Note that \cite{LV1970} defines the module of a round annulus by $\ln(R/r)$, i.e., without the $1/2\pi$ factor.) 
Hence
\begin{align}\label{deltawedge1}
\delta \wedge 1 \leq C e^{-\pi d_{\mathbb{H} }(\R_-, \eta)}.
\end{align}
On the other hand, comparing with a half-disk of radius $\delta$ about $1$, the standard boundary estimate for SLE$_{\kappa}$ in the upper half-plane \cite[Theorem 1.1]{AK2008} implies that $\PP(\gamma \cap \eta \neq \emptyset) \le C \,\delta^{\beta}$, which together with (\ref{deltawedge1}) gives the desired bound.
\end{proof}
\begin{figure}
\bigskip\medskip
\begin{center}
\begin{overpic}[width=.35\textwidth]{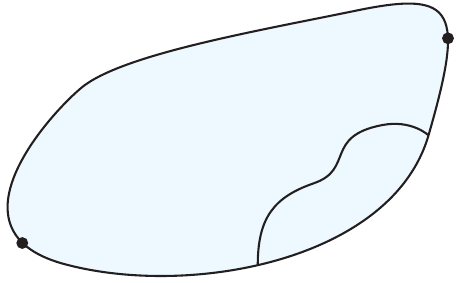}
      \put(43,30){$D$}
      \put(65,24){$\eta$}
      \put(-1,3){$p$}
      \put(101,52){$q$}
      \put(76,5){$J_+$}
      \put(40,55){$J_-$}
\end{overpic}
     \begin{figuretext}\label{excursioncrosscut.pdf}
       The domain $D$ and the crosscut $\eta$ of Lemma \ref{lem:boundary-excursion-ED}. 
       \end{figuretext}
     \end{center}
\end{figure}

\subsubsection{SLE$_{\kappa}(\rho)$}
Let $0 < \kappa < 8$. We will work with SLE$_\kappa(\rho)$, for $\rho \in \mathbb{R}$ chosen appropriately, as defined by weighting SLE$_\kappa$ by a local martingale via Girsanov's theorem; we will explain this below. Let $(\xi^1, \xi^2) \in \R^2$ be given with $ \xi^1 < \xi^2$. Suppose $(\xi^1_t)_{t \ge 0}$ is Brownian motion started from $\xi^1$ under the measure $\PP$, with filtration $\mathcal{F}_t$. We refer to $\PP$ as the SLE$_\kappa$ measure. 
Let $(g_t)_{t \ge 0}$ be the SLE$_\kappa$ Loewner flow defined by equation (\ref{Loewnereq}) with $\xi_t = \xi_t^1$ and set
\begin{align}\label{slekapparhoxi2}
\xi_t^2:=g_t(\xi^2).
\end{align}
We call $\xi^2$ the force point.
Define
\[
\lambda(r) =  \frac{r}{2 a}\left(r - \beta \right), \quad \zeta(r)=\lambda(-r)-r = \frac{r}{2a} \left(r+2a-1\right).
\]
Note that $\zeta \ge 0$ whenever $0< \kappa \le 4$ and $r \ge 0$.  It\^o's formula shows that 
\begin{equation}\label{SLEkr_mg}
M_t^{(\rho)} = \left(\frac{\xi^2_t-\xi^1_t}{\xi^2-\xi^1}\right)^{r} g'_t(\xi^2)^{\zeta(r)} , \quad t \ge 0,  
\end{equation}
is a local $\PP$-martingale for any $\rho \in \R$, where $r=\rho/\kappa$.  
In fact,
\[
 \frac{dM_t^{(\rho)}}{M_t^{(\rho)}}=\frac{-r}{\xi^2_t-\xi^1_t}d\xi^1_t.
\]
The SLE$_\kappa(\rho)$ measure $\PP^{\rho} = \PP_{\xi^1, \xi^2}^{\rho}$ is defined by weighting $\PP$ by the martingale $M^{(\rho)}$,  that is,
\begin{align}\label{PP2def}
\PP^{\rho} \left( V \right) = \E[M_t^{(\rho)} 1_V] \quad \text{for $V \in \mathcal{F}_t$}.
\end{align}
Then, using Girsanov's theorem, the equation for $(\xi^1_t)_{t \ge 0}$ changes to
\begin{align}\label{slekapparhoxi1}
  d\xi^1_t = \frac{r}{\xi^1_t-\xi^2_t}dt + dW_t, 
\end{align}
where $(W_t)_{t \ge 0}$ is $\PP^{\rho}$-Brownian motion. This is the defining equation for the driving term of SLE$_\kappa(\rho)$. (Since $M^{(\rho)}$ is a local martingale we need to stop the process before $M^{(\rho)}$ blows up; we will not always be explicit about this. We will not need to consider SLE$_\kappa(\rho)$ after the time the path hits or swallows the force point.) We refer to the Loewner chain driven by $(\xi^1_t)_{t \ge 0}$ under $\PP^{\rho}$ as {\it SLE$_\kappa(\rho)$ started from $(\xi^1, \xi^2)$}. If $\rho$ is sufficiently negative, the SLE$_{\kappa}(\rho)$ path will almost surely hit the force point. In this case it can be useful to reparametrize so that the quantity
\begin{align}\label{Ctdef}
C_{t} = C_{t}(\xi^{2}) = \frac{\xi_{t}^{2} - O_{t}}{g'_{t}(\xi^{2})}, 
\end{align}
decays deterministically; this is called the radial parametrization in this context. Here  $O_t$ is defined as the image under $g_t$ of the rightmost point in the hull at time $t$; in particular, $O_t = g_{t}(0+)$ if $0 < \kappa \leq 4$, see, e.g., \cite{ABV2016}. Geometrically $C_{t}$ equals (1/4 times) the conformal radius seen from $\xi^{2}$ in $H_{t}$ after Schwarz reflection. We define a time-change $s(t)$ so that $\hat{C}_{t}:=C_{s(t)} = e^{-at}$. A computation shows that if
\[
A_{t} = \frac{\xi^{2}_{t}-O_{t}}{\xi^{2}_{t}-\xi^{1}_{t}}
\] 
then $s'(t) = (\hat{\xi}^2_t-\hat{\xi}^1_t)^2(\hat{A}_t^{-1}-1)$, where $\hat{A}_t = A_{s(t)}$, see, e.g., Section~2.2 of \cite{ABV2016}. An important fact is that $(\hat{A}_t)_{t \ge 0}$ is positive recurrent with respect to SLE$_{\kappa}(\rho)$ if $\rho$ is chosen appropriately.

\begin{lemma}\label{lem:invariant}
Suppose $0 < \kappa < 8$ and  $ \rho < \kappa/2 - 4$. Consider SLE$_{\kappa}(\rho)$ started from $(0,1)$. Then $\hat{A}_{t}$ is positive recurrent with invariant density
\[
\pi_{A}(x) = c'\, x^{-\beta -a\rho}(1-x)^{2a-1}, \quad c' = \frac{\Gamma(2-2a - a\rho)}{\Gamma(2a) \Gamma(2-4a - a\rho)}.
\] 
In fact, there is $\alpha > 0$ such that if $f$ is integrable with respect to the density $\pi_A$, then as $t \to \infty$,
\[
\E \left[f(\hat{A}_t) \right] = c' \int_0^1 f(x) \, \pi_A(x) \, dx \left(1+O(e^{-\alpha t}) \right).
\]
\end{lemma}
\begin{proof}
See, e.g., Section~5.3 of \cite{L-Mink}.
\end{proof}
\subsubsection{Relationship between multiple SLE and SLE$_\kappa(\rho)$}\label{relationshipsubsubsec}
Suppose $\kappa \le 4$ and consider a system of two multiple chordal SLEs curves started from $(\xi^1, \xi^2)$ both aiming at $\infty$; recall \eqref{multiplegtdef} and \eqref{multiplexidef}. Suppose we first grow $\gamma_2$ up to a fixed capacity time $t$. The conditional law of $g_t \circ \gamma_1$ is then an SLE$_\kappa(2)$ in $\HH$ started from $(\xi^1_t, \xi^2_t)$. In particular, the marginal law of $\gamma_1$ is that of an SLE$_\kappa(2)$ started from $\xi^1$ with force point $\xi^2$.
%
Indeed, if we choose the particular growth speeds $\lambda_1 = a$ and $\lambda_2 = 0$, then the defining equations (\ref{multiplegtdef}) and (\ref{multiplexidef}) reduce to
\begin{align}\label{commutingSLEreduced}
\begin{cases}
\partial_t g_t(z) = \frac{a}{g_t(z) - \xi^{1}_t}, & g_0(z) = z,
	\\
 d\xi_t^{1} = \frac{a}{\xi^{1} - \xi^{2}} dt + dB_t^{1}, & \xi_0^1 = \xi^1, 
	\\
 d\xi_t^{2} = \frac{a}{\xi_t^2 - \xi_t^1} dt,	& \xi_0^2 = \xi^2, 
\end{cases}
\end{align}
where $(B_t^1)_{t \ge 0}$ is $\PP$-Brownian motion. Evaluating the equation for $g_t(z)$ at $z = \xi^2$ we infer that $\xi_t^2 = g_t(\xi^2)$. 
Comparing this with the equations (\ref{slekapparhoxi2}) and (\ref{slekapparhoxi1}) defining SLE$_\kappa(\rho)$, we 
conclude that $\gamma_1$ has the same distribution under the multiple SLE$_\kappa$ measure $\PP$ as it has under the SLE$_\kappa(2)$-measure $\PP^2$ started from $(\xi^1, \xi^2)$.

\subsubsection{Two-sided radial SLE and radial parametrization}\label{twosidedradialsubsec}
Recall that if $z \in \HH$ is fixed then the SLE$_{\kappa}$ Green's function in $H_{t}$ equals
\begin{align}\label{Gtdef}
G_t = G_t(z)= \Upsilon_t^{d-2}(z) S_t(z)^\beta,
\end{align}
which is a covariant $\PP$-martingale. \emph{Two-sided radial SLE} in $\HH$ through $z$ is the process obtained by weighting chordal SLE$_\kappa$ by $G$. (This is the same as SLE$_{\kappa}(\kappa-8)$ with force point $z \in \HH$.)
Since two-sided radial SLE approaches its target point, it is natural to parametrize so that the conformal radius (seen from $z$) decays deterministically. More precisely, we change time so that $\Upsilon_{s(t)}(z) = e^{-2at}$; this parametrization depends on $z$. The Loewner equation implies
\[
d\ln \Upsilon_{t} = - 2a\frac{y_{t}^{2}}{|z_{t}|^{4}} dt, \qquad z_{t}=x_{t}+ iy_{t}=g_{t}(z) - \xi^{1}_{t}.
\]
Hence $s'(t) = |\tilde{z}_{t}|^{4}/\tilde{y}_{t}^{2}$, where $\tilde{S}_{t} = S_{s(t)}$, $\tilde{z}_{t} = z_{s(t)}$, etc., denote the time-changed processes.
If $\Theta_t = \arg z_t$, then using that
$$d\Theta_t = (1-2a)\frac{x_ty_t}{|z_t|^4}dt + \frac{y_t}{|z_t|^2}d\xi_t^1,$$
we find that $\tilde{\Theta}_{t} = \Theta_{s(t)}$ satisfies
\[
d \tilde{\Theta}_{t} = (1-2a) \cot\tilde{\Theta}_{t}\, dt + d\tilde{W}_{t},
\]
where $(\tilde{W}_{t})_{t \ge 0}$ is standard $\PP$-Brownian motion. 
The time-changed martingale can be written
\begin{align}\label{tildeGtdef}
\tilde{G}_{t} = e^{-2a(d-2) t}\tilde{S}_{t}^\beta.
\end{align}
The measure $\PP^{*} = \PP_z^{*}$ is defined by weighting chordal SLE$_{\kappa}$ by $\tilde{G}$, that is,
\begin{align}\label{twosidedradialdef}
\PP^* \left( V \right) = \tilde{G}_0^{-1}\E[\tilde{G}_t 1_V], \qquad V \in \tilde{\mathcal{F}}_t.
\end{align}
This produces two-sided radial SLE$_{\kappa}$ in the radial parametrization.

Since $d\tilde{G}_t = \beta \tilde{G}_t \cot(\tilde{\Theta}_t) d\tilde{W}_t$, Girsanov's theorem implies that the equation for $\tilde{\Theta}_{t}$ changes to the \emph{radial Bessel equation} under the new measure $\PP^{*}$:
\[
d \tilde{\Theta}_{t} = 2a \cot\tilde{\Theta}_{t} \, dt + d\tilde{B}_{t},
\]
where $\tilde{B}_{t}$ is $\PP^{*}$-standard Brownian motion. 

We will use the following lemma about the radial Bessel equation, see, e.g., Section~3 of \cite{Lawler_Bessel_notes}.

\begin{lemma}\label{lem:radial-bessel}
Let $0 < \kappa < 8, \, a = 2/\kappa$ and  suppose the process $(\Theta_t)_{t \ge 0}$ is a solution to the SDE 
\begin{equation}\label{twosided-SDE}
d \Theta_t = 2a \cot \Theta_t \, dt+ dB_t, \quad \Theta_0 =\Theta.
\end{equation}
Then $(\Theta_t)_{t \ge 0}$ is positive recurrent with invariant density
\begin{equation*}
\psi(x) = \frac{c_*}{2} \sin^{4a} x,
\end{equation*}
where $c_*$ is the constant in (\ref{cstardef}). 
In fact, there is $\alpha > 0$ such that if $f$ is integrable with respect to the density $\psi$, then as $t \to \infty$,
\[
\E\left[ f(\Theta_t) \right] = \int_0^\pi f(x) \, \psi(x) \, dx \, (1+O(e^{-\alpha t})),
\]
where the error term does not depend on $\Theta_0$.
\end{lemma}

\section{Martingale observables as CFT correlation functions}\label{martingalesec}
\subsection{Screening}
The CFT framework of Kang and Makarov \cite{KM2013} can be used to generate martingale observables for SLE systems, see in particular Lecture~14 of \cite{KM2013}. The ideas of \cite{KM2013} have been extended to incorporate several multiple SLEs started from different points in \cite{AKM}. (See also, e.g.,  \cite{BBK2005, GC2006} for related work.) We will also make use of the screening method \cite{DF1984} which produces observables in the form of contour integrals, which we call Dotsenko-Fateev integrals. 
From the CFT perspective (in the sense of \cite{KM2013}), one starts from a CFT correlation function with appropriate field insertions giving a corresponding (known) SLE$_{\kappa}$ martingale. 
Adding additional paths means inserting additional boundary fields. This will create observables for the system of SLEs. But in the cases we consider, the extra fields change the boundary behavior so that the new observable does not encode the desired geometric information anymore. To remedy this, carefully chosen auxiliary fields are inserted and then integrated out along integration contours. (The mismatching ``charges'' are ``screened'' by the contours.) The correct choices of insertions and integration contours depend on the particular problem, and different choices correspond to solutions with different boundary behavior. 

\begin{rem}\upshape We mention in passing that from a different point of view, it is known that the Gaussian free field with suitable boundary data can be coupled with SLE paths as local sets for the field \cite{MS2012}. Jumps in boundary conditions for the GFF are implemented by vertex operator insertions on the boundary. By the nature of the coupling, correlation functions for the field will give rise to SLE martingales. 
\end{rem}

In what follows, we briefly summarize how we used these ideas to arrive at the explicit expressions (\ref{Pdef}) and (\ref{GGdef}) for the Schramm probability $P(z, \xi)$ and the Green's function $\GG(z, \xi^1, \xi^2)$, respectively. We refer to \cite{KM2013, AKM} for an introduction to the underlying CFT framework and we will use notation from these references. Since the discussion is purely motivational, we make no attempt in this section to be complete or rigorous. This is in contrast to the other sections of the paper which are rigorous. Indeed, we shall only use the results of this section as guesses for solutions to be studied more closely later on. 

Consider a system of two multiple SLEs started from $(\xi^1, \xi^2) \in \R^2$. If $\lambda_1$ and $\lambda_2$ denote the growth speeds of the two curves, the evolution of the system is described by equations (\ref{multiplegtdef}) and (\ref{multiplexidef}).
In the language of \cite{KM2013}, the presence of two multiple SLE curves in $\mathbb{H}$ started from $\xi^1$ and $\xi^2$ corresponds to the insertion of the operator
$$\mathcal{O}(\xi^1, \xi^2) = V_{\star, (b)}^{i\sqrt{a}}(\xi^1)V_{\star, (b)}^{i\sqrt{a}}(\xi^2),$$
where $V_{\star,(b)}^{i\sigma}(z)$ denotes a rooted vertex field inserted at $z$ (see \cite{KM2013}, p. 96) and the parameter $b$ satisfies the relation 
\begin{align}\label{abrelation}
2\sqrt{a}(\sqrt{a} + b) =1, \qquad a=2/\kappa.
\end{align}
Notice that we define $a=2/\kappa$ while \cite{KM2013} defines ``$a$'' by $\sqrt{2/\kappa}$. The framework of \cite{KM2013} (or rather an extension of this framework to the case of multiple curves \cite{AKM}) suggests that if $\{z_j\}_1^n \subset \C$ are points and $\{X_j\}_1^n$ are fields satisfying certain properties, then the correlation function  
\begin{align}\label{Mtz1zn}
M_t^{(z_1, \dots, z_n)} = \hat{\E}_{\mathcal{O}(\xi_t^1, \xi_t^2)}^{\mathbb{H}}[(X_1||g_t^{-1})(z_1) \cdots (X_n|| g_t^{-1})(z_n)]
\end{align}
is a (local) martingale observable for the system when evaluated in the ``Loewner charts'' $(g_t)$, where for each $t$ the map $g_t$ ``removes'' the whole system of curves which is growing at constant capacity speed.
It turns out that the observables relevant for Schramm's formula and for the Green's function belong to a class of correlation functions of the form
\begin{align}\label{Mtzudef}
M_t^{(z,u)} = \hat{\E}_{\mathcal{O}(\boldsymbol{\xi}_t)}[(V_{\star, (b)}^{i\sigma_1}||g_t^{-1})(z)\overline{(V_{\star, (b)}^{i\sigma_2}||g_t^{-1})(z)}(V_{\star, (b)}^{is}||g_t^{-1})(u)],
\end{align}
where $z \in \mathbb{H}$, $u \in \C$, and $\sigma_1, \sigma_2, s \in \R$ are real constants. 
We will later integrate out the variable $u$, but it is essential to include the screening field $(V_{\star, (b)}^{is}||g_t^{-1})(u)$ in the definition (\ref{Mtzudef}) in order to arrive at observables with the appropriate conformal dimensions at $z$ and at infinity. 
The observable $M_t^{(z,u)}$ can be written as
\begin{align}\label{Mmartingale}
M_t^{(z,u)} = &\; g_t'(z)^{\frac{\sigma_1^2}{2} -\sigma_1 b} \overline{g_t'(z)}^{\frac{\sigma_2^2}{2}-\sigma_2 b}  g_t'(u)^{\frac{s^2}{2}-s b}
A(Z_t, \xi_t^1, \xi_t^2, U_t),
\end{align}
where $Z_t = g_t(z)$, $U_t = g_t(u)$, and the function $A(z,\xi^1, \xi^2, u)$ is defined by
\begin{align}\nonumber
A(z,\xi^1, \xi^2, u) 
= &\;  (z - \bar{z})^{\sigma_1\sigma_2} \big[(z - \xi^1)(z - \xi^2)\big]^{\sigma_1 \sqrt{a}} 
 \big[(\bar{z} - \xi^1)(\bar{z} - \xi^2)\big]^{\sigma_2 \sqrt{a}} 	
	\\ \label{Adef}
& \times  (z - u)^{\sigma_1s} (\bar{z} - u)^{\sigma_2 s}
\big[(u - \xi^1)(u - \xi^2)\big]^{s \sqrt{a}}. 
\end{align}
It\^{o}'s formula implies that the CFT generated observable $M_t^{(z,u)}$ is indeed a local martingale for any choice of $z,u \in \mathbb{H}$ and $\sigma_1,\sigma_2,s \in \R$. 
Since (\ref{Mmartingale}) is a local martingale for each value of the screening variable $u$, and the observable transforms as a one-form in $u$, we expect the integrated observable 
\begin{align}\label{integratedMdef}
\mathcal{M}_t^{(z)} = \int_\gamma M_t^{(z,u)} du
\end{align}
to be a local martingale for any choice of $z \in \mathbb{H}$, $\sigma_1,\sigma_2,s \in \R$, and of the integration contour $\gamma$, at least as long as the integral in (\ref{integratedMdef}) converges and the branches of the complex powers in (\ref{Adef}) are consistently defined. The integral in (\ref{integratedMdef}) is referred to as a ``screening'' integral.

By choosing $\lambda_2 = 0$, we expect the observable $\mathcal{M}_t^{(z)}$ defined in (\ref{integratedMdef}) to be a local martingale for SLE$_\kappa(2)$ started from $(\xi^1, \xi^2)$. We later check these facts in the cases of interest by direct computation, see Propositions \ref{schrammPDEprop} and \ref{greenPDEprop}. We next describe how the martingales relevant for Schramm's formula and for the Green's function for SLE$_\kappa(2)$ arise as special cases of $\mathcal{M}_t^{(z)}$ corresponding to particular choices of $\sigma_1,\sigma_2,s \in \R$ and of the contour $\gamma$. 

\subsection{Prediction of Schramm's formula}
In order to obtain the local martingale relevant for Schramm's formula we choose the following values for the parameters (``charges'') in (\ref{Mmartingale}):
\begin{align}\label{schrammchoice}
\sigma_1 = -2\sqrt{a}, \qquad \sigma_2 = 2b, \qquad s = -2\sqrt{a}.
\end{align}
The choice (\ref{schrammchoice}) can be motivated as follows. First of all, by choosing $s = -2\sqrt{a}$ we ensure that $s^2/2 -s b = 1$ (see (\ref{abrelation})). This implies that $\mathcal{M}_t^{(z)}$ involves the one-form $g_t'(u)^{s^2/2-s b}du = g_t'(u) du$. After integration with respect to $du$ this leads to a conformally invariant screening integral. 
To motivate the choices of $\sigma_1$ and $\sigma_2$, let $P(z, \xi^1, \xi^2)$ denote the probability that the point $z\in \mathbb{H}$ lies to the left of an SLE$_\kappa(2)$-path started from $(\xi^1,\xi^2)$. Then we expect $\partial_zP$ to be a martingale observable with conformal dimensions \begin{align}\label{schrammdimensions}
\lambda(z) = 1, \qquad \lambda_*(z) = 0, \qquad \lambda_\infty = 0.
\end{align}
The parameters in (\ref{schrammchoice}) are chosen so that the observable $\mathcal{M}_t^{(z)}$ in (\ref{integratedMdef}) has the conformal dimensions in (\ref{schrammdimensions}). We emphasize that it is the inclusion of the screening field in (\ref{Mtzudef}) that makes it possible to obtain these dimensions. In particular, by including it we can have $\lambda_{\infty} = 0$. We have considered the derivative $\partial_zP$ instead of $P$ because then we are able to  construct a nontrivial martingale with the correct dimensions. 

In the special case when the parameters $\sigma_1,\sigma_2,s$ are given by (\ref{schrammchoice}), the local martingale (\ref{integratedMdef}) takes the form
\begin{align}\nonumber
\mathcal{M}_t^{(z)} = &\; g_t'(z) 
(Z_t - \bar{Z}_t)^{\alpha - 2}(Z_t - \xi_t^1)^{-\frac{\alpha}{2}}(Z_t - \xi_t^2)^{-\frac{\alpha}{2}} (\bar{Z}_t - \xi_t^1)^{1-\frac{\alpha}{2}}(\bar{Z}_t - \xi_t^2)^{1-\frac{\alpha}{2}}
	\\ \label{MSchrammmartingale}
& \times \int_\gamma (Z_t - U_t)^\alpha (\bar{Z}_t - U_t)^{\alpha -2} \big[(U_t - \xi_t^1)(U_t - \xi_t^2)\big]^{-\frac{\alpha}{2}} g_t'(u) du.
\end{align}
We expect from the above discussion that there exists an appropriate choice of the integration contour $\gamma$ in (\ref{integratedMdef}) such that $\partial_z P(z, \xi^1, \xi^2) = \text{const} \times \mathcal{M}_0^{(z)}$, that is, we expect
\begin{align*}
\partial_zP(z, \xi^1, \xi^2) 
= & \; c(\kappa) y^{\alpha - 2} (z - \xi^1)^{-\frac{\alpha}{2}}(z- \xi^2)^{-\frac{\alpha}{2}}(\bar{z} - \xi^1)^{1-\frac{\alpha}{2}}(\bar{z} - \xi^2)^{1-\frac{\alpha}{2}}
	\\
& \times \int_\gamma (u-z)^\alpha (u-\bar{z})^{\alpha -2} (u - \xi^1)^{-\frac{\alpha}{2}}(u - \xi^2)^{-\frac{\alpha}{2}} du,
\end{align*} 
where $c(\kappa)$ is a complex constant and $y = \im z$. 
Setting $\xi^1= 0$ and $\xi^2 = \xi$ in this formula, we arrive at the prediction (\ref{Pdef}) for the Schramm probability $P(z,\xi)$. Indeed, the integration with respect to $x$ in (\ref{Pdef}) recovers $P$ from $\partial_z P$ and ensures that $P$ tends to zero as $\re z \to \infty$. On the other hand, the choice of the integration contour from $\bar{z}$ to $z$ in (\ref{Mdef}) is mandated by the requirement that $P(z, \xi)$ should satisfy the correct boundary conditions as $z$ approaches the real axis. See \cite[Lecture 15]{KM2013} (see also, e.g., \cite{KBZ2012} and the references therein).
Moreover, $P(z,\xi)$ must be a real-valued function tending to $1$ as $\re z \to -\infty$; this fixes the constant $c(\kappa)$. 
\subsection{Prediction of the Green's function}
In order to obtain the local martingale relevant for the SLE$_\kappa(2)$ Green's function, we choose the following values for the parameters in (\ref{Mmartingale}):
\begin{align}\label{greenchoice}
\sigma_1 = b - \sqrt{a}, \qquad \sigma_2 = b - \sqrt{a}, \qquad s = -2\sqrt{a},
\end{align}
As in the case of Schramm's formula, the choice $s = -2\sqrt{a}$ ensures that
$\mathcal{M}_t^{(z)}$ involves the one-form $g_t'(u) du$.
Moreover, if we let $\GG(z, \xi^1, \xi^2)$ denote the Green's function for SLE$_\kappa(2)$ started from $(\xi^1,\xi^2)$, then we expect $\GG$ to have the conformal dimensions (cf. page 124 in \cite{KM2013})
\begin{align}\label{greendimensions}
\lambda(z)= \lambda_*(z) = \frac{2-d}{2}, \qquad \lambda_\infty = 0.
\end{align}
The parameters $\sigma_1$ and $\sigma_2$ in (\ref{greenchoice}) are determined so that the observable $\mathcal{M}_t^{(z)}$ in (\ref{integratedMdef}) has the conformal dimensions in (\ref{greendimensions}). For example, a generalization of Proposition~15.5 in \cite{KM2013} to the case of two curves implies that $\lambda_\infty = (2\sqrt{a}-b)\Sigma + \frac{\Sigma^2}{2} = 0$ where $\Sigma = \sigma_1 + \sigma_2 - 2\sqrt{a}$.

\begin{rem}\label{remarkongenerality}\upshape
We can see here that the choice $\rho=2$ is special: we have only two possible ways to add one screening field, corresponding to $s=-2\sqrt{a}$ or $s=1/\sqrt{a}$. But the extra $\rho = 2$ corresponds to additional charges $\sigma = \sigma_* = 2/\sqrt{8 \kappa}$ (we are using $\sigma = \rho/\sqrt{8 \kappa}$), so at infinity we have an additional charge $\sigma+ \sigma_* = 2 \sqrt{a}$. Consequently, the $\rho = 2$ charge can be screened by only one screening field. If we add more $\rho$ insertions, they can be screened by one screening field if their charges sum up to $2\sqrt{a}$. This suggests that every SLE$_\kappa$ observable with $\lambda_q=0$ gives an SLE$_\kappa(2)$ observable with $\lambda_q=0$ after screening. Similarly, since adding $n$ additional $\rho_{j} =2$ gives additional charges at $\infty$ of $2n\sqrt{a}$, one could expect that one can construct a martingale for a system of $n$ SLEs by adding $n$ screening charges.
\end{rem}
 
In the special case when the parameters $\sigma_1,\sigma_2,s$ are given by (\ref{greenchoice}), the local martingale (\ref{integratedMdef}) takes the form
\begin{align}\label{greenintegratedM}
\mathcal{M}_t^{(z)} = &\; |g_t'(z)|^{2-d} \int_\gamma \mathcal{A}(Z_t,\xi_t^1, \xi_t^2, g_t(u)) g_t'(u)du,
\end{align}
where
\begin{align}\nonumber
\mathcal{A}(z,\xi^1, \xi^2, u) 
= &\; (z - \bar{z})^{\alpha + \frac{1}{\alpha} - 2} |z - \xi^1|^{-\beta} |z - \xi^2|^{-\beta} 
	\\ 
& \times (z - u)^\beta (\bar{z} - u)^{\beta} \big[(u - \xi^1)(u - \xi^2)\big]^{-\frac{\alpha}{2}}.
 \end{align}
We expect from the above discussion that there exists an appropriate choice of the integration contour $\gamma$ in (\ref{integratedMdef}) such that $\GG(z, \xi^1, \xi^2) = \text{const} \times \mathcal{M}_0^{(z)}$, that is, we expect
\begin{align}\label{GGJ}
\GG(z, \xi^1, \xi^2) 
= & \; c(\kappa) y^{\alpha + \frac{1}{\alpha} - 2} 
|z - \xi^1|^{-\beta} |z - \xi^2|^{-\beta} J(z,\xi^1, \xi^2),
\end{align} 
where 
$$J(z,\xi^1, \xi^2) = \int_\gamma (u-z)^\beta (u-\bar{z})^{\beta} (u - \xi^1)^{-\frac{\alpha}{2}} (\xi^2-u)^{-\frac{\alpha}{2}} du$$
and $c(\kappa)$ is a complex constant. 
By requiring that $G$ satisfy the correct boundary conditions, we arrive at the prediction (\ref{GGdef}) for the Green's function for SLE$_\kappa(2)$. 
The trickiest step is the determination of the appropriate screening contour $\gamma$. This contour must be chosen so that the Green's function satisfies the appropriate boundary conditions as $(z,\xi^1, \xi^2)$ approaches the boundary of the domain $\mathbb{H} \times \{-\infty  < \xi^1 < \xi^2 < \infty\}$. The complete verification that the Pochhammer integration contour in (\ref{Mdef}) leads to the correct boundary behavior is presented in Lemma \ref{hboundarylemma} and relies on a complicated analysis of integral asymptotics.
We first arrived at the Pochhammer contour in (\ref{Mdef}) via the following simpler argument.

Let $\GG_\xi(z) = \GG(z, -\xi,\xi)$, $J_\xi(z) = J(z, -\xi,\xi)$. Let also
$I_\xi(z) = I(z, -\xi,\xi)$ where $I$ is the function defined in (\ref{Idef}), i.e.,
\begin{align}\label{Idef2}
I_\xi(z) = \int_A^{(z+,\xi+,z-,\xi-)} (u - z)^{\alpha -1} (u - \bar{z})^{\alpha -1} (\xi + u)^{-\frac{\alpha}{2}} (\xi - u)^{-\frac{\alpha}{2}} du.
\end{align}

We make the ansatz that 
\begin{align}\label{Jxizansatz}
J_\xi(z) = \sum_{i=1}^4 c_i(\kappa) \int_{\gamma_i} (u - z)^{\alpha -1} (u - \bar{z})^{\alpha -1} (\xi + u)^{-\frac{\alpha}{2}} (\xi - u)^{-\frac{\alpha}{2}}  du,
\end{align}
where the contours $\{\gamma_i\}_1^4$ are Pochhammer contours surrounding the pairs $(\xi, z)$, $(\xi, \bar{z})$, $(-\xi, z)$, and $(-\xi, \bar{z})$, respectively.  
The integral involving the pair $(\xi, z)$ is $I_\xi(z)$. The integrals involving the pairs $(\pm \xi, z)$ are related via complex conjugation to the integrals involving the pairs $(\pm \xi, \bar{z})$. Moreover, by performing the change of variables $u \to -\bar{u}$, we see that the integral involving the pair $(-\xi, z)$ can be expressed in terms of $I(-\bar{z})$. 
Thus, using the requirement that $J(z,\xi)$ be real-valued, we can without loss of generality assume that $J(z,\xi)$ is a real linear combination of the real and imaginary parts of $I_\xi(z)$ and $I_\xi(-\bar{z})$.

At this stage it is convenient, for simplicity, to assume $4 < \kappa < 8$ so that $1 < \alpha < 2$. Then we can collapse the contour in the definition (\ref{Idef2}) of $I_\xi(z)$ onto a curve from $\xi$ from $z$; this gives
$$I_\xi(z) = (1-e^{2 i \pi \alpha}+e^{i \pi \alpha}-e^{-i \pi \alpha}) \hat{I}_\xi(z),$$
where $\hat{I}_\xi(z)$ is defined by
$$\hat{I}_\xi(z) = \int_\xi^z (u - z)^{\alpha -1} (u - \bar{z})^{\alpha -1} (\xi + u)^{-\frac{\alpha}{2}} (\xi - u)^{-\frac{\alpha}{2}} du.$$
Since $\hat{I}$ obeys the symmetry $\im \hat{I}_\xi(z) = \im \hat{I}_\xi(-\bar{z})$, our ansatz takes the form
\begin{align}\label{Jxiansatz}
J_\xi(z) = A_1 \re \hat{I}_\xi(z) + A_2\re \hat{I}_\xi(-\bar{z})  + A_3 \im \hat{I}_\xi(z),
\end{align}
where $A_j = A_j(\kappa)$, $j = 1,2,3$, are real constants. 

\begin{rem}\upshape
It is not necessary to include further contours surrounding pairs such as $(z, \bar{z})$ and $(-\xi, \xi)$ in the ansatz (\ref{Jxizansatz}) for $J_\xi(z)$, because the contributions from such pairs can be obtained as linear combinations of the contributions from the four pairs already included. This is most easily seen in the case $1 < \alpha < 2$ where each Pochhammer contour can be collapsed to a single curve connecting the two points in the pair.
\end{rem}

Up to factors which are independent of $y$, we expect the Green's function $\GG_\xi(z)$ to satisfy 
\begin{subequations}\label{GGboundaryvalues}
\begin{align}
& \GG_\xi(x+iy) \sim y^{d-2} = y^{\frac{1}{\alpha} - 1}, \qquad y \to \infty, \quad x \; \text{fixed},
	\\
& \GG_\xi(\xi + iy) \sim y^{d-2}y^{\beta + 2a} = y^{\frac{1}{\alpha} + \frac{3\alpha}{2} -2}, \qquad y \downarrow 0.
\end{align}
\end{subequations}
Indeed, since the influence of the force point $\xi^2$ goes to zero as $\im \gamma(t)$ becomes large, the first relation follows by comparison with SLE$_\kappa$. The second relation can be motivated by noticing that the boundary exponent for SLE$_\kappa(\rho)$ at the force point $\xi^2$ is $\beta + \rho a$, see Lemma \ref{lem:dec10.4}. 
In terms of $J_\xi(z)$, the estimates (\ref{GGboundaryvalues}) translate into
\begin{subequations}
\begin{align} \label{Jpropertiesa}
& J_\xi(x + iy) \sim y^{\alpha - 1}, \qquad y \to \infty, \ x \; \text{fixed},
	\\ \label{Jpropertiesb}
& J_\xi(\xi + iy) \sim y^{\frac{3\alpha}{2}-1}, \qquad y \downarrow 0.
\end{align}
\end{subequations}
We will use these conditions to fix the values of the $A_j$'s. 

We obtain one constraint on the $A_j$'s by considering the asymptotics of $J_\xi(iy)$ as $y \to \infty$. Indeed, for $x = 0$ we have
\begin{align*}
\hat{I}_\xi(iy) = & \int_\xi^{iy} (u^2 + y^2)^{\alpha -1} (\xi^2 - u^2)^{-\frac{\alpha}{2}} du
	\\
 = &\; \frac{i}{2}  \sqrt{\pi } \xi ^{-\alpha } y^{2 \alpha -2} \bigg\{\frac{y \Gamma (\alpha )}{\Gamma(\alpha +\frac{1}{2})} \,
   _2F_1\left(\frac{1}{2},\frac{\alpha }{2},\alpha +\frac{1}{2},-\frac{y^2}{\xi^2}\right)
   	\\
& +\frac{i \xi  \Gamma(1-\frac{\alpha }{2})}{\Gamma(\frac{3}{2}-\frac{\alpha }{2})} \, _2F_1\left(\frac{1}{2},1-\alpha,\frac{3}{2}-\frac{\alpha }{2},-\frac{\xi^2}{y^2}\right)\bigg\},
\end{align*}
where ${}_2F_1$ denotes the standard hypergeometric function. This implies
\begin{align*}
\hat{I}_\xi(iy) = &\;  
y^{\alpha -1 } \left(\frac{i \Gamma(\frac{1}{2}-\frac{\alpha }{2}) \Gamma
   (\alpha )}{2 \Gamma(\frac{\alpha +1}{2})}+O\left(\frac{1}{y^2}\right)\right)
   	\\
& + y^{2(\alpha - 1)} \left(-\frac{\pi ^{3/2} \xi ^{1-\alpha } (\csc(\frac{\pi  \alpha}{2})+i \sec(\frac{\pi  \alpha }{2}))}{2 (\Gamma(\frac{3}{2}-\frac{\alpha }{2}) \Gamma(\frac{\alpha}{2}))}+O\left(\frac{1}{y^2}\right)\right),
   \qquad y \to \infty.
\end{align*}
Substituting this expansion into (\ref{Jxiansatz}), we find an expression for $J_\xi(iy)$ involving two terms which are proportional to $y^{2(\alpha - 1)}$ and $y^{\alpha -1}$, respectively, as $y \to \infty$. In order to satisfy the condition (\ref{Jpropertiesa}), we must choose the $A_j$ so that the coefficient of the larger term involving $y^{2(\alpha - 1)}$ vanishes. This leads to the relation
\begin{align}\label{Aconstraint1}
\frac{A_1 + A_2}{A_3} = - \tan\frac{\pi\alpha}{2}.
\end{align}

We obtain a second constraint on the $A_j$'s by considering the asymptotics of $J_\xi(iy)$ as $z \to \xi$. Indeed, for $x = \xi$ we have
\begin{align*}
\hat{I}_\xi(\xi + iy) = e^{\frac{i\pi}{2}(1 + \frac{\alpha}{2})} \int_0^y (y^2 - s^2)^{\alpha -1} (2\xi + is)^{\alpha -1} s^{-\frac{\alpha}{2}} ds.
\end{align*}	
Hence
\begin{align}\nonumber
\hat{I}_\xi(\xi + iy)
& \sim 
e^{\frac{i\pi}{2}(1 + \frac{\alpha}{2})} (2\xi)^{-\frac{\alpha}{2}} \int_0^y (y^2 - s^2)^{\alpha -1} s^{-\frac{\alpha}{2}} ds
	\\\label{hatInearxi}
& = \frac{2^{-\frac{\alpha }{2}-1} e^{\frac{1}{4} i \pi  (\alpha +2)} \xi ^{-\alpha /2}
   \Gamma \left(\frac{1}{2}-\frac{\alpha }{4}\right) \Gamma (\alpha )}{\Gamma \left(\frac{3 \alpha }{4}+\frac{1}{2}\right)} y^{\frac{3 \alpha
   }{2}-1}, \qquad y \downarrow 0, \ \xi > 0.
\end{align}
Similarly, for $x = -\xi$, we have
\begin{align*}
\hat{I}(-\xi + iy, \xi)
= &\; \int_{\xi}^{-\xi}((u+\xi)^2 + y^2)^{\alpha -1} (\xi + u)^{-\frac{\alpha}{2}} (\xi - u)^{-\frac{\alpha}{2}} du
	\\
& + \int_0^y(i(s-y))^{\alpha -1} (i(s+y))^{\alpha -1} (is)^{-\frac{\alpha}{2}} (2\xi - is)^{-\frac{\alpha}{2}} i ds.
\end{align*}	
Hence
\begin{align}\nonumber
\hat{I}(-\xi + iy, \xi)
 \sim 
& -\int_{-\xi}^{\xi}(u+\xi)^{\frac{3\alpha}{2} -2} (\xi - u)^{-\frac{\alpha}{2}} du
+ e^{\frac{i\pi}{2}(1 - \frac{\alpha}{2})} (2\xi)^{-\frac{\alpha}{2}} 
\int_0^y (y^2 - s^2)^{\alpha -1} s^{-\frac{\alpha}{2}} ds
	\\\nonumber
=&\; 2 \xi ^{\alpha -1} \left(\frac{\, _2F_1\left(1,2-\frac{3 \alpha }{2},2-\frac{\alpha
   }{2};-1\right)}{\alpha -2}+\frac{\, _2F_1\left(1,\frac{\alpha }{2},\frac{3 \alpha
   }{2};-1\right)}{2-3 \alpha }\right)
   	\\ \label{hatInearminusxi}
&   +\frac{i 2^{-\frac{\alpha }{2}-1} e^{-\frac{1}{4} i
   \pi  \alpha } \xi ^{-\alpha /2} \Gamma \left(\frac{1}{2}-\frac{\alpha }{4}\right)
   \Gamma (\alpha )}{\Gamma \left(\frac{3 \alpha
   }{4}+\frac{1}{2}\right)} y^{\frac{3 \alpha }{2}-1}, \qquad y \downarrow 0, \ \xi > 0.
\end{align}
Substituting the expansions (\ref{hatInearxi}) and (\ref{hatInearminusxi}) into (\ref{Jxiansatz}), we find an expression for $J_\xi(\xi + iy)$ involving two terms which are of order  $O(y^{\frac{3 \alpha }{2}-1})$ and $O(1)$, respectively, as $y \to 0$. 
In order to satisfy the condition (\ref{Jpropertiesb}), we must choose the $A_j$ so that the coefficient of the larger term of $O(1)$ vanishes. This implies
\begin{align}\label{Aconstraint2}
A_2 = 0.
\end{align}
Using the constraints (\ref{Aconstraint1}) and (\ref{Aconstraint2}), the expression (\ref{Jxiansatz}) becomes
$$J_\xi(z) = B_1 \im\big(e^{-\frac{i\pi \alpha}{2}}\hat{I}_\xi(z)\big)
 = B_2 \im\big(e^{-i\pi \alpha}I_\xi(z)\big),$$
where $B_j = B_j(\kappa)$, $j = 1,2$, are real constants. Recalling (\ref{GGJ}), this gives the following expression for $\GG_\xi(z) = \GG(z, -\xi,\xi)$:
\begin{align*}
\GG(z, -\xi,\xi) 
= \frac{1}{\hat{c}}y^{\alpha + \frac{1}{\alpha} - 2} |z + \xi|^{1 - \alpha} |z - \xi|^{1 - \alpha}  \Im\big(e^{-i\pi\alpha} I(z,-\xi,\xi) \big), \qquad
z \in \mathbb{H}, \ \xi > 0,
\end{align*}
where $\hat{c}(\kappa)$ is an overall real constant yet to be determined.
Using translation invariance to extend this expression to an arbitrary starting point $(\xi^1, \xi^2)$, we find (\ref{GGdef}).
The derivation here used that $4 < \kappa < 8$, but by analytic continuation we expect the same formula to hold for $0 < \kappa \leq 4$.

\begin{rem}\upshape
We remark here that the non-screened martingale obtained via Girsanov has the conformal dimensions
\begin{align}\label{girsanovdimensions}
\lambda(z)= \lambda_*(z) = \frac{2-d}{2}, \qquad \lambda_\infty = -\beta.
\end{align}
\end{rem}
\section{Schramm's formula}\label{schrammsec}
This section proves Theorem~\ref{thm:schramm}. The strategy is the same as in Schramm's original argument \cite{S2001}. Assume $0 < \kappa \le 4$, i.e., $\alpha = 8/\kappa \ge 2$. We write the function $\mathcal{M}(z, \xi)$ defined in (\ref{Mdef}) as
\begin{align}\label{Mdef2}
\mathcal{M}(z, \xi) = &\; y^{\alpha - 2} z^{-\frac{\alpha}{2}}(z- \xi)^{-\frac{\alpha}{2}}\bar{z}^{1-\frac{\alpha}{2}}(\bar{z} - \xi)^{1-\frac{\alpha}{2}}J(z, \xi), \qquad z \in \mathbb{H}, \ \xi > 0,
\end{align}
where $z = x+iy$ and $J(z, \xi)$ is defined by
\begin{align}\label{Jdef}
J(z, \xi) = \int_{\bar{z}}^z(u-z)^{\alpha}(u- \bar{z})^{\alpha - 2} u^{-\frac{\alpha}{2}}(u - \xi)^{-\frac{\alpha}{2}} du, \qquad z \in \mathbb{H}, \ \xi > 0,
\end{align}
and the contour from $\bar{z}$ to $z$ passes to the right of $\xi$ as in Figure \ref{Jcontour.pdf}.
We want to prove that the probability that the system started from $(0,\xi)$ passes to the right of $z = x+iy$ is given by
$$P(z, \xi) = \frac{1}{c_\alpha} \int_x^\infty \re \mathcal{M}(x' +i y, \xi) dx', \qquad x \in \R, \ y > 0, \ \xi > 0.$$
The idea is to apply It\^o's formula and a stopping time argument to prove that the prediction is correct. 
Once we have proved Theorem~\ref{thm:schramm}, we easily obtain fusion formulas by simply collapsing the seeds.

\subsection{Proof of Theorem~\ref{thm:schramm}}
In \cite{LV2018B}, we carefully analyze the function $P(z,\xi)$ and show that it is well-defined, smooth, and fulfills the correct boundary conditions. We summarize these facts here and then use them to give the short proof of Theorem~\ref{thm:schramm}.

\begin{lemma} \label{Plemma}
The function $P(z, \xi)$ defined in \eqref{Pdef} is a well-defined smooth function of $(z, \xi) \in \mathbb{H} \times (0, \infty)$ which satisfies
\begin{subequations}
\begin{align}\label{Peverywherea1}
&  |P(z, \xi)| \leq C(\arg z)^{\alpha -1}, \qquad z \in \mathbb{H}, \ \xi > 0,
  	\\ \label{Peverywhereb1}
 & |P(z, \xi) - 1| \leq C(\pi - \arg z)^{\alpha -1}, \qquad z \in \mathbb{H}, \ \xi > 0.
\end{align}
\end{subequations}
\end{lemma}
\begin{proof}
See \cite{LV2018B}.
\end{proof}

\begin{prop}[PDE for Schramm's formula]\label{schrammPDEprop}
Let $\alpha >1$. The function $\tilde{\mathcal{M}}$ defined by
$$\tilde{\mathcal{M}}(x,y,\xi^1, \xi^2) = \mathcal{M}(x - \xi^1 + iy, \xi^2 - \xi^1),$$ 
where $\mathcal{M}$ is given by (\ref{Mdef}), satisfies the two linear PDEs
\begin{align}\label{schrammPDEs}
\left(\mathcal{A}_j - \frac{2}{(x+iy-\xi^j)^2} \right)\tilde{\mathcal{M}} = 0, \qquad j = 1,2,
\end{align}
where the differential operators $\mathcal{A}_j$ are defined by
\begin{align}
\mathcal{A}_j = &\; \frac{4}{\alpha} \partial_{\xi^j}^2 + \frac{2(x-\xi^j)}{y^2 + (x-\xi^j)^2} \partial_x
- \frac{2y}{y^2 + (x-\xi^j)^2} \partial_y \nonumber \\ \label{eq:thePDE}
& + \frac{2}{\xi^1 - \xi^2} \partial_{\xi^1}
+ \frac{2}{\xi^2 - \xi^1} \partial_{\xi^2}, \qquad j = 1,2.
\end{align}
Moreover, the function $\tilde P$ defined by \[\tilde{P}(x,y,\xi^1,\xi^2)=P(x- \xi^1 + iy, \xi^2-\xi^1),\] where $P(z,\xi)$ is defined by \eqref{Pdef}, satisfies the linear PDEs
\[
\mathcal{A}_j\tilde{P}=0, \qquad j=1,2.
\]
\end{prop}
\begin{proof}
Let $z = x+iy$ and $\bar{z} = x-iy$. 
We have
\begin{align}\label{tildecalMm}
\tilde{\mathcal{M}}(x,y,\xi^1, \xi^2) = \int_{\bar{z}}^{z} m(x,y, \xi^1, \xi^2,u) du,
\end{align}
where the integrand $m$ is given by
\begin{align*}
m(x,y,\xi^1, \xi^2, u) = &\; y^{\alpha - 2} (z-\xi^1)^{-\frac{\alpha}{2}}(z- \xi^2)^{-\frac{\alpha}{2}}(\bar{z}-\xi^1)^{1-\frac{\alpha}{2}} (\bar{z} - \xi^2)^{1-\frac{\alpha}{2}}
	\\
&\times (u-z)^{\alpha}(u-\bar{z})^{\alpha - 2} (u-\xi^1)^{-\frac{\alpha}{2}}(u - \xi^2)^{-\frac{\alpha}{2}}.
\end{align*}
Let
\[
\mathcal{B}_j=\mathcal{A}_j - \frac{2}{(x+iy-\xi^j)^2}, \quad j=1,2.
\]
A long but straightforward computation shows that $m$ obeys the equations
\begin{align}\label{BjmPDE}
\mathcal{B}_jm + \frac{2}{u-\xi^j} \partial_u m - \frac{2}{(u-\xi^j)^2} m = 0, \qquad j = 1,2.
\end{align}
Suppose first that $\alpha > 2$. Then we can take the differential operator $\mathcal{B}_j$ inside the integral when computing $\mathcal{B}_j\mathcal{M}$ without any extra terms being generated by the variable endpoints. Hence (\ref{BjmPDE}) implies
$$\mathcal{B}_j\mathcal{M} = - \int_{\bar{z}}^{z} \bigg(\frac{2}{u-\xi^j} \partial_u m - \frac{2}{(u-\xi^j)^2} m\bigg) du, \qquad j = 1,2.$$
An integration by parts with respect to $u$ shows that the integral on the right-hand side vanishes. This shows (\ref{schrammPDEs}) for $\alpha > 2$. The equations in (\ref{schrammPDEs}) follow in the same way for $\alpha \in (1,2)$ if we first replace the contour from $\bar{z}$ to $z$ in (\ref{tildecalMm}) by a Pochhammer contour:
$$\tilde{\mathcal{M}}(x,y,\xi^1, \xi^2) = \frac{1}{(1 - e^{2\pi i \alpha})^2}\int_A^{(z+, \bar{z}+,z-,\bar{z}-)} m(x,y, \xi^1, \xi^2,u) du, \qquad \alpha \neq \Z.$$
If $\alpha = 2$, then
$$m(x,y,\xi^1, \xi^2, u) = \frac{(u-z)^2}{(z-\xi^1)(z-\xi^2)(u-\xi^1)(u-\xi^2)}$$
and (\ref{schrammPDEs}) can be verified by a direct computation. 

It remains to check the last assertion. We have
\[
\tilde P(x,y,\xi^1,\xi^2) =  \int_x^\infty \tilde{m}(x',y,\xi^1, \xi^2)dx',
\]
where $\tilde m = \frac{1}{c_\alpha} \Re \tilde{\mathcal{M}}$. 
Write
\begin{align}\label{calDjdef}
\mathcal{A}_j(x) = \mathcal{D}_j
+ f_j(x)\partial_x + g_j(x) \partial_y,
\end{align}
where
\begin{align*}
& \mathcal{D}_j = \frac{4}{\alpha} \partial_{\xi^j}^2 + \frac{2}{\xi^1 - \xi^2} \partial_{\xi^1}
+ \frac{2}{\xi^2 - \xi^1} \partial_{\xi^2},
	\\
& f_j(x) = \frac{2(x-\xi^j)}{y^2 + (x-\xi^j)^2}, \qquad 
g_j(x) = - \frac{2y}{y^2 + (x-\xi^j)^2},
\end{align*}
and we have only indicated the dependence on $x$ explicitly.
Since $c_\alpha \in \R$ and $\mathcal{A}_j$ has real coefficients, we have
$$(\mathcal{A}_j \tilde{P})(x) = \frac{1}{c_\alpha} \Re \mathcal{A}_j(x) \int_x^\infty \tilde{\mathcal{M}}(x',y,\xi^1, \xi^2)dx'.$$
Employing (\ref{calDjdef}) twice, we find
\begin{align}\nonumber
\mathcal{A}_j(x) \int_x^\infty \tilde{\mathcal{M}}(x')dx'
= &\; \int_x^\infty \mathcal{D}_j\tilde{\mathcal{M}}(x')dx'
- f_j(x)\tilde{\mathcal{M}}(x) + g_j(x) \int_x^\infty \partial_y\tilde{\mathcal{M}}(x')dx'
	\\\nonumber
= &\; \int_x^\infty (\mathcal{A}_j(x') - f_j(x')\partial_{x'} - g_j(x') \partial_y)\tilde{\mathcal{M}}(x')dx'
	\\\label{AjxintMtilde}
& - f_j(x) \tilde{\mathcal{M}}(x) + g_j(x) \int_x^\infty \partial_y\tilde{\mathcal{M}}(x')dx'.
\end{align}	
Using (\ref{schrammPDEs}) to replace $\mathcal{A}_j(x')$ and integrating by parts in the term involving $f_j(x')$, it follows that the right-hand side of (\ref{AjxintMtilde}) equals
\begin{align*}
& \int_x^\infty \bigg(\frac{2}{(x'+iy-\xi^j)^2}  - g_j(x') \partial_y\bigg)\tilde{\mathcal{M}}(x')dx'
	\\
& =  \int_x^\infty \bigg(\frac{2}{(x'+iy-\xi^j)^2} + \partial_{x'}f_j(x')  + (g_j(x) - g_j(x')) \partial_y\bigg)\tilde{\mathcal{M}}(x')dx'.
\end{align*}
Since
$$\frac{2}{(x'+iy-\xi^j)^2} + \partial_{x'}f_j(x')
= -\frac{4iy(x'-\xi^j)}{((x'-\xi^j)^2 + y^2)^2}$$
is purely imaginary and $g_j(x)$ is real-valued, this yields
\begin{align}\nonumber
\mathcal{A}_j \tilde{P} 
= &\; \frac{1}{c_\alpha} \Re \int_x^\infty \bigg(\frac{2}{(x'+iy-\xi^j)^2} + \partial_{x'}f_j(x')  + (g_j(x) - g_j(x')) \partial_y\bigg)\tilde{\mathcal{M}}(x')dx'
	\\\nonumber
= &\; \frac{1}{c_\alpha}  \int_x^\infty \frac{4y(x'-\xi^j)}{((x'-\xi^j)^2 + y^2)^2}
\im \tilde{\mathcal{M}}(x')dx'
	\\\label{calAjtildeP}
& + \frac{1}{c_\alpha}  \int_x^\infty (g_j(x) - g_j(x')) \partial_y\Re \tilde{\mathcal{M}}(x')dx'.
\end{align}
Since $\partial_y\Re \tilde{\mathcal{M}}(x') = -\partial_{x'} \Im \tilde{\mathcal{M}}(x')$ (see Lemma 7.8 in \cite{LV2018B}), we can integrate by parts again to see that
\begin{align}\label{intxinftygj}
 \int_x^\infty (g_j(x) - g_j(x')) \partial_y\Re \tilde{\mathcal{M}}(x')dx'
 = - \int_x^\infty \frac{4y(x'-\xi^j)}{(y^2 + (x'-\xi^j)^2)^2} \Im \tilde{\mathcal{M}}(x')dx'.
\end{align}
Combining (\ref{calAjtildeP}) and (\ref{intxinftygj}), we conclude that $\mathcal{A}_j \tilde{P}=0 $.
\end{proof}

Consider a system of multiple SLEs in $\HH$ started from $0$ and $\xi > 0$, respectively. 
Write $\xi_t^1$ and $\xi_t^2$ for the Loewner driving terms of the system and let $g_t$ denote the solution of (\ref{multiplegtdef}) which uniformizes the whole system at capacity $t$. Then $\xi_t^1$ and $\xi_t^2$ are the images of the tips of the two curves under the conformal map $g_t$.
Given a point $z \in \mathbb{H}$, let $Z_t = g_t(z) $ and let  $\tau(z)$ denote the time that $\Im g_t(z)$ first reaches $0$.

A point $z \in \mathbb{H}$ lies to the left of both curves iff it lies to the left of the leftmost curve $\gamma_1$ started from $0$. Moreover, since the system is commuting, its distribution is independent of the order at which the two curves are grown. Hence we may assume that the growth speeds $\lambda_1$ and $\lambda_2$ are given by $\lambda_1 = 1$ and $\lambda_2 =0$, but this assumption is not essential. We are therefore now  in the setting of SLE$_\kappa(2)$ started from $\xi^1$ with force point at $\xi^2$.


\begin{lemma}\label{Ptclaim}
Let $z \in \mathbb{H}$. Define $P_t(z)$ by
\begin{align*}
  P_t(z) = P(Z_t - \xi_t^1, \xi_t^2 - \xi_t^1), \qquad 0 \leq t < \tau(z).
\end{align*}
Then $P_t(z)$ is an SLE$_\kappa(2)$ martingale.
\end{lemma}
\begin{proof}
It\^{o}'s formula combined with Proposition~\ref{schrammPDEprop} immediately implies that $P_t$ is a local martingale for the SLE$_\kappa(2)$ flow; the drift term vanishes. Since $P$ is bounded by Lemma \ref{Plemma}, it follows that $P_t$ is actually a martingale.

 \end{proof}

%
%
%

\begin{lemma}\label{Theta1claim}
Let $z \in \mathbb{H}$, and $\Theta_t^1 = \arg(Z_t - \xi_t^1)$. 
Then,  
$$\lim_{t \uparrow \tau(z)} \Theta_t^1 = 0 \quad 
\bigg(\text{resp. } \lim_{t \uparrow \tau(z)} \Theta_t^1 = \pi\bigg),$$ 
if and only if $z$ lies to the right (resp. left) of the curve $\gamma_1$ starting at $0$. 
\end{lemma}
\begin{proof}
See the proof of Lemma 3 in \cite{S2001}.
\end{proof}

\begin{lemma}
Let $\tilde{P}(z, \xi)$ be the probability that the point $z \in \mathbb{H}$ lies to the left of the two curves starting at $0$ and $\xi > 0$, respectively. Then $\tilde{P}(z,\xi) = P(z,\xi)$, where $P(z,\xi)$ is the function defined in \eqref{Pdef}.
\end{lemma}
\begin{proof}
By Lemma~\ref{Theta1claim}, the angle $\Theta_t^1 = \arg(Z_t - \xi_t^1)$ approaches $\pi$ as $t \uparrow \tau(z)$ on the event that $z \in \mathbb{H}$ lies to the left of both curves. But \eqref{Peverywhereb1} shows that
$$|P_t(z) - 1|
= |P(Z_t - \xi_t^1, \xi_t^2 - \xi_t^1) - 1|
 \leq C(\pi - \Theta_t^1)^{\alpha -1}, \qquad z \in \mathbb{H}, \ t \in [0, \tau(z)).$$
Consequently, on the event that $z$ lies to the left of both curves, $P_t(z) \to 1$ as $t \uparrow \tau(z)$.
 A similar argument relying on \eqref{Peverywherea1} shows that on the event that $z \in \mathbb{H}$ lies between or to the right of the two curves, then $P_t(z) \to 0$ as $t \uparrow \tau(z)$. 

Let $\tau_n(z)$ be the stopping time defined by
$$\tau_n(z) = \inf\bigg\{ t \geq 0\, : \, \sin \Theta_t^1 \leq \frac{1}{n}\bigg\}.$$
Since $P_t(z)$ is a martingale, we have
$$P_0(z) = \E \left[P_{\tau_n(z)}(z) \right], \qquad z \in \mathbb{H}, \ n = 1, 2, \dots.$$
By using the dominated convergence theorem,
$$\lim_{n \to \infty} \E \left[P_{\tau_n(z)}(z) \right] =\tilde{P}(z, \xi).$$
Since $P_0(z) = P(z, \xi)$, this concludes the proof of the lemma and of Theorem~\ref{thm:schramm}. 
\end{proof}

If $\alpha = \frac{8}{\kappa} > 1$ is an integer, the integral (\ref{Jdef}) defining $J(z, \xi)$ can be computed explicitly. However, the formulas quickly get very complicated as $\alpha$ increases. We consider here the simplest case of $\alpha =2$ (i.e. $\kappa = 4$). We remark that this case is particularly simple for one curve as well;  indeed, the probability that an SLE$_4$ path passes to the right of $z$ equals $(\arg z)/\pi$. 

\begin{prop}
Let $\kappa = 4$. Then the function $P(z,\xi)$ in (\ref{Pdef}) is given explicitly by
\begin{align}\nonumber
P(z, \xi) = &\; \frac{1}{4 \pi ^2 \xi }\bigg\{-2 \arctan\left(\frac{x}{y}\right) \left(\pi  \xi -2 \xi  \arctan\left(\frac{x-\xi
   }{y}\right)+2 y\right)
	\\ \label{Palpha2}
& + \pi ^2 \xi +(4 y-2 \pi  \xi ) \arctan\left(\frac{x-\xi }{y}\right)\bigg\}, \qquad z = x+iy \in \mathbb{H}, \ \xi > 0.
   \end{align}
\end{prop}
\begin{proof}
Let $\alpha = 2$. Then $c_\alpha = -2\pi^2$ and an explicit evaluation of the integral in (\ref{Jdef}) gives  
\begin{align*}
 J(z, \xi) = 2iy + \frac{2i}{\xi}\big((z - \xi)^2\arg(z - \xi) - z^2\arg z\big), \qquad z \in \mathbb{H}, \ \xi > 0.
\end{align*}
Using that
$$\arg z = \frac{\pi}{2} - \arctan \frac{x}{y} \qquad \text{and} \qquad \arg(z -\xi) = \frac{\pi}{2} - \arctan \frac{x - \xi}{y},$$
it follows that the function $\mathcal{M}$ in (\ref{Mdef}) can be expressed as
\begin{align*}
& \mathcal{M}(z, \xi) = \frac{2i}{z(z-\xi)\xi} \bigg\{(z - \xi)^2\bigg(\frac{\pi}{2} - \arctan \frac{x - \xi}{y}\bigg) - z^2\bigg(\frac{\pi}{2} - \arctan \frac{x}{y}\bigg) + \xi y\bigg\}
\end{align*}
for $z = x + iy \in \mathbb{H}$ and $\xi > 0$. 
Taking the real part of this expression and integrating with respect to $x$, we find that the function $P(z,\xi)$ in (\ref{Pdef}) is given by
\begin{align*}
P(z, \xi) = &\; \frac{1}{c_\alpha} \int_x^\infty \re \mathcal{M}(x' +i y, \xi) dx'
	\\
= &\; \frac{1}{2 \pi ^2\xi }\bigg\{(2 y-\pi  \xi ) \left(\frac{\pi }{2}-\arctan\frac{x'-\xi }{y}\right)
-(\pi  \xi +2 y) \left(\frac{\pi }{2}-\arctan\frac{x'}{y}\right)
	\\
& -2 \xi \arctan\Big(\frac{x'}{y}\Big) \arctan\Big(\frac{x'-\xi }{y}\Big)\bigg\}\bigg|_{x'=x}^\infty.
\end{align*}
The expression (\ref{Palpha2}) follows.
\end{proof}

\begin{rem}\upshape
In the fusion limit, equation (\ref{Palpha2}) is consistent with the results of \cite{GC2006}. Indeed, in the limit $\xi \downarrow 0$ the expression (\ref{Palpha2}) for $P(z, \xi)$ reduces to 
$$P(z, 0+) = \frac{1}{4}-\frac{1}{\pi ^2 (1+t^2)}-\frac{\arctan{t}}{\pi } + \frac{(\arctan t)^2}{\pi^2}, \qquad t := \frac{x}{y},$$
which is equation (25) in \cite{GC2006}. 
\end{rem}

\section{The Green's function}\label{greensec}
In this section we prove Theorem~\ref{thm:green}. We recall from the discussion in Section \ref{mainsec} that the proof breaks down into proving Propositions \ref{prop:slekr} and \ref{G2prop}. Proposition \ref{prop:slekr} establishes existence of a Green's function for SLE$_\kappa(\rho)$ and provides a representation for this Green's function in terms of an expectation with respect to two-sided radial SLE. Proposition \ref{G2prop} then shows that the CFT prediction $\GG_\xi(z)$ defined in (\ref{GGdef}) obeys this representation in the case of $\rho = 2$.

\subsection{Existence of the Green's function: Proof of Proposition \ref{prop:slekr}}\label{greensubsec1}
The basic idea of the proof of existence is similar to the ``standard'' one for the Green's function for SLE$_\kappa$ in $\HH$ which we now briefly recall. (See, e.g., \cite{LR2015} for further discussion.) Consider chordal SLE$_\kappa$ from $0$ to $\infty$ in $\HH$. If $\tau = \tau_\ee = \inf\{t\ge 0 : \Upsilon_t(z) \le \ee\}$, we have
\[
 \lim_{\ee \downarrow 0 } \ee^{d-2}\PP(\tau < \infty) = G(z)  \lim_{\ee \downarrow 0 }\E^*[1_{\tau < \infty} S_\tau(z)^{-\beta}]  = G(z)  \lim_{\ee \downarrow 0 }  \E^*[S_\tau(z)^{-\beta}] = c_* G(z).
\]
Here $G(z) = \Upsilon(z)^{d-2} S(z)^\beta$ is the SLE$_\kappa$ Green's function and $\E^*$ refers to the two-sided radial SLE measure with marked point $z \in \HH$. The computation uses that $\PP^*(\tau < \infty) = 1$ and knowledge of the invariant distribution of $S$ under $\PP^*$. 

In the present setting, there are several complications. The SLE$_\kappa$ measure $\PP$ is now weighted by a local martingale \eqref{mg} to accommodate the boundary force point, and one of the main hurdles is to control the magnitude of this weight. Moreover, the expression corresponding to $\E^*[S_\tau(z)^{-\beta}]$ involves an extra weight. Proving convergence is therefore significantly harder and our argument requires control of crossing events for paths near the marked interior point (Lemma~\ref{lem:crossing}).  

Now let us proceed to the proof.
Let $0 < \kappa \le 4$ and $0 \le \rho < 8-\kappa$ and consider SLE$_\kappa(\rho)$ started from  $(\xi^1, \xi^2)$ with $\xi^1 < \xi^2$. We recall our parameters 
\[a=2/\kappa, \quad r=\rho a/2=\rho/\kappa, \quad \zeta(r)=\frac{r}{2a} \left(r+2a-1\right),\]
and the normalized local martingale
\begin{equation}\label{mg}
M_t^{(\rho)} = \left(\frac{\xi^2_t-\xi^1_t}{\xi^2-\xi^1}\right)^{r} g'_t(\xi^2)^{\zeta(r)} 
\end{equation}
by which we can weight SLE$_{\kappa}$ in order to obtain SLE$_{\kappa}(\rho)$, see Section~\ref{sect:prelim}. Let us first derive a few simple estimates on $M_t^{(\rho)}$. 
Note that for our parameter choices we have $r, \zeta \ge 0$. 
The identity
\begin{align}\label{gtprimexi2expression}
g_t'(\xi^2) = \exp\bigg(-\int_0^t \frac{ads}{(\xi_s^2 - \xi_s^1)^2}\bigg)
\end{align}
which follows from (\ref{SLEkapparhodefa}) shows that 
\begin{align}\label{gtprimexi2bound}
0 \leq g_t'(\xi^2) \leq 1, \qquad t \geq 0. 
\end{align}
Moreover, if $E_t$ denotes the interval $E_t = (\xi_t^1, \xi_t^2) \subset \R$ and $\gamma$ the curve generating the Loewner chain $(g_t)_{t \ge 0}$, then conformal invariance of harmonic measure gives
$$\lim_{s\to \infty} s \pi \hm_{\mathbb{H}\smallsetminus \gamma[0,t]}(is, g_t^{-1}(E_t))
= \lim_{s\to \infty} s \pi \hm_{\mathbb{H}}(g_t(is), E_t) = |\xi_t^2 - \xi_t^1|.$$ 
Since the left-hand side is bounded above by a constant times $1 + \diam(\gamma[0,t])$, this gives the estimate 
\begin{align}\label{xi1xi2diam}
|\xi_t^2 - \xi_t^1| \leq C (1 + \diam(\gamma[0,t])), \qquad t \geq 0,
\end{align}
which together with \eqref{gtprimexi2bound} shows that $M_t^{(\rho)}$ only gets large when $\diam(\gamma[0,t])$ gets large. 

We will also need a geometric regularity estimate. In order to state it, let $z \in \HH$ and $0<\ee_1< \ee_2 < \Im z$. Let $\gamma: (0,1] \to \HH$ be a simple curve such that
\[
\gamma(0+) = 0, \quad |\gamma(1)-z| = \ee_1, \quad |\gamma(t) - z| > \ee_1, \, t \in [0,1). 
\]
Write $H = \HH \smallsetminus \gamma$ where $\gamma= \gamma[0,1]$. For $\ee > 0$ let $\ball_\ee = \ball_\ee(z)$ be the disk of radius $\ee$ about $z$  and let $U$ be the connected component containing $z$ of $\ball_{\ee_2} \cap H$. The set $\partial \ball_{\ee_2} \cap \partial U$ consists of crosscuts of $H$. There is a unique outermost one which separates $z$ from $\infty$ in $H$ and we denote this crosscut
\begin{equation}\label{crosscutdef}
\ell = \ell(z, \gamma, \ee_2).
\end{equation}
Outermost means that $\ell$ separates $z$ and any other such crosscut from $\infty$.
See Figure~\ref{crosscutL.pdf}.

\begin{lemma}\label{lem:crossing}
Let $0 < \kappa \le 4$.
There exists $C < \infty$ such that the following holds. Let $z \in \HH$ and $0 < \ee_1 <  \ee_2 < \Im z$. For $\ee>0$ define the stopping times
\begin{align}\label{tautauprimedef}
\tau_\ee = \inf\{t \ge 0 :  \Upsilon_t(z) \le \ee\}, \qquad \tau_\ee'= \inf\{t \ge 0 :  |\gamma(t) - z| \le \ee\}.
\end{align}
If 
\begin{align}\label{lambdaee1ee2def}
\lambda= \lambda_{\ee_1,\ee_2}=\inf\{t \ge \tau_{\ee_1}' :  \gamma(t) \cap \ell \neq \emptyset \}, 
\end{align}
where 
\begin{align}\label{crosscutelldef}
\ell = \ell(z, \gamma_{\tau_{\ee_1}'}, \ee_2)
\end{align}
is as in \eqref{crosscutdef}, then for $0 < 10\ee < \ee_1$, on the event $\{\tau_{\ee_1}' < \infty\}$, 
\begin{align}\label{PPlambdatauepsilon}
\PP \left( \lambda < \tau_\ee < \infty \mid \gamma_{\tau_{\ee_1}'} \right)\le C \left(\frac{\ee}{\ee_1} \right)^{2-d} \left(\frac{\ee_1}{\ee_2} \right)^{\beta/2},
\end{align}
where $\beta = 4a -1$.
\end{lemma}

\begin{figure}
\begin{center}
\begin{overpic}[width=.5\textwidth]{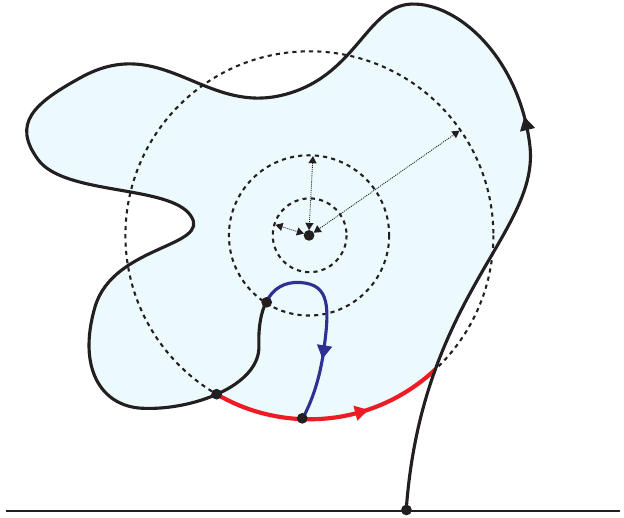}
      \put(29.3,32.5){\small $\gamma(\tau_{\epsilon_1}')$}
      \put(30,15){\small $\gamma(\sigma)$}
      \put(44,11){\small $\gamma(\lambda)$}
      \put(57,18.5){\small $\ell$}
      \put(54,26){\small $\gamma'$}
      \put(86,63){\small $\gamma$}
      \put(64.5,-3){\small $0$}
      \put(61.5,57){\small $\epsilon_2$}
      \put(46,53){\small $\epsilon_1$}
      \put(45.3,43.5){\small $\epsilon$}
      \put(48.5,42){\small $z$}
      \put(32,58){\small $V$}
\end{overpic}
\smallskip
     \begin{figuretext}\label{crosscutL.pdf}
       Schematic picture of the curve $\gamma$ (solid), the open set $V$ (shaded), and the crosscut $\ell$ defined in (\ref{crosscutdef}). If the path reenters $V$ and hits $\ball_\epsilon(z)$ the ``bad'' event that $\lambda < \tau < \infty$ occurs. The probability of this event is estimated in Lemma~\ref{lem:crossing}. The path $\gamma'=\gamma[\sigma, \lambda]$ is a crosscut of $V$ that may either separate $z$ from the $\partial V \smallsetminus \ell$ or not. In either case, the function $S_{\lambda}(z)$ defined in (\ref{StUpsilontdef}) can be estimated by the Beurling estimate since there is only one ``side'' of the curve facing $z$ inside $V$.        
       \end{figuretext}
     \end{center}
\end{figure}

\begin{proof}
Let $\ball_1 = \ball_{\ee_1}(z)$ and $\ball_2 = \ball_{\ee_2}(z)$. Given $\gamma_{\tau_{\ee_1}'}$ we consider the outermost separating crosscut $\ell = \ell(z, \gamma_{\tau_{\ee_1}'}, \ee_2)$. Let $\sigma = \max \{t \le \tau_{\ee_1}': \gamma(t) \in \ell\}$, which is not a stopping time but almost surely $\ell$ is a crosscut of $H_\sigma$ which separates $z$ from $\infty$. Write $V$ for the simply connected component containing $z$ of $H_{\sigma} \smallsetminus \ell$. Because one of the endpoints of $\ell$ is the tip $\gamma(\sigma)$, $g_\sigma(\partial V \smallsetminus \ell) - W_\sigma$ is a bounded open interval $I$ contained in either the positive or negative real axis. (This also uses that $\ell$ is a crosscut.)  Almost surely, the curve $\gamma'=\gamma[\sigma, \lambda]$ is a crosscut of $V$ starting and ending in $\ell$. Note that $g_\sigma(\gamma') -W_\sigma$ is a curve in $\HH$ connecting $0$ with $g_\sigma(\ell)-W_\sigma$, the latter which is a crosscut of $\HH$ separating $I$ and the point $g_\sigma(z) - W_\sigma$ from $\infty$ in $\HH$. Therefore, there is one ``side'' (i.e., one of $g_\sigma^{-1}(W_\sigma \pm \mathbb{R}_+)$) of the curve $\gamma$ such that any curve connecting $z$ with it must intersect $\ell$. If we now write $\delta := \dist(\gamma_\lambda, z) \le \ee_1$, we can use \eqref{Somegaestimate}, the maximum principle, and then the Beurling estimate (Lemma~\ref{beurling1}) to see that (see Figure~\ref{crosscutL.pdf}) on the event $\tau_{\ee_1}' < \infty$,
\[
S_{\lambda}(z) \le C \hm_{V \smallsetminus \gamma'}(z, \ell) \le C \hm_{(\ball_2 \cap V) \smallsetminus \gamma'}(z,\partial \ball_2) \le C \left(\frac{\delta}{\ee_2} \right)^{1/2}.
\] 
Consequently, on the event that $\tau_{\ee_1}' < \infty$ and $\delta \ge 2 \ee$, the one-point estimate Lemma~\ref{lem:one-point} and the distortion estimate (\ref{Koebe}) show that
\[
\PP \left( \tau_\ee < \infty \mid \gamma_{\lambda} \right) \leq C \left(\frac{\ee}{\Upsilon_\lambda(z)}\right)^{2-d} S_\lambda(z)^\beta
\le C \left(\frac{\ee}{\delta}\right)^{2-d} \left(\frac{\delta}{\ee_2} \right)^{\beta/2}.
\]
Since  $\delta \le \ee_1$ and 
\begin{align}\label{beta2dgeq0}
\beta/2 - (2-d) \ge 0 \quad \text{for} \quad  \kappa \le 4,
\end{align}
the right-hand side gets larger if $\delta$ is replaced by $\epsilon_1$. Thus, on the event that $\tau_{\ee_1}' < \infty$ and $\delta \ge 2 \ee$,
\[
\PP \left( \tau_\ee < \infty \mid \gamma_{\lambda} \right) \leq C \left(\frac{\ee}{\ee_1}\right)^{2-d} \left(\frac{\ee_1}{\ee_2} \right)^{\beta/2},
\]
which proves (\ref{PPlambdatauepsilon}) for all curves with $\delta \ge 2 \ee$.

On the event that  $\tau_{\ee_1}' < \infty$ and $\ee < \delta \le 2\ee$, we can use the boundary estimate of Lemma~\ref{lem:boundary-excursion-ED} as follows. We may view $\partial \ball_\delta(z)$ as a crosscut of $\HH \smallsetminus \gamma[0, \lambda]$ possibly considering only a subarc. We use the extension rule (see Chapter~IV of \cite{GM2005}) to estimate from below the extremal distance between $g^{-1}_{\lambda}(W_\lambda + \mathbb{R}_+)$ (or $g^{-1}_{\lambda}(W_\lambda - \mathbb{R}_+)$, whichever does not intersect $\partial \ball_\delta(z)$ viewed as a crosscut) and $\ball_\delta(z)$ in $\mathbb{H} \smallsetminus \gamma[0, \lambda]$ by the extremal distance between $\partial \ball_2$ and $\ball_\delta(z)$ in $\ball_2 \cap V$. By comparing with the round annulus, the latter is at least $\ln(\ee_2/\ee)/(2\pi)$. Therefore, Lemma~\ref{lem:boundary-excursion-ED} gives
\[
\PP \left(\tau_\ee < \infty \mid \gamma_{\lambda}\right) \leq C e^{-\beta \pi \frac{\ln(\ee_2/\ee)}{2\pi}} = C \, \left(\frac{\ee}{\ee_2} \right)^{\beta/2}.
\] 
It follows from (\ref{beta2dgeq0}) that, on the event that $\tau_{\ee_1}' < \infty$ and $\ee < \delta \le 2\ee$,
$$\PP \left(\tau_\ee < \infty \mid \gamma_{\lambda}\right)  \leq C \left(\frac{\ee}{\ee_1}\right)^{2-d} \left(\frac{\ee_1}{\ee_2} \right)^{\beta/2},$$
which proves (\ref{PPlambdatauepsilon}) also for curves with $\ee < \delta \le 2\ee$.
\end{proof}

\subsubsection{Proof of Proposition \ref{prop:slekr}}
We may without loss of generality assume $\xi^{1}=0$ and $|z| = 1$. Constants are allowed to depend on $z$ and $\xi^2$ as well as on $\kappa$ and $\rho$.

We will apply Lemma~\ref{lem:crossing} with 
$$\ee_1 = \ee^{1/2} \quad \text{and} \quad \ee_2 = \ee^{1/4}.$$ 
By choosing $\ee$ sufficiently small, we may assume that $\ee_2 = \ee^{1/4} < \Im z$. For $\ee > 0$, let $\tau = \tau_\epsilon$, $\tau' = \tau_\epsilon'$, and $\lambda = \lambda_{\ee_1, \ee_2}$ be defined by (\ref{tautauprimedef}) and (\ref{lambdaee1ee2def}). Let $\ell = \ell(z, \gamma_{\tau'_{\ee_1}} , \ee_2)$ denote the separating crosscut in (\ref{crosscutelldef}). 
Let $ p \in(0, 1/4)$ be a constant (to be chosen later) and set $\Sigma=\Sigma_{\ee} = \inf\{t \ge 0: |\gamma(t)| \ge \ee^{-p}\}$.

We define a ``good'' event $E = E_{\ee}$  by 
\[E = E_{\ee}  = E_1 \cap E_2\]
where
\[ E_1 = \{\tau < \lambda\}, \qquad E_2 = \{ \tau < \Sigma \}.
\]

We claim that 
\begin{equation}\label{aug17.1}
\lim_{\ee \downarrow 0}\ee^{d-2}\PP^{\rho} \left( \tau < \infty, E^{c} \right)= 0,
\end{equation}
where $E^c$ denotes the complement of $E$. To prove (\ref{aug17.1}), it is clearly enough to show that
\begin{align}\label{limitE1c}
\lim_{\ee \downarrow 0}\ee^{d-2}\PP^{\rho} \left( \tau < \infty, E_1^{c} \right)= 0
\end{align}
and
\begin{align}\label{limitE2c}
\lim_{\ee \downarrow 0}\ee^{d-2}\PP^{\rho} \left( \tau < \infty, E_2^{c} \right)= 0.
\end{align}
Let $\sigma_{0}=0$ and define $\sigma_k$ and $U_k$ for $k=1,2,\ldots$, by
\[
\sigma_k = \inf_{t \ge 0}\{|\gamma(t)| \ge 2^k \}, \qquad
U_k = \{\sigma_{k-1} \le \tau < \sigma_k < \infty\}.
\]
Using \eqref{SLEkr_mg}, \eqref{gtprimexi2bound}, and \eqref{xi1xi2diam}, we obtain the estimates
\begin{align}\nonumber
\PP^\rho \left( \tau < \infty, \, E_i^c \right) 
& = \E\left[M_\tau^{(\rho)} 1_{\tau < \infty} 1_{E_i^c} \right] 
\leq C \E\left[(1 + \diam(\gamma[0,\tau])^r) 1_{\tau < \infty} 1_{E_i^c} \right] 
	\\\label{aug17.01}
& \le  C \left( \Prob(\tau<\infty, E_i^c) +   \sum_{k=1}^\infty 2^{kr} \PP\left(E_i^c, \, U_k \right) \right), \quad i = 1,2.
\end{align}
We will use (\ref{aug17.01}) to prove (\ref{limitE1c}) and (\ref{limitE2c}).

We first prove (\ref{limitE1c}).
Lemma~\ref{lem:crossing} and Lemma~\ref{lem:one-point} yield 
\begin{align}\nonumber
\Prob(\tau<\infty, E_1^c)
& \leq \PP \left( \lambda < \tau < \infty \mid \gamma_{\tau_{\ee_1}'} \right) \Prob(\tau_{\ee_1}'<\infty)
	\\ \label{ProbtauE1c}
& \leq C \left(\frac{\ee}{\ee_1} \right)^{2-d} \left(\frac{\ee_1}{\ee_2} \right)^{\beta/2}
\bigg(\frac{\ee_1}{\Upsilon_{\mathbb{H}}(z)}\bigg)^{2-d}
= o(\ee^{2-d}).
\end{align}
We next estimate the series on the right-hand side of (\ref{aug17.01}). We first consider $i=1$. For this, suppose $j = -1, \ldots, J := \lfloor \log_2(\ee^{-1}) \rfloor +1$ and define
\[
V_j^k = \{ \tau'_{\ee 2^{j}} \le \sigma_{k-1} < \tau'_{\ee 2^{j-1}}\}.
\]
The distortion estimate (\ref{Koebe}) implies $\tau \leq \tau_{\epsilon 2^{-1}}'$. Also, $\tau'_{\ee 2^J} = 0$ because $|z| = 1$. 
Hence, $U_k = \cup_{j=-1}^J (U_k \cap V_j^k)$ and so
\begin{equation} \label{eq:Us}
\PP \left(E_1^c, \, U_k \right)= \sum_{j =-1}^{J}\PP\left(E_1^c, \, U_k, \, V_j^k\right). 
\end{equation}
Let us first assume $-1 \leq j \leq \lfloor \frac{1}{2}\log_2 \ee^{-1} \rfloor$. We claim that, on the event $V_{j}^k \cap \{\sigma_{k-1} < \infty\}$,
\begin{equation} \label{eq:Us2}
\PP\left(\tau < \infty, E_1^c \mid \gamma_{\sigma_{k-1}} \right)
\le \PP \left( \tau < \infty \mid \gamma_{\sigma_{k-1}} \right)
\le C \bigg(\frac{2^j \ee}{2^{2k}}\bigg)^{\beta/2} \bigg(\frac{\ee}{2^j \ee}\bigg)^{2-d}. 
\end{equation}
Indeed, the first estimate in (\ref{eq:Us2}) is trivial and the second follows from Lemma~\ref{lem:one-point} as follows. The curve $\gamma_{\sigma_{k-1}}$ is a crosscut of  $D:= 2^{k-1} \mathbb{D} \cap \mathbb{H}$ and so partitions $D$ into exactly two components, one of which contains $z$. Consequently, using \eqref{Somegaestimate} and the maximum principle, we have $S_{\sigma_{k-1}}(z) \leq C \hm_{D \smallsetminus \gamma[0, \sigma_{k-1}]}(z, \partial D)$, that is, we can estimate $S_{\sigma_{k-1}}(z) $ by the probability of a Brownian motion from $z$ to reach distance $2^{k-1}$ from $0$ before hitting the real line or the curve $\gamma[0, \sigma_{k-1}]$. Thus, given the path up to time $\sigma_{k-1}$, on $V^k_j$, the Beurling estimate (Lemma~\ref{beurling1}) shows that the probability that a Brownian motion starting at $z$ reaches the circle of radius $2 \Im z$ about $z$ without exiting $H_{\sigma_{k-1}}$ is $O((\ee2^j / \Im z)^{1/2})$. Given this, we claim that the probability to reach distance  $2^{k-1}$ from $0$ is $O(\Im z/2^k)$. Indeed, this follows easily, e.g., from the fact that $\hm_{\DD \cap \HH}(z, \partial \DD)=2(\pi - \nu(z))/\pi$, where $\nu(z)$ is the obtuse angle $\angle (-1, z, 1)$. Hence, since $\Im z  \le 1,$ we see that $S_{\sigma_{k-1}}(z)^\beta \leq C \,(\ee 2^{j}/2^{2k})^{\beta/2}$. Moreover, by the distortion estimate (\ref{Koebe}), we have $\Upsilon_{H_{\sigma_{k-1}}}(z) \asymp 2^j \epsilon$ on $V_j^k$. 
Hence, the one-point estimate (Lemma~\ref{lem:one-point}) gives \eqref{eq:Us2}. 

Lemma~\ref{lem:one-point} also shows that
\[
\PP\left( V_{j}^k ,\,  \sigma_{k-1} < \infty \right) 
\le C  (2^{j-1}\ee)^{2-d} S_0^\beta,
\]
which combined with \eqref{eq:Us2} gives
$$ \PP\left(E_1^c, \, U_k, \, V_j^k\right)
\leq C  \bigg(\frac{2^j \ee}{2^{2k}}\bigg)^{\beta/2} \bigg(\frac{\ee}{2^j \ee}\bigg)^{2-d}
 (2^{j-1}\ee)^{2-d} 
 \leq C \epsilon^{2-d+ \frac{\beta}{2}} 2^{j(\frac{\beta}{2}-(2-d)) -k\beta}.$$
Summing over $j$ from $-1$ to $\lfloor \frac{1}{2}\log_2 \ee^{-1} \rfloor$ and recalling (\ref{beta2dgeq0}), we find
\begin{equation}\label{aug17.111}
\sum_{j=-1}^{\lfloor \frac{1}{2}\log_2 \ee^{-1} \rfloor} \PP\left(E_1^c, \, U_k, \, V_j^k\right) 
\leq C \epsilon^{2-d + \frac{\beta}{2}} \epsilon^{-\frac{\beta}{4} + \frac{2-d}{2}} 2^{-k \beta}
\le C \, \ee^{2-d + \frac{\beta}{4}} 2^{-k \beta}. 
\end{equation}

Suppose now that $\lfloor \frac{1}{2}\log_2 \ee^{-1} \rfloor + 1 \leq j \leq J$. Lemma~\ref{lem:crossing} with $\ee_1 = \ee^{1/2}$ and $\ee_2 = \ee^{1/4}$ implies that, on the event $\{\tau'_{\ee_1} < \infty\}$,
\begin{align}\label{PPtauE1cgamma}
\PP\left(\tau < \infty, \, E_1^c \, \middle| \, \gamma_{\tau'_{\ee_1}}\right) 
\leq C \left(\frac{\ee}{\ee_1} \right)^{2-d} \left(\frac{\ee_1}{\ee_2} \right)^{\beta/2}
= C \, \ee^{\frac{2-d}{2} + \frac{\beta}{8}}
\end{align}
for all sufficiently small $\epsilon$.
Moreover, on the event $V_j^k \cap \{ \sigma_{k-1} < \infty\}$, we can estimate $S_{\sigma_{k-1}}$ as when proving \eqref{eq:Us2} and use Lemma~\ref{lem:one-point} to see that
\begin{align}\label{PPtauprimeee12}
\PP\left(\tau'_{\ee_1} < \infty \mid \gamma_{\sigma_{k-1}}\right) \le C \, \left(\frac{2^j \ee}{2^{2k}} \right)^{\beta/2}  \left( \frac{\ee_1}{2^j \ee} \right)^{2-d}.
\end{align}
We conclude from (\ref{PPtauE1cgamma}) and (\ref{PPtauprimeee12}) that 
\begin{align*}
\PP \left(E_1^c, \, U_k, \, V_j^k \right) 
\le  C \, \ee^{2-d +\frac{\beta}{8}} (2^{j}\ee)^{\frac{\beta}{2}-(2-d)} 2^{-k \beta}.
\end{align*}
Summing over $j = \lfloor \frac{1}{2}\log_2 \ee^{-1} \rfloor +1, \ldots, J$ and using that $\beta/2-(2-d) \ge 0$ with equality only if $\kappa = 4$ (see (\ref{beta2dgeq0})), we infer that
\begin{equation*}
\sum_{j= \lfloor \frac{1}{2}\log_2 \ee^{-1} \rfloor +1}^J \PP\left(E_1^c, \, U_k, \, V_j^k\right) 
\leq C \epsilon^{2-d + \frac{\beta}{8}} \log_2(\epsilon^{-1}) \epsilon^{\frac{\beta}{4} - \frac{2-d}{2}} 2^{-k \beta}
= o(\ee^{2-d}) 2^{-k \beta},
\end{equation*}
where the factor $\log_2(\epsilon^{-1})$ needs to be included only when $\kappa = 4$.
Together with (\ref{eq:Us}) and (\ref{aug17.111}), this gives
\begin{align}\label{2rkPPtauinfty}
2^{rk} \PP \left(E_1^c, \, U_k \right) \le 2^{(r - \beta)k}o(\ee^{2-d}) .
\end{align}
Since $r-\beta < 0$ (this is equivalent to the condition $\rho  <  8-\kappa$), we can sum the right-hand side of (\ref{2rkPPtauinfty}) over all integers $k \geq 1$ and the result is $o(\ee^{2-d})$. 
In view of (\ref{aug17.01}) and (\ref{ProbtauE1c}), this proves (\ref{limitE1c}).

We next prove (\ref{limitE2c}). Note that $E^c_2 \subset \cup_{\lfloor p \log_2 (\ee^{-1}) \rfloor}^\infty U_k$. Given $\gamma_{\sigma_{k-1}}$, we can estimate harmonic measure as when proving \eqref{eq:Us2} to find $S_{\sigma_{k-1}}(z)^\beta \le C \Upsilon_{H_{\sigma_{k-1}}}^{\beta/2}2^{-\beta k}$ on $U_k$.
Therefore, estimating as in \eqref{aug17.01} with the help of (\ref{xi1xi2diam}), and then using the one-point estimate (see Lemma~\ref{lem:one-point}) and the fact that $\beta/2-(2-d) \ge 0$, we obtain 
\begin{equation}\label{Uk-estimate}
\PP^{\rho}(U_k) \le C 2^{rk}\PP(U_k) 
\leq C 2^{rk} \bigg(\frac{\epsilon}{\Upsilon_{H_{\sigma_{k-1}}}(z)}\bigg)^{2-d} S_{\sigma_{k-1}}(z)^\beta
\le C 2^{(r-\beta) k} \ee^{2-d}.
\end{equation}
Since $r-\beta < 0$, we can sum over all integers $k$ with $k \ge p \log_2 (\ee^{-1})$ and the result is $o(\ee^{2-d})$. 
We obtain $\lim_{\ee \downarrow 0} \ee^{d-2}\PP^\rho(\tau < \infty, E_2^c) = 0$ for any $p > 0$.
This proves (\ref{limitE2c}) and hence completes the proof of \eqref{aug17.1}.

In view of \eqref{aug17.1}, it only remains to prove that
\[
\lim_{\ee \downarrow 0}\ee^{d-2}\PP^\rho\left(\tau < \infty, \, E\right) = c_*G^{\rho}(z, \xi^1, \xi^2).
\]
 According to equation (\ref{twosidedradialdef}), we have
$$\E^*[f] = \tilde{G}_0^{-1}\E[\tilde{G}_t f], \qquad t \geq 0,$$
whenever $f \in L^1(\PP^*)$ is measurable with respect to $\tilde{\mathcal{F}}_t$.
We change to the radial time parametrization and set $t = -(\ln \epsilon)/(2a)$, so that $\epsilon = e^{-2at}$ and $s(t) = \tau=\tau_\epsilon$. Then $\tilde{G}_t = G_{\tau_\epsilon}$ and the function $M_{\tau_\ee }^{(\rho)} 1_{\tau_\ee < \infty} 1_E$ is measurable with respect to $\tilde{\mathcal{F}}_t = \mathcal{F}_{\tau_\epsilon}$, so we find
\begin{align}\label{PPrhotau}
\PP^\rho\left( \tau < \infty, \, E\right)= \E\left[ M_{\tau }^{(\rho)} 1_{\tau  < \infty} 1_E \right] 
= G_0 \, \E^*\left[G_{\tau }^{-1} M_{\tau }^{(\rho)} 1_{\tau   < \infty} 1_E \right],
\end{align}
where $G_{0}$ is the SLE$_{\kappa}$ Green's function. 
Thanks to the boundary conditions of the martingale $G_t$, we have $G_{s(t)} = G_\infty = 0$ on the event $\tau = \infty$. This means that $\E^*[1_{\tau = \infty}] = G_0^{-1}\E[G_\infty 1_{\tau = \infty}] = 0$. Hence we can remove the factor $1_{\tau  < \infty}$ from the right-hand side of (\ref{PPrhotau}).
Thus, using the definition (\ref{Gtdef}) of $G$,
$$
\PP^\rho \left( \tau  < \infty, \, E \right)
= \ee^{2-d} \, G_0 \, \E^*\left[ M_{\tau }^{(\rho)} S_{\tau  }^{-\beta}  1_E \right],
$$
where $S_{\tau } = S_{\tau  }(z)$.
We need to show that
\begin{equation}\label{eq:tau'}
\lim_{\ee \downarrow 0} \E^*\left[ M_{\tau}^{(\rho)} S_{\tau}^{-\beta}  1_{E}  \right] = c_* \E^* \left[ M_T^{(\rho)} \right], \qquad \tau=\tau_\ee,
\end{equation}
and where $T$ is the time at which the path reaches $z$. 
 Let $\tau'' = \tau_{\ee^{1/2}/4}$. Then $\tau'_{\ee^{1/2}} \le \tau'' \le \tau$ if $\ee$ is small enough. We claim that we can choose $p$ so that
\begin{equation} \label{m-estimate}
\left| M_{\tau}^{(\rho)} - M_{\tau''}^{(\rho)}  \right|1_E = o(1).
\end{equation}
Indeed, suppose we are on the event $E=E_1 \cap E_2$. Then, by the definition of $E_1$, $\diam \gamma[\tau'', \tau] \le 2\ee^{1/4}$ if $\ee$ is small enough. Moreover, if $R:=\diam \gamma[0, \tau]$ then $R \le \ee^{-p}$ by the definition of $E_2$. Hence $\ee^{1/4}R \le \delta$ where $\delta : = \ee^{1/4 -p}$. Combining these bounds, we see that (see Proposition~3.82 of \cite{L2005}) $\diam g_{s}(\gamma[s, \tau]) \leq C \delta^{1/2}$ uniformly for $s \in [\tau'', \tau]$. Using that $\hcap K \le C \, (\diam K)^2$ for a half-plane hull $K$, it follows that
\begin{align}\label{tauDeltatau}
\tau-\tau'' \le C \, \delta \quad \text{and} \quad \text{$\Delta_{s} = \Delta_{\tau''} + O(\delta^{1/2})$ for $s \in[ \tau'', \tau]$},
\end{align}
where $\Delta_t: = \xi^2_t - \xi^1_t$. Therefore, using also (\ref{gtprimexi2expression}) and \eqref{gtprimexi2bound}, 
\begin{align*}
\Delta_\tau^r g_\tau'(\xi^2)^{\zeta}-\Delta_{\tau''}^r g_{\tau''}'(\xi^2)^{\zeta} 
& = (\Delta_\tau^r - \Delta_{\tau''}^r)g_\tau'(\xi^2)^{\zeta} 
+ \Delta_{\tau''}^r(g_\tau'(\xi^2)^{\zeta} - g_{\tau''}'(\xi^2)^{\zeta})
	\\
& = \Delta_{\tau''}^r g_{\tau''}'(\xi^2)^{\zeta}\left[\exp\left(-\zeta \int_{\tau''}^\tau a/\Delta_s^2 ds\right) - 1\right] + O(\delta^{1/2}).
\end{align*}
On the event $\Delta_{\tau''} \le \delta^{1/4}$, the right-hand side is $o(1)$ by (\ref{gtprimexi2bound}). By \eqref{xi1xi2diam} we may therefore assume $ \delta^{1/4} <  \Delta_{\tau''} \le CR$. Now choose any $p \in (0,1/4)$ so that $\ee^{(1-r)p +1/4} = o(1)$. Then by Taylor expansion, using \eqref{xi1xi2diam} and (\ref{tauDeltatau}),
\[
\Delta_{\tau''}^r\left[\exp\left(-\zeta \int_{\tau''}^\tau a/\Delta_s^2 ds\right) - 1\right] 
\leq C \Delta_{\tau''}^r \int_{\tau''}^\tau \frac{a}{\Delta_s^2}ds 
\le C \, \Delta_{\tau''}^{r-2}|\tau-\tau''|= o(1),
\]
where we have used that $p \in (0,1/4)$ and $r \ge 0$.
This concludes the proof of \eqref{m-estimate}. We fix $p$ so that \eqref{m-estimate} holds for the remainder.

Using the invariant distribution (see Lemma~\ref{lem:radial-bessel} with $f(x) = \sin^{-\beta}(x)$) we have that $\E^*\left[S_{\tau}^{-\beta} \right] = O(1)$, so
\begin{equation}\label{a1}
\left| \E^*\left[\left( M_{\tau}^{(\rho)} - M_{\tau''}^{(\rho)}  \right) S_\tau^{-\beta} 1_E\right] \right| \le o(1) \, \E^*\left[S_{\tau}^{-\beta} \right] = o(1),
\end{equation}
On the other hand, since $M^{(\rho)}_{\tau''}1_{U_k} \le C 2^{kr}$ the same argument that proved \eqref{aug17.1} shows that
\begin{equation}\label{a2}
\E^*\left[  M_{\tau''}^{(\rho)} S_\tau^{-\beta}   1_{E^c} \right] = o(1)
\end{equation}
as $\ee \to 0$. It follows from \eqref{a1} and \eqref{a2} that
\[
\E^*\left[ M_{\tau}^{(\rho)} S_\tau^{-\beta}   1_{E}  \right] = \E^*\left[ M_{\tau''}^{(\rho)} S_\tau^{-\beta}   \right] + o(1).
\]
Moreover, 
\[
\E^*\left[  M_{\tau''}^{(\rho)} S_\tau^{-\beta}  \right] = \E^*\left[ M_{\tau''}^{(\rho)} \, \E^*\left[ S_\tau^{-\beta} \mid \mathcal{F}_{\tau''} \right]\right].
\]
Using Lemma~\ref{lem:radial-bessel} we see that there is $\alpha > 0$ such that
\[
\E^*\left[  S_\tau^{-\beta} \mid \mathcal{F}_{\tau''} \right] = \frac{c_*}{2} \int_0^{\pi} \sin \theta \, d\theta \, \left(1+O(\ee^\alpha)\right) = c_*(1+O(\ee^\alpha)).
\]
To complete the proof of Proposition~\ref{prop:slekr} it only remains to show that
\[
\lim_{\ee \downarrow 0}\E^*\left[ M_\tau^{(\rho)} \right] = \E^*\left[M_T^{(\rho)} \right].
\]
For this we check that the sequence of integrands $\{M_\tau^{(\rho)}\}$ is uniformly integrable as $\ee \downarrow 0$. Since the only way in which $M_\tau^{(\rho)}$ can get large is by the path reaching a large diameter, uniform integrability follows almost immediately from what we have done. Since we will need to use this below (for $\rho =2$) we formulate this as a separate lemma and give the proof in which we keep the notation introduced in this subsection.
\begin{lemma}\label{uniformly-integrable}
Let $\rho  \in [0, 8-\kappa)$. Then the sequence of integrands $\{M_\tau^{(\rho)}\}$ is uniformly integrable as $\ee \downarrow 0$, where $\tau=\tau_\ee$.
\end{lemma}

\begin{proof}

We want to show that for every $\epsilon_0 > 0$, there exists an $R > 0$ such that
\begin{align}\label{Muniformlyintegrable}
\E^*\Big[M_{\tau}^{(\rho)} 1_{\{M_{\tau}^{(\rho)} > R\}}\Big] < \epsilon_0
\end{align}
for all small $\epsilon > 0$. Let us prove (\ref{Muniformlyintegrable}).


The estimate (\ref{xi1xi2diam}) yields
$$|\xi_{\tau}^2 - \xi_{\tau}^1| \leq C\, 2^{k} \quad \text{on} \quad U_k, \qquad k \geq0,$$
so, in view of (\ref{gtprimexi2bound}), there exists a constant $C_0>0$ such that
$$|M_{\tau}^{(\rho)}| \leq C_0 2^{kr} \quad \text{on} \quad U_k ,\qquad k \geq0.$$

Suppose $R > 0$ and let $N := \lfloor r^{-1}\log_2(R/C_0) \rfloor$.
Then
$$|M_{\tau}^{(\rho)}| \leq C_0 2^{kr} \leq R \quad \text{on} \quad U_k,\qquad 0 \leq k \leq N-1.$$
Hence
\begin{align}\label{EstarFepsilon}
\E^*[M_{\tau}^{(2)} 1_{\{M_{\tau}^{(2)} > R\}}] 
\leq \E^*[M_{\tau}^{(\rho)} 1_{\cup_{j= N}^\infty U_k}]
\leq C_0 \sum_{j=N}^\infty 2^{kr} \PP^*\big(U_k\big).
\end{align}
Since $U_k$ is $\mathcal{F}_{\tau}$-measurable, (\ref{Estargirsanov}) gives
$$\Prob^*\big( U_k\big) = \E^*\big[ 1_{U_k} \big] 
= G_0^{-1} \E[G_{\tau} 1_{U_k}]
\leq C \epsilon^{d-2} \E[  1_{U_k}],$$
where we have used the following estimate in the last step: 
$$|G_{\tau}| = (\Upsilon_0(z) \epsilon)^{d-2} S_\tau(z)^\beta \leq C \epsilon^{d-2}.$$
It follows from \eqref{Uk-estimate} that
\begin{align}\label{PEj}
\Prob \left(U_k\right) \leq C \epsilon^{2-d} 2^{-k\beta} \quad \text{for} \quad k > \log_2(4|z|).
\end{align}
Therefore, we find
$$\Prob^*\left(  U_k \right) \leq C 2^{-k\beta}, \qquad  k > \log_2(4|z|).$$
Employing this estimate in (\ref{EstarFepsilon}) we obtain
\begin{align*}
\E^*[M_{\tau}^{(\rho)} 1_{\{M_{\tau}^{(\rho)} > R\}}] 
\leq C \sum_{k=N}^\infty 2^{-k(\beta-r)} \le C 2^{-N(\beta-r)}.
\end{align*}
The condition $N > \log_2(4|z|)$ is fulfilled for all sufficiently large $R$. Since $N = [r^{-1}\log_2(R/C_0)]$ and $\beta - r > 0$, given $\epsilon_0 > 0$, by choosing $R$ large enough, we can make $\E^*[M_{\tau}^{(\rho)} 1_{\{M_{\tau}^{(\rho)} > R\}}] < \epsilon_0$ for all $\epsilon > 0$, which implies (\ref{Muniformlyintegrable}) and concludes the proof of Lemma~\ref{uniformly-integrable}. 
%
\end{proof}

Given Lemma~\ref{uniformly-integrable}, the proof of Proposition~\ref{prop:slekr} is now complete.
\qed

\subsection{Probabilistic representation for $\mathcal{G}$: Proof of Proposition \ref{G2prop}}\label{greensubsec2}
Let $0 < \kappa \leq 4$. Our goal is to show that
\begin{align}\label{GGGEstar}
\GG(z, \xi^1, \xi^2) = (\im z)^{d-2} \sin^{\beta}(\arg(z-\xi^1)) \E^*[M_T^{(2)}], \qquad z \in \mathbb{H}, \ \xi^1 < \xi^2,
\end{align}
where $\E^*$ denotes expectation with respect to two-sided radial SLE$_\kappa$ from $\xi^1$ through $z$, stopped at the hitting time $T$ of $z$ and $\mathcal{G}$ is our prediction for the Green's function. 
Our first step is to use scale and translation invariance to reduce the relation (\ref{GGGEstar}), which depends on the four real variables $x=\re z, y=\im z, \xi^1, \xi^2$, to an equation involving only two independent variables. 

\subsubsection{The function $h(\theta^1, \theta^2)$}
It follows from (\ref{Idef}) and (\ref{GGdef}) that $\GG$ satisfies the scaling behavior
$$\GG(\lambda z, \lambda \xi^1, \lambda \xi^2) = \lambda^{d-2} \GG(z, \xi^1, \xi^2), \qquad \lambda > 0.$$
Hence we can write
\begin{align*}
\GG(z,\xi^1, \xi^2) = y^{d-2} \mathcal{H}(z,\xi^1, \xi^2),
\end{align*}
where the function $\mathcal{H}$ is homogeneous and translation invariant:
\begin{subequations}\label{Hscaletranslation}
\begin{align}
& \mathcal{H}(\lambda z, \lambda\xi^1, \lambda\xi^2) = \mathcal{H}(z,\xi^1, \xi^2), \qquad \lambda > 0,
	\\ 
& \mathcal{H}(z,\xi^1, \xi^2) = \mathcal{H}(x + \lambda, y, \xi^1 + \lambda, \xi^2 + \lambda), \qquad \lambda \in \R.
\end{align}
\end{subequations}
It follows that the value of $\mathcal{H}(x,y, \xi^1, \xi^2)$ only depends on the two angles $\theta^1$ and $\theta^2$ defined by
$$\theta^1 = \arg(z - \xi^1), \qquad \theta^2 = \arg(z - \xi^2).$$ 
In particular, if we let $\Delta$ denote the triangular domain 
$$\Delta = \{(\theta^1, \theta^2) \in \R^2\, | \, 0 < \theta^1 < \theta^2 < \pi\},$$
then we can define a function $h:\Delta \to \R$ for $\alpha \in (1, \infty) \smallsetminus \Z$ by the equation
\begin{align}\label{hdef}
\GG(z,\xi^1, \xi^2) = y^{d-2} h(\theta^1, \theta^2), \qquad z \in \mathbb{H}, \; -\infty < \xi^1 < \xi^2 < \infty.
\end{align}
Using Lemma \ref{halphaintegerlemma}, we can extend the definition of $h$  to all $\alpha \in (1, \infty)$ by continuity. We write $h(\theta^1, \theta^2; \alpha)$ if we want to indicate the $\alpha$-dependence of $h(\theta^1, \theta^2)$ explicitly.
In terms of $h$, we can then reformulate equation (\ref{GGGEstar}) as follows:
\begin{align}\label{hEstar}
h(\theta^1, \theta^2; \alpha) = (\sin^{\beta} \theta^1 )\E^*[M_T^{(2)}], \qquad (\theta^1, \theta^2) \in \Delta, \ \beta \geq 1.
\end{align}


The following lemma, which is crucial for the proof of (\ref{hEstar}), describes the behavior of $h$ near the boundary of $\Delta$. In particular, it shows that $h(\theta^1, \theta^2)$ vanishes as $\theta^1$ approaches $0$ or $\pi$, and that the restriction of $h$ to the top edge $\theta^2 = \pi$ of $\Delta$ equals $\sin^{\beta} \theta^1$.
In other words, the lemma verifies that $\GG(z,\xi^1, \xi^2)$ satisfies the appropriate boundary conditions.
\begin{lemma}[Boundary behavior of $h$]\label{hboundarylemma}
Let $\alpha \geq 2$. Then the function $h(\theta^1, \theta^2)$ defined in (\ref{hdef}) is a smooth function of $(\theta^1, \theta^2) \in \Delta$ and has a continuous extension to the closure $\bar{\Delta}$ of $\Delta$. This extension satisfies
\begin{align}\label{h2ontopedge}
& h(\theta^1, \pi) = \sin^\beta{\theta^1}, \qquad \theta^1 \in [0, \pi],
	\\
& h(\theta, \theta) = h_f(\theta), \qquad \theta \in (0, \pi),
\end{align}
where $h_f(\theta)$ is defined in (\ref{hfdef}).
Moreover, there exists a constant $C > 0$ such that
\begin{align}\label{h2estimate1}
0 \leq h(\theta^1, \theta^2) \leq C \sin^{\beta}\theta^1, \qquad (\theta^1, \theta^2) \in \bar{\Delta},
\end{align}
and
\begin{align}\label{h2estimate2}
\frac{|h(\theta^1, \theta^2) - h(\theta^1, \pi)|}{\sin^{\beta}\theta^1} \leq C\frac{|\pi - \theta^2|}{\sin\theta^1}, \qquad (\theta^1, \theta^2) \in \Delta.
\end{align}
\end{lemma}
\begin{proof}
The rather technical proof involves asymptotic estimates of the integral in (\ref{Idef}) and is given in \cite{LV2018B}.
\end{proof}

\begin{prop}[PDE for Green's function]\label{greenPDEprop}
Suppose $\alpha \geq 2$. The function $\mathcal{G}(x + iy, \xi^1, \xi^2)$ defined in (\ref{GGdef}) satisfies the two linear PDEs
$$\left(\mathcal{A}_j+\frac{2(\alpha-1)(y^2 - (x-\xi^j)^2)}{\alpha(y^2 + (x-\xi^j)^2)^2} \right)\mathcal{G} = 0, \qquad j =1,2,$$
where the differential operators $\mathcal{A}_j$ were defined in \eqref{eq:thePDE}.
\end{prop}
\begin{proof}
We have
$$\mathcal{G}(x + iy, \xi^1, \xi^2) = \im \int_{\gamma} g(x,y,\xi^1, \xi^2, u) du,$$
where $\gamma$ denotes the Pochhammer contour in (\ref{Idef}) and the integrand $g$ is given by
\begin{align*}
g(x,y,\xi^1, \xi^2, u) = &\; \hat{c}^{-1} y^{\alpha + \frac{1}{\alpha} - 2} |z - \xi^1|^{1 - \alpha} |z - \xi^2|^{1 - \alpha} e^{-i\pi\alpha} 
	\\
&\times (u - z)^{\alpha -1} (u - \bar{z})^{\alpha -1} (u - \xi^1)^{-\frac{\alpha}{2}} (\xi^2 - u)^{-\frac{\alpha}{2}}.
\end{align*}
A long but straightforward computation shows that $g$ obeys the equations
$$\mathcal{A}_j g + \frac{2}{u-\xi^j} \partial_u g - \frac{2}{(u-\xi^j)^2} g = 0, \qquad j = 1,2.$$
It follows that
$$\mathcal{A}_j \mathcal{G} = -\im \int_{\gamma}\bigg(\frac{2}{u-\xi^j} \partial_u g - \frac{2}{(u-\xi^j)^2} g\bigg) du, \qquad j = 1,2.$$
The lemma follows because an integration by parts with respect to $u$ shows that the integral on the right-hand side vanishes. 
\end{proof}

The derivation of formula (\ref{hEstar}) relies on an application of the optional stopping theorem to the martingale observable associated with $\GG$. The following lemma gives an expression for this local martingale in terms of $h$.

\begin{lemma}[Martingale observable for SLE$_\kappa(2)$]\label{greenmartingalelemma}
Let $\theta_t^j = \arg(g_t(z) - \xi_t^j)$, $j = 1,2$. Then
\begin{align}\label{Upsilonhmartingale}
\mathcal{M}_t = \Upsilon_t(z)^{d-2}h(\theta_t^1, \theta_t^2)
\end{align}
is a local martingale for SLE$_\kappa(2)$ started from $(\xi^1, \xi^2)$.
\end{lemma}
\begin{proof}
The proof follows from localization and a direct computation using It\^o's formula using Proposition~\ref{greenPDEprop}. In fact, since
$$\Upsilon_t^{d-2}h(\theta_t^1, \theta_t^2) = |g_t'(z)|^{2-d} \GG(Z_t,\xi_t^1, \xi_t^2),$$
we see that $\mathcal{M}_t$ is the martingale observable relevant for the Green's function found in Section \ref{martingalesec} (cf. equation (\ref{greenintegratedM})). 
\end{proof}
%

Let $z = x  + iy \in \mathbb{H}$ and consider SLE$_\kappa(2)$ started from $(\xi^1, \xi^2)$ with $\xi^1 < \xi^2$. We can without loss of generality assume that $|z| \le 1$.
For each $\epsilon > 0$, we define the stopping time $\tau_\epsilon$ by
$$\tau_\epsilon = \inf\{t \geq 0: \, \Upsilon_t \leq \epsilon \Upsilon_0\},$$
where $\Upsilon_t = \Upsilon_t(z)$. 
Let $\epsilon > 0$ and $n \geq 1$. Then, since $\Upsilon_t$ is a nonincreasing function of $t$ and $\Upsilon_0 = y$, we have
\begin{align}\label{Upsilonbound}
y \geq \Upsilon_{t \wedge n \wedge \tau_\epsilon} \geq \Upsilon_{\tau_\epsilon} = \epsilon y, \qquad t \geq 0.
\end{align}
Hence, in view of the boundedness (\ref{h2estimate1}) of $h$, Lemma \ref{greenmartingalelemma} implies that 
$(\mathcal{M}_{t \wedge \tau_\epsilon \wedge n})_{t \geq 0}$ is a true martingale for SLE$_\kappa(2)$. 
The optional stopping theorem therefore shows that
\begin{align}\label{hoptionalstopping}
h(\theta^1, \theta^2) = \Upsilon_0^{2-d} \E^2\big[\Upsilon_{n \wedge \tau_\epsilon}^{d-2}h(\theta_{n \wedge \tau_\epsilon}^1, \theta_{n \wedge \tau_\epsilon}^2)\big].
\end{align}

Recall that $\PP$ and $\PP^2$ denote the SLE$_\kappa$ and SLE$_\kappa(2)$ measures respectively, and that $\E$ and $\E^2$ denote expectations with respect to these measures.
Equations (\ref{SLEkr_mg}) and (\ref{PP2def}) imply
$$\PP^{2}(V) = \E\big[M_t^{(2)} 1_V\big] \quad \text{for $V \in \mathcal{F}_t$},$$
where 
$$M_t^{(2)} = \left(\frac{\xi^2_t-\xi^1_t}{\xi^2-\xi^1}\right)^a g'_t(\xi^2)^{\frac{3a-1}{2}}.$$
In particular,
\begin{align}\label{E2Eshift}
\E^{2}[f] = \E\big[M_{n \wedge \tau_\epsilon}^{(2)} f\big],
\end{align}
whenever $M_{n \wedge \tau_\epsilon}^{(2)} f$ is an $\mathcal{F}_{n \wedge \tau_\epsilon}$-measurable $L^1(\PP)$ random variable.

\begin{lemma}\label{Mt2lemma}
 For each $t \geq 0$, we have $M_t^{(2)} \in L^1(\PP)$.  
\end{lemma}
\begin{proof}
Since $\xi_t^1$ is the driving function for the Loewner chain $g_t$, we have (see, e.g., Lemma 4.13 in \cite{L2005})
$$\diam(\gamma_t) \leq C \max\Big\{\sqrt{t}, \sup_{0 \leq s \leq t} |\xi_s^1|\Big\}, \qquad t \geq 0.$$
Combining this with \eqref{gtprimexi2bound} and \eqref{xi1xi2diam}, we find
$$|M_t^{(2)}| \le C\, |\xi^2_t-\xi^1_t|^a \leq C  (1 + \diam(\gamma_t))^a \leq C \Big(1 + \max\Big\{\sqrt{t}, \sup_{0 \leq s \leq t} |\xi_s^1|\Big\}\Big)^a, \qquad t \geq 0.$$
Since $\xi_t^1$ is a $\PP$-Brownian motion, it follows that $M_t^{(2)} \in L^1(\PP)$ for each $t \geq 0$.
\end{proof}

As a consequence of (\ref{h2estimate1}), (\ref{Upsilonbound}), and Lemma \ref{Mt2lemma}, the random variable $$M_{n \wedge \tau_\epsilon}^{(2)} \Upsilon_{n \wedge \tau_\epsilon}^{d-2}h(\theta_{n \wedge \tau_\epsilon}^1, \theta_{n \wedge \tau_\epsilon}^2)$$ is $\mathcal{F}_{n \wedge \tau_\epsilon}$-measurable and belongs to $L^1(\PP)$.
Thus, we can use (\ref{E2Eshift}) to rewrite (\ref{hoptionalstopping}) as
\begin{align*}
h(\theta^1, \theta^2) = \Upsilon_0^{2-d} \E\big[M_{n \wedge \tau_\epsilon}^{(2)} \Upsilon_{n \wedge \tau_\epsilon}^{d-2}h(\theta_{n \wedge \tau_\epsilon}^1, \theta_{n \wedge \tau_\epsilon}^2)\big].
\end{align*}
We split this into two terms depending on whether $\tau_\epsilon \leq n$ or $\tau_\epsilon > n$ as follows:
\begin{align}\label{hE}
h(\theta^1, \theta^2) = \Upsilon_0^{2-d} \E \big[M_{\tau_\epsilon}^{(2)} \Upsilon_{\tau_\epsilon}^{d-2}h(\theta_{\tau_\epsilon}^1, \theta_{\tau_\epsilon}^2)1_{\tau_\epsilon \leq n}\big]
+ F_{\epsilon, n}(\theta^1, \theta^2),
\end{align}
where
$$F_{\epsilon, n}(\theta^1, \theta^2) = \Upsilon_0^{2-d} \E\big[M_n^{(2)} \Upsilon_n^{d-2}h(\theta_n^1, \theta_n^2)1_{\tau_\epsilon > n}\big].$$
We prove in Lemma~\ref{lem:jan11} below that $F_{\epsilon, n}(\theta^1, \theta^2) \to 0$ as $n \to \infty$ for each fixed $\epsilon > 0$. 
Assuming this, we conclude from (\ref{hE}) that
\begin{align}\label{hE2}
h(\theta^1, \theta^2) = \Upsilon_0^{2-d}  \lim_{n\to \infty}  \E\big[M_{\tau_\epsilon}^{(2)} \Upsilon_{\tau_\epsilon}^{d-2}h(\theta_{\tau_\epsilon}^1, \theta_{\tau_\epsilon}^2) 1_{\tau_\epsilon \leq n}\big].
\end{align}
Equations (\ref{Gtdef}) and (\ref{twosidedradialdef}) imply
\begin{align}\label{Estargirsanov}
\PP^*\left( V\right) = G_0^{-1}\E[G_t 1_V] \quad \text{for $V \in \mathcal{F}_t$},
\end{align}
where $G_t = \Upsilon_t^{d-2} \sin^\beta\theta^1_t$.
In particular,
$$\E\big[M_{\tau_\epsilon}^{(2)} f\big] = G_0 \E^*\big[G_{\tau_\epsilon}^{-1} M_{\tau_\epsilon}^{(2)} f\big],$$
whenever $M_{\tau_\epsilon}^{(2)} f$ is an $\mathcal{F}_{\tau_\epsilon}$-measurable random variable in $L^1(\PP)$.
Using Lemma \ref{Mt2lemma} again, we see that the function $M_{\tau_\epsilon}^{(2)} \Upsilon_{\tau_\epsilon}^{d-2}h(\theta_{\tau_\epsilon}^1, \theta_{\tau_\epsilon}^2) 1_{\tau_\epsilon \leq n}$ is $\mathcal{F}_{\tau_\epsilon}$-measurable and belongs to $L^1(\PP)$ for $n \geq 1$ and $\epsilon > 0$.
Thus (\ref{hE2}) can be expressed in terms of an expectation for two-sided radial SLE$_\kappa$ through $z$ as follows:
\begin{align}\nonumber
h(\theta^1, \theta^2) & = \Upsilon_0^{2-d} G_0  \lim_{n\to \infty} \E^*\big[G_{\tau_\epsilon}^{-1} M_{\tau_\epsilon}^{(2)} \Upsilon_{\tau_\epsilon}^{d-2}h(\theta_{\tau_\epsilon}^1, \theta_{\tau_\epsilon}^2) 1_{\tau_\epsilon \leq n}\big]
	\\\label{hE3}
& = \Upsilon_0^{2-d} G_0 \E^*\big[G_{\tau_\epsilon}^{-1} M_{\tau_\epsilon}^{(2)} \Upsilon_{\tau_\epsilon}^{d-2}h(\theta_{\tau_\epsilon}^1, \theta_{\tau_\epsilon}^2)\big],
\end{align}
where the second equality is a consequence of dominated convergence and the fact that $\E^*[1_{\tau_\epsilon < \infty}] = 1$.
Using that $G_t = \Upsilon_t^{d-2} \sin^\beta\theta^1_t$, we arrive at
\begin{align}\label{hE4}
h(\theta^1, \theta^2) =  \sin^{\beta}\theta^1 \E^*\big[M_{\tau_\epsilon}^{(2)}\frac{ h(\theta^1_{\tau_\epsilon}, \theta^2_{\tau_\epsilon})}{\sin^{\beta}\theta_{\tau_\epsilon}^1}\big].
\end{align}
In the limit as $\tau_\epsilon \to T$, where $T$ denotes the hitting time of $z$, we have $\theta_{\tau_\epsilon}^2 \to \theta_T^2 = \pi$. Hence we use (\ref{h2ontopedge}) to write (\ref{hE4}) as
\begin{align*}
h(\theta^1, \theta^2) & = \sin^{\beta}\theta^1 \E^*\big[M_{\tau_\epsilon}^{(2)}\big]	
+ E(\theta^1, \theta^2),
\end{align*}
where
\begin{align*}
E(\theta^1, \theta^2) = &\; \sin^{\beta}\theta^1 \E^*\bigg[M_{\tau_\epsilon}^{(2)} \frac{ h(\theta^1_{\tau_\epsilon}, \theta^2_{\tau_\epsilon}) - h(\theta^1_{\tau_\epsilon}, \pi)}{\sin^{\beta}\theta_{\tau_\epsilon}^1} \bigg].
\end{align*}
But the estimate (\ref{h2estimate2}) implies
\begin{align}\label{Etheta1theta2Estar}
|E(\theta^1, \theta^2)| & \leq c \, \E^*\left[M_{\tau_\epsilon}^{(2)}  |\pi - \theta^2_{\tau_\epsilon}| (\sin \theta^1_{\tau_\ee})^{-1}\right]
\end{align}
and, as $\epsilon \downarrow 0$, we will show in Lemma \ref{Estarlemma} below that the right-hand side of (\ref{Etheta1theta2Estar}) tends to zero whereas by Lemma~\ref{uniformly-integrable}, $\E^*[M_{\tau_\epsilon}^{(2)} ] \to \E^*[M_T^{(2)}]$.
Equation (\ref{hEstar}) therefore follows from Lemma \ref{lem:jan11} and Lemma \ref{Estarlemma}, which we now prove.

\begin{lemma}\label{lem:jan11}
Let $$F_{\epsilon, n}(\theta^1, \theta^2) = \Upsilon_0^{2-d} \E\big[M_n^{(2)} \Upsilon_n^{d-2}h(\theta_n^1, \theta_n^2) 1_{\tau_\epsilon > n}\big].$$
For each $\ee > 0$,
\[
\lim_{n \to \infty}F_{\epsilon, n}(\theta^1, \theta^2)  = 0.
\]
\end{lemma}
\begin{proof}
Recall that $|z| \le 1$ and that all constants are allowed to depend on $z$.
Setting $\rho = 2$ in (\ref{SLEkr_mg}), we find
$$M_t^{(2)} = \left(\frac{\xi^2_t-\xi^1_t}{\xi^2-\xi^1}\right)^a g'_t(\xi^2)^{\frac{3a-1}{2}}.$$
Since $ |g'_t(\xi^2)|^{\frac{3a-1}{2}} \le 1$, $\Upsilon_n^{d-2} 1_{\tau_\ee > n} \le C\ee^{d-2}$, and  $h(\theta_n^1, \theta_n^2) \le C \sin^\beta \theta_n^{1}$, it is enough to prove that
\[
\lim_{n \to \infty} \E\left[ |\xi^2_{n}-\xi^1_{n}|^a  \sin^\beta \theta_n^{1}  \right]=0.
\]
For $k=1,2,\ldots$, let $U_k$ be the event that $2^{k} \sqrt{2a n} \le \rad \, \gamma[0,n] \le 2^{k+1} \sqrt{2 a n}$, where $\rad \, K = \sup\{|z|: z \in K\}$. Since our parametrization is such that $\hcap \gamma[0,t] = at$, we have $\Prob(\cup_{k \ge 0}^\infty U_k) =1$.
Fix $0< p < 1/4$. For each integer $j \ge 0 $, let $V_j$ be the event that $2^{j}n^{p} \le |\gamma(n)| \le 2^{j+1}n^{p}$ and let $V_{-1}$ be the event that $|\gamma(n)| < n^{p}$. (We could phrase these events in terms of stopping times.) We have
\[
\E_{}\left[ |\xi^2_{n}-\xi^1_{n}|^a  \sin^\beta \theta_n^{1}  \right] \le \sum_{k \ge 0}\E_{}\left[ |\xi^2_{n}-\xi^1_{n}|^a  \sin^\beta \theta_n^{1} 1_{U_k}\right].
\]
For each $k$, we will estimate 
\[
\E\left[ |\xi^2_{n}-\xi^1_{n}|^a \sin^\beta \theta_n^{1}  1_{U_k}\right] 
 \le \sum_{j=-1}^{J}\E\left[ |\xi^2_{n}-\xi^1_{n}|^a  \sin^\beta \theta_n^{1} 1_{V_j}  1_{U_k}\right]  ,
\]
where $ J=\lceil (\frac{1}{2}-p)\log_2 (n) + k  + \frac{1}{2}\log_2(2a)\rceil.$
By Theorem~1.1 of \cite{FL2015} we have for $j=-1,0,1, \ldots$,
\begin{equation}\label{jan11.1}
\Prob(U_k \cap V_j) \le C (n^{p-1/2}2^{j-k})^\beta \Prob(U_k) \le C (n^{p-1/2}2^{j-k})^\beta.
\end{equation}
On the event $V_{-1} \cap U_k$ we then use the trivial estimate $ \sin \theta_n^{1} \le 1$ to find
\[\E \left[ |\xi^2_{n}-\xi^1_{n}|^a  \sin^\beta \theta_n^{1}  1_{V_{-1}}  1_{U_k}\right] \le C (2^k n^{1/2})^a \Prob(U_k \cap V_{-1}) \le C  (2^k n^{1/2})^a  (n^{p-1/2}2^{-k})^\beta.\]
Since $a-\beta < 0$ for $a > 1/3$ this is summable over $k$, and we see that
\begin{equation}\label{jan11.2}
\E \left[ |\xi^2_{n}-\xi^1_{n}|^a \sin^\beta \theta_n^{1}  1_{V_{-1}} \right] \le C n^{-((1/2-p)\beta -a/2)}. 
\end{equation}
When $j\ge 0$, we can estimate harmonic measure (see the proof of Proposition~\ref{prop:slekr}), to find
\[
 \sin^\beta \theta_n^{1} 1_{V_j} \le C (n^{-p} 2^{-j})^\beta 1_{V_j}, 
\]
and therefore
\[
 |\xi^2_{n}-\xi^1_{n}|^a  \sin^\beta \theta_n^{1}  1_{V_{j}}  1_{U_k} \le C  (n^{1/2}2^k)^a(n^{-p} 2^{-j})^\beta1_{V_{j}}  1_{U_k}. 
\]
So using \eqref{jan11.1},
\[
\sum_{j=0}^{J}\E\left[ |\xi^2_{n}-\xi^1_{n}|^a  \sin^\beta \theta_n^{1}  1_{V_j}  1_{U_k}\right]  \le C_p  2^{-k(\beta-a)} n^{-(\beta/2-a/2)}  (\log_2 n + k).
\]
This is summable over $k$  when  $a>1/3$  for any choice of $p$ and the sum is $o(1)$ as $n\to \infty$. Since $0<p < 1/4$, the exponent in \eqref{jan11.2} is strictly negative whenever $a \ge 1/2$. The proof is complete. 
\end{proof}

\begin{lemma}\label{Estarlemma}
For any $(\theta^1, \theta^2) \in \Delta$, it holds that
\begin{align}\label{Estarlimit1}
\lim_{\epsilon \downarrow 0} \E^* \left[M_{\tau_\epsilon}^{(2)} \right] = \E^*\left[M_T^{(2)}\right]
\end{align}
and
\begin{align}\label{Estarlimit2}
\lim_{\epsilon \downarrow 0}\E^*\left[M_{\tau_\epsilon}^{(2)}  |\pi - \theta^2_{\tau_\epsilon}| (\sin \theta^1_{\tau_\ee})^{-1}\right] = 0.
\end{align}
\end{lemma}
\begin{proof}
For \eqref{Estarlimit1} it is enough to show that the family $\{M_{\tau_\epsilon}^{(2)}\}$ is uniformly integrable. This follows from Lemma~\ref{uniformly-integrable} in the special case $\rho = 2$ which is in the interval $[0, 8-\kappa)$ since $\kappa \le 4$.

It remains to prove \eqref{Estarlimit2}. Note that there is a constant $c$ (depending on $z$) such that
\[
\E^*\left[M_{\tau_\epsilon}^{(2)}  |\pi - \theta^2_{\tau_\epsilon}| (\sin \theta^1_{\tau_\ee})^{-1}\right] \le c \, \ee^{1/2} \, \E^*\left[M_{\tau_\epsilon}^{(2)}  (\sin \theta^1_{\tau_\ee})^{-1}\right].
\]
Indeed, $|\pi - \theta^2_{\tau_\epsilon}|$ is bounded above by a constant times the harmonic measure from $z$ in $H_{\tau_\epsilon}$ of $[\xi^2, \infty)$, which by the Beurling estimate (Lemma~\ref{beurling1}) is $O(\ee^{1/2})$. On the other hand, recalling the definition of the measure $\Prob^*$ and that $\beta - 1 \ge 0$ when $\kappa \le 4$, we see that
\[
\E^*\left[M_{\tau_\epsilon}^{(2)}  (\sin \theta^1_{\tau_\ee})^{-1}\right] 
= \frac{\ee^{d-2}}{\sin^\beta \theta^1} \E\left[M_{\tau_\epsilon}^{(2)} S_{\tau_\epsilon}^{\beta-1} 1_{{\tau_\epsilon} < \infty} \right] 
\le \frac{\ee^{d-2}}{\sin^\beta \theta^1}  \E\left[M_{\tau_\epsilon}^{(2)}1_{{\tau_\epsilon} < \infty} \right]. 
\]
Using Proposition~\ref{prop:slekr} we see that last term converges, and is in particular bounded as $\ee \to 0$. This completes the proof.
\end{proof}

\section{Two paths near the same point: Proof of Lemma \ref{2SLElemma}}\label{lemmasec}
This section proves the correlation estimate Lemma~\ref{2SLElemma} and this will complete the proof of Theorem~\ref{thm:green}.

We could quote Theorem~1.8 of \cite{MW} for a slightly different version of the next lemma, but since our proof is short and slightly different we will give it here.

\begin{lemma}\label{lem:dec10.4}
Suppose $0 < \kappa \le 4$ and $\rho > \max\{-2, \kappa/2-4\}$ and consider an SLE$_{\kappa}(\rho)$ curve $\gamma$ started from $(0,1)$.
Let $C_\infty$ denote the function $C_t$ defined in (\ref{Ctdef}) evaluated at $t = \infty$.
Then there exists a $q > 0$ such that
\begin{equation}\label{sept20.1}
\PP^{\rho }_{0,1} \left( C_{\infty}(1) \le \ee \right) = \tilde{c} \, \ee^{\beta + \rho a}\left(1+O(\ee^{q})\right), \qquad \epsilon \downarrow 0,
\end{equation}
where the constant $\tilde{c} = \tilde{c}(\kappa, \rho)$ is given by
$$\tilde{c}=\frac{\Gamma(6a + a\rho)}{2a\Gamma(2a) \Gamma(4a + a\rho)}.$$
In particular, there is a constant $C< \infty$ such that if $\eta$ is a crosscut separating $1$ from $0$ in $\HH$, then
\[
\PP_{0, 1}^\rho \left( \gamma_\infty \cap \eta \neq \emptyset \right) \le C\, e^{-\pi{(\beta + \rho a)}d_\HH(\RR_-, \eta)}
\]
\end{lemma}
\begin{proof}
Write $\xi^1_t$ for the driving term of $\gamma$ and let $\xi^2_t = g_t(1)$, where $g_t$ is the Loewner chain of $\gamma$. We get SLE$_\kappa(\rho)$ started from $(0, 1)$ by weighting SLE$_\kappa$ by the local martingale (see (\ref{SLEkr_mg}))
\[
M_t^{(\rho)} =(\xi^2_t-\xi^1_t)^{r} g'_t(1)^{\zeta(r)}, \quad r = \rho a/2.\]
Let
\[
N_{t} = C_{t}(1)^{-(\beta + a \rho)}A_{t}^{\beta + a \rho}, \quad A_{t} = \frac{\xi_{t}^{2} - O_{t}}{\xi_{t}^{2} - \xi_{t}^{1}}.
\]
Direct computation shows that $N_{t}$ is a local martingale for SLE$_{\kappa}(\rho)$ started from $(0,1)$, which satisfies $N_{0}=1$. Moreover, 
\[
M_{t}^{(\rho)}N_{t} = M_{t}^{(\kappa-8-\rho)},
\]
where $M_{t}^{(\kappa-8-\rho)}$ is the local SLE$_{\kappa}$ martingale corresponding to the choice $r=r_{\kappa}(\kappa-8-\rho) = -\beta-\rho a/2$. We set
\[
s(t) = \inf\{s \ge 0: C_{s}(1) =e^{-at} \}
\]
and write $\hat{M}_{t}^{\rho} = M_{s(t)}^{(\rho)}$, etc.\ for the time-changed processes.
We have
\begin{align*}
\PP_{0,1}^{\rho}\left( s(t) < \infty \right)& = \E\left[\hat{M}_{t}^{\rho} 1_{s(t) < \infty}\right] \\
& = \E\left[\hat{M}_{t}^{\rho} \hat{N}_{t} \hat{N}_{t}^{-1} 1_{s(t) < \infty}\right] \\
& = e^{-a(\beta+a\rho)t}\E\left[\hat{M}_{t}^{\kappa-8-\rho} \hat{A}_{t}^{-(\beta+a\rho)} 1_{s(t) < \infty}\right] \\
& = e^{-a(\beta+a\rho)t}\tilde\E^{*}\left[ \hat{A}_{t}^{-(\beta+a\rho)}1_{s(t) < \infty}\right],
\end{align*}
where $\tilde\E^{*}$ refers to expectation with respect to SLE$_{\kappa}(\tilde{\rho})$ started from $(0,1)$, where $\tilde{\rho} := \kappa-8 - \rho$. If $\rho > \kappa/2-4$, then $\beta+a\rho > 0$. The key observation is that under the measure $\tilde\PP^{*}$ we have that $s(t) < \infty$ almost surely and that $\hat{A}_{t}$ is positive recurrent and converges to an invariant distribution. This uses that $\rho > \kappa/2-4$ so that $\tilde \rho < \kappa/2-4$; see \cite[Section 5.3]{L-Mink}. Indeed, if we apply Lemma~\ref{lem:invariant} to the SLE$_{\kappa}(\tilde \rho)$ process from $(0,1)$ and note that $-\beta - a\tilde{\rho} = \beta + a\rho$, we find the following formula for the limiting distribution:
\[
\pi(x) = c' \, x^{\beta + a \rho}(1-x)^{2a-1}, \quad c'=\frac{\Gamma(6a + a\rho)}{\Gamma(2a) \Gamma(4a + a\rho)}.
\]
It follows that there exists a $q > 0$ such that
\[
\tilde\E^{*}\left[ \hat{A}_{t}^{-(\beta+a\rho)}1_{s(t) < \infty}\right] = c' \int_{0}^{1}(1-x)^{2a-1}\, dx \left(1 + O(e^{-qt}) \right),
\]
which gives \eqref{sept20.1}. By (\ref{sept20.1}) and the distortion estimates (\ref{Koebe}), if $\tau_{\ee} = \inf\{t\ge 0 : \dist(\gamma_{t},1) \le \ee)$, then
\[
\PP_{0,1}^{\rho}\left( \tau_{\ee} < \infty \right) \asymp \ee^{\beta + a\rho}.
\]
The last assertion follows as in Lemma~\ref{lem:boundary-excursion-ED}.
\end{proof}
\begin{lemma}\label{lemma:dec8.2}
There is a constant $0 < C_1 < \infty$ such that the following holds. 
Let $D$ be a simply connected domain containing $0$ and with three marked boundary points $u,v,w$. Suppose $\gamma_u, \gamma_v$ are crosscuts of $D$ which are disjoint except at $w$, and which connect $w$ with $u$ and $w$ with $v$, respectively, and such that neither crosscut disconnects $0$ from the other. Write $D_u$ and $D_v$ for the components of $0$ of $D \smallsetminus \gamma_u$ and $D\smallsetminus \gamma_v$ and let $D_{u, v} = D_u \cap D_v$. 
For all $\ee > 0$ small enough it holds that if
\begin{equation}\label{dec8.1}
\ee  \, \left( 1+C_1 \sqrt{\ee} \right) < \Upsilon_{D_u}(0)  \le 4\ee \quad \text{and} \quad \Upsilon_{D_v}(0) \ge \sqrt{\ee},
\end{equation}  
then
\[
\Upsilon_{D_{u,v}}(0) > \ee.
\]
\end{lemma}
\begin{proof}
Set $\tilde{\epsilon} := \sqrt{\ee}/2$. Let $\phi_u: D_u \to \mathbb{D}$ be the conformal map with $\phi_u(0)=0, \phi_u'(0) > 0$. Note that $\phi_u(\gamma_v)$ is a crosscut of $\mathbb{D}$. By the distortion estimates (\ref{Koebe}) and the inequalities in (\ref{dec8.1}), $\dist(0,\partial D_{u}) \asymp \Upsilon_{D_u}(0) \asymp \epsilon$ while $\dist(0, \gamma_v) \ge \Upsilon_{D_v}(0)/2 \geq \tilde{\epsilon}$. Therefore conformal invariance, the maximum principle, and the Beurling estimate (Lemma~\ref{beurling1})  show that
\begin{align}\label{hmDDless}
\hm_{\DD}(0, \phi_u(\gamma_v)) = \hm_{D_u}(0, \gamma_v) \le \hm_{\tilde{\epsilon} \DD \smallsetminus \gamma_u}(0, \tilde{\epsilon} \partial \DD)  \leq C (\ee/\tilde{\epsilon})^{1/2}.
\end{align}
We next show that 
\begin{align}\label{hmDDbigger}
\diam( \phi_u(\gamma_v)) \leq C \hm_{\DD}(0, \phi_u(\gamma_v)). 
\end{align}
To prove (\ref{hmDDbigger}), we may assume $\phi_u(\gamma_v)$ separates $1$ from $0$ and that $ \diam(\phi_u(\gamma_v))$ is small. We map $\DD$ conformally to $\HH$ in such a way that $0$ maps to $i$ and $1$ to $0$. Write $\eta$ for the image of $\phi_u(\gamma_v)$ in $\HH$. Let $w \in \eta$ be a point of maximal modulus. There is a curve $\eta_0$ from $0$ to $w$ in $\HH$ that is separated from $\infty$ by $\eta$ in $\HH$ and by conformal invariance and the separation property, $\hm_{\DD}(0, \phi_u(\gamma_v)) = \hm_{\HH}(i, \eta) \ge \hm_{\HH}(i, \eta_0)$. Hall's lemma (see Chapter~III of \cite{GM2005}) yields $\hm_{\HH}(i, \eta_0)  \ge (2/3) \hm_{\HH}(i, \eta_0^*)$, where  $\eta_0^*=\{|z|: z \in \eta_0\}$. Since $\hm_{\HH}(i, \eta_0^*) \geq c\diam(\eta_0^*) \geq c \diam(\eta_0) \geq c\diam( \phi_u(\gamma_v))$,
this proves (\ref{hmDDbigger}). 
By (\ref{hmDDless}) and (\ref{hmDDbigger}), we have
\begin{equation}\label{dec8.2}
\diam ( \phi_u(\gamma_v) ) \le C \, (\ee/\tilde{\epsilon})^{1/2}.
\end{equation}

Write $D'$ for the component containing $0$ of $\mathbb{D} \smallsetminus \phi_u(\gamma_v)$. Let $\psi: D' \to \mathbb{D}$ be the conformal map with $\psi(0)=0, \, \psi'(0)>0$ (since $D' \subset \mathbb{D}$, we actually have $\psi'(0) \geq 1$).
Then the normalized Riemann map of $D_{u,v}$ is $\psi \circ \phi_u$ and so we have
\begin{align}\label{UpsilonDuv0}
\Upsilon_{D_{u,v}}(0) = \Upsilon_{D_u}(0) \cdot \psi'(0)^{-1}.
\end{align}
Using \eqref{dec8.2}, there is $C_0 < \infty$ such that if $\ee > 0$ is small enough, then
\begin{align}\label{psiprime0bound}
1 \le \psi'(0) \le 1 + C_0 \, \ee/\tilde{\epsilon}.
\end{align}
Indeed, this follows, e.g., using Proposition~3.58 of \cite{L2005} and the fact that the half-plane capacity of a hull of diameter $d$ is $O(d^2)$.
Since $\tilde{\epsilon}=\sqrt{\ee}/2$, (\ref{UpsilonDuv0}) and (\ref{psiprime0bound}) imply
\[
\Upsilon_{D_{u,v}}(0) \ge \Upsilon_{D_u}(0) \frac{1}{1 + 2C_0\sqrt{\ee}}.
\]
Consequently, if $\Upsilon_{D_u}(0) > \ee \left(1+C_1 \sqrt{\ee} \right)$ and $C_1 > 2C_0$, then if $\ee$ is small enough,
\[
\Upsilon_{D_{u,v}}(0)  > \ee,
\]
which is the desired estimate.
%
\end{proof}

\begin{lemma}\label{lemma:correlation}
Consider commuting SLE$_\kappa$ in $\mathbb{H}$ started from $(0, 1)$. Let $\gamma_1$ and $\gamma_2$ be the curves starting from $0$ and $1$, respectively.
There exist constants $0<c,C < \infty$ such that if $0< \ee  < c \Im z$, then 
\begin{equation}\label{dec10.1}
\PP_{0,1}\left(\Upsilon^1_\infty(z) \le \ee, \, \Upsilon^2_\infty(z) \le \ee \right) \le C\, ( \ee/\Im z)^{2-d + \beta/2 + a},
\end{equation}
where we write $\Upsilon^1_\infty(z)$ and $\Upsilon^2_\infty(z)$ for $1/2$ times the conformal radius of $z$ in $\mathbb{H}\smallsetminus \gamma_{1,\infty}$ and $\mathbb{H}\smallsetminus  \gamma_{2,\infty}$, respectively.
\end{lemma}
\begin{rem}
Notice that if $a \ge 1/2$, i.e., $\kappa \le 4$, then $\beta/2 \ge 2-d$, i.e., half the boundary exponent is larger than the bulk exponent, with strict inequality if $\kappa < 4$. Therefore, \eqref{dec10.1} implies that for every $\kappa \le 4$,
\begin{equation}\label{dec10.2}
\PP_{0,1}\left(\Upsilon^1_\infty(z) \le \sqrt{\ee}, \, \Upsilon^2_\infty(z) \le \sqrt{\ee} \right) =  o(\ee^{2-d}).
\end{equation}
\end{rem}

\begin{proof}[Proof of Lemma~\ref{lemma:correlation}]
Let $z = x+iy \in \mathbb{H}$. We will write $\PP$ for $\PP_{0,1}$. We first grow $\gamma_1$ starting from $\xi^1=0$. The distribution is that of an SLE$_\kappa(2)$ started from $(0, 1)$. Let $\bar{\ball}_r$ denote the closure of the ball $\ball_r = \ball_r(z)$. 
If $\Upsilon^j_\infty(z) \le \ee$, $j = 1,2$, then $\gamma_{j,\infty}$ intersects $\bar{\ball}_{2\ee}$ by (\ref{Koebe}).
Let $\tau_\ee$ be the first time $\gamma_1$ hits $\bar{\ball}_{2\ee}$ and assume $\ee \le y/10$ say. 
By Proposition~\ref{prop:slekr} we have \begin{equation}\label{dec9.3}\PP \left( \tau_\ee < \infty \right)  \le C\, (\ee/y)^{2-d}.\end{equation}
On the event that $\tau_\ee<\infty$, we stop $\gamma_1$ at $\tau_\ee$ and write $H^1_\ee = \mathbb{H} \smallsetminus \gamma_{1,\tau_\ee}$. Then almost surely, $\eta:=\partial \ball_{2\ee} \cap H^1_{\ee}$ is a crosscut of $H^1_{\ee}$.
We have the following estimate on the extremal distance (using the extension rule)
\begin{equation}\label{dec9.1}
d_{H^1_\ee}([1, \infty), \eta) \ge d_{\mathbb{H}}(\mathbb{R}, \ball_{2\ee}) \ge d_{\ball_y}(\partial \ball_y, \ball_{2\ee}) = \frac{1}{2\pi}\log(y/(2\ee)),
\end{equation}
on the event that $\tau_\ee < \infty$. Conditioned on $\gamma_{1, \tau_\ee}$ (after uniformizing $H^1_\ee$) the distribution of $\gamma_2$ is that of SLE$_\kappa(2)$ started from $(\xi^2_{\tau_\ee}, \xi^1_{\tau_\ee})$. We claim that on the event that $\tau_\ee < \infty$
\begin{equation}\label{dec9.2}
\Prob \left( \gamma_{2, \infty} \cap \eta \neq \emptyset \mid \gamma_{1,\tau_\ee}\right) \le C\, (\ee/y)^{\beta/2 + a}.
\end{equation}
Indeed, note that $g_{\tau_\ee}(\eta)$ is a crosscut of $\mathbb{H}$ separating $\xi^1_{\tau_\ee}$ from $\xi^2_{\tau_\ee}$ and, by conformal invariance and  \eqref{dec9.1},  
\[d_{\mathbb{H}}([ \xi^2_{\tau_\ee}, \infty), g_{\tau_\ee}(\eta)) = d_{H^1_\ee}([1, \infty), \eta) \ge \frac{1}{2\pi} \log (y/(2\ee)) .\]
The estimate \eqref{dec9.2} now follows from Lemma~\ref{lem:dec10.4} with $\rho=2$.  We conclude the proof by combining \eqref{dec9.2} with \eqref{dec9.3}.
\end{proof}
We can now give the proof of Lemma~\ref{2SLElemma}.
\begin{proof}[Proof of Lemma~\ref{2SLElemma}]
Without loss of generality, consider a system of multiple SLEs started from $(0,1)$ with corresponding measure $\Prob=\Prob_{0,1}$.  We want to prove that
\[
\lim_{\ee \downarrow 0}\ee^{d-2} \PP \left(\Upsilon_\infty(z) \le \ee \right) = \lim_{\ee \downarrow 0}\ee^{d-2}\PP \left(\Upsilon^1_\infty(z) \le \ee \right) + \lim_{\ee \downarrow 0}\ee^{d-2}\PP \left(\Upsilon^2_\infty(z) \le \ee \right).
\]
We can write
\begin{align*}
\PP \left(\Upsilon_\infty(z) \le \ee \right)  = & \, \PP \left(\Upsilon^1_\infty(z) \le \ee \right) + \PP \left(\Upsilon^2_\infty(z) \le \ee \right)  \\ 
& -  \PP \left(\Upsilon^1_\infty(z) \le \ee, \,  \Upsilon^2_\infty(z) \le \ee \right) \\
&+ \PP \left(\Upsilon^1_\infty(z) > \ee, \,  \Upsilon^2_\infty(z) > \ee, \, \Upsilon_\infty(z) \le \ee \right).
\end{align*}
We know from Theorem~\ref{prop:slekr} that the renormalized limits of the first two terms on the right exist.
We will show that the remaining terms decay as $o(\ee^{2-d})$ and this will prove the lemma. For the third term the required estimate follows immediately from Lemma~\ref{lemma:correlation}, so it remains to estimate the last term. 

By (\ref{Koebe}), $\Upsilon_\infty \le \ee$ implies $\dist(\gamma_{1,\infty} \cup \gamma_{2,\infty}, z ) \le 2\ee$. We assume that $\dist(\gamma_{1,\infty}, z) \le 2\ee$; the other case is handled in the same way. This in turn implies $\Upsilon^1_\infty(z) \le 4\ee$. Using Lemma~\ref{lemma:dec8.2} we see that there is a constant $C_1$ such that if $\ee \, (1+C_1\sqrt{\ee})< \Upsilon^1_\infty(z) \le 4\ee$ and $\Upsilon^2_\infty(z) \ge \sqrt{\ee}$, then $\Upsilon_\infty(z) > \ee$ for all $\ee$ small enough. Hence we can estimate as follows:
\begin{align*}
\PP (\Upsilon^1_\infty(z) & > \ee, \, \Upsilon^2_\infty(z) > \ee, \, \Upsilon_\infty(z)  \le  \ee, \, \dist(\gamma_{1,\infty}, z) \le 2\ee )  
	 \\
\le & \; \PP \left( \ee \, (1+C_1\sqrt{\ee})< \Upsilon^1_\infty(z) \le 4\ee, \,  \Upsilon^2_\infty(z) < \sqrt{\ee} \right )  
+ \PP \left(\ee \le \Upsilon^1_\infty(z) \le \ee \, (1+C_1\sqrt{\ee})  \right) 
	\\
 \le  & \; \PP \left(\Upsilon^1_\infty(z) \le \sqrt{\ee}, \,  \Upsilon^2_\infty(z) \le \sqrt{\ee} \right) 
+  \PP \left(\ee \le \Upsilon^1_\infty(z) \le \ee(1+C_1\sqrt{\ee})  \right). 
\end{align*}
By \eqref{dec10.2},
\[
 \PP \left( \Upsilon^1_\infty(z) \le \sqrt{\ee}, \,  \Upsilon^2_\infty(z) \le \sqrt{\ee} \right) = o(\ee^{2-d}).
\]
On the other hand, we can use Proposition~\ref{prop:slekr} to see that
\begin{align*}
\PP \left(\ee \le \Upsilon^1_\infty(z) \le \ee(1+C_1\sqrt{\ee})  \right) & = c_*G^{2}(z, 0, 1)\ee^{2-d}\left[ (1+C_1 \sqrt{\ee})^{2-d}-1 \right] + o(\ee^{2-d})\\
& =  o(\ee^{2-d}).\end{align*} 
This completes the proof.
\end{proof}

\section{Fusion}\label{fusionsec}

\subsection{Schramm's formula}
The function $P(z, \xi)$ in (\ref{Pdef}) extends continuously to $\xi = 0$; hence we obtain an expression for Schramm's formula in the fusion limit by simply setting $\xi = 0$ in the formulas of Theorem \ref{thm:schramm}. In this way, we recover the formula of \cite{GC2006} and can give a rigorous proof of this formula.

\begin{thm}[Schramm's formula for fused SLE$_\kappa(2)$]\label{fusedschrammthm}
Let $0 < \kappa \le  4$. Consider chordal SLE$_\kappa(2)$ started from $(0,0+)$. 
Then the probability $P_f(z)$ that a given point $z = x + iy \in \mathbb{H}$ lies to the left of the curve is given by 
\begin{align}\label{Pfdef}
P_f(z) = \frac{1}{c_\alpha} \int_x^\infty \re \mathcal{M}_f(x' +i y) dx',
\end{align}
where $c_\alpha \in \R$ is the normalization constant in (\ref{calphadef}) and
\begin{align}\nonumber
\mathcal{M}_f(z) = &\; y^{\alpha - 2} z^{-\alpha}\bar{z}^{2-\alpha} \int_{\bar{z}}^z(u-z)^{\alpha}(u- \bar{z})^{\alpha - 2} u^{-\alpha} du, \qquad z \in \mathbb{H},
\end{align}
with the contour passing to the right of the origin. 
The function $P_f(z)$ can be alternatively expressed as
\begin{align}\label{PGamsaCardy}
P_f(z) = \frac{\Gamma(\frac{\alpha}{2})\Gamma(\alpha)}{2^{2 - \alpha}\pi \Gamma(\frac{3\alpha}{2} - 1)} \int_{\frac{x}{y}}^\infty S(t') dt',
\end{align}
where the real-valued function $S(t)$ is defined by
\begin{align*}S(t) = &\; (1 + t^2)^{1- \alpha} 
\bigg\{{}_2F_1\bigg(\frac{1}{2} + \frac{\alpha}{2}, 1 - \frac{\alpha}{2}, \frac{1}{2}; - t^2\bigg)
	\\
& - \frac{2\Gamma(1 + \frac{\alpha}{2})\Gamma(\frac{\alpha}{2})t}{\Gamma(\frac{1}{2} + \frac{\alpha}{2})\Gamma(- \frac{1}{2} + \frac{\alpha}{2})}
 {}_2F_1\bigg(1 + \frac{\alpha}{2}, \frac{3}{2} - \frac{\alpha}{2}, \frac{3}{2}; -t^2\bigg)\bigg\}, \qquad t \in \R.
\end{align*}
\end{thm}

\begin{rem}\upshape
  Formula (\ref{PGamsaCardy}) for $P_f(z)$ coincides with equation (15) of \cite{GC2006}. 
\end{rem}
\begin{proof}[Proof of Theorem \ref{fusedschrammthm}]
The expression (\ref{Pfdef}) for $P_f(z)$ follows immediately by letting $\xi \to 0$ in (\ref{Pdef}). 
Since the right-hand side of (\ref{PGamsaCardy}) vanishes as $x\to \infty$, the representation (\ref{PGamsaCardy}) will follow if we can prove that
\begin{align}\label{SMrelation}
\frac{\Gamma(\frac{\alpha}{2})\Gamma(\alpha)}{2^{2 - \alpha}\pi \Gamma(\frac{3\alpha}{2} - 1)}  S(x/y)
= \frac{y}{c_\alpha}\re \mathcal{M}(x+iy, 0), \qquad x \in \R, \ y > 0, \ \alpha > 1.
\end{align}

In order to prove (\ref{SMrelation}), we write $\mathcal{M}_f = y^{\alpha - 2} z^{-\alpha}\bar{z}^{2-\alpha}J_f(z),$
where $J_f(z)$ denotes the function $J(z, \xi)$ defined in (\ref{Jdef}) evaluated at $\xi = 0$, that is,
\begin{align}\label{Jfexp}
J_f(z) = \int_{\bar{z}}^z(u-z)^{\alpha}(u- \bar{z})^{\alpha - 2} u^{-\alpha} du,
\end{align}
where the contour passes to the right of the origin. Let us first assume that $x > 0$. Then we can choose the vertical segment $[\bar{z},z]$ as contour in (\ref{Jfexp}). The change of variables $v = \frac{u-z}{\bar{z} - z}$, which maps the segment $[z, \bar{z}]$ to the interval $[0,1]$, yields
\begin{align*}
J_f(z) = & \; e^{-i\pi \alpha}(z - \bar{z})^{2\alpha - 1}z^{-\alpha} 
\int_0^1  v^{\alpha} 
(1-v)^{\alpha - 2} \Big(1 - v\frac{z-\bar{z}}{z}\Big)^{-\alpha} dv, \qquad x > 0,
\end{align*}
where we have used that $(z - v(z-\bar{z}))^{-\alpha} = z^{-\alpha}(1 - v\frac{z-\bar{z}}{z})^{-\alpha}$ for $v \in [0,1]$ and $x > 0$.
The hypergeometric function ${}_2F_1$ can be defined for $w \in \C \smallsetminus [0, \infty)$ and $0 < b < c$ by\footnote{Throughout the paper, we use the principal branch of ${}_2F_1$ which is defined and analytic for $w \in \C \smallsetminus [1, \infty)$.}
\begin{align}\nonumber
{}_2F_1(a,b,c;w) = \frac{\Gamma(c)}{\Gamma(b)\Gamma(c-b)} \int_0^1 v^{b-1} (1-wv)^{-a} (1-v)^{c-b-1}dv.
\end{align}
This gives, for $x > 0$,
\begin{align}\label{Mfdef2}
\mathcal{M}_f(z) = -iy^{3\alpha - 3} z^{-2\alpha}\bar{z}^{2-\alpha}2^{2\alpha-1}
\frac{\Gamma(\alpha + 1)\Gamma(\alpha-1)}{\Gamma(2\alpha)}
{}_2F_1\Big(\alpha, \alpha + 1, 2\alpha;1 - \frac{\bar{z}}{z}\Big).
\end{align}
The argument $w = 1 - \frac{\bar{z}}{z}$ of ${}_2F_1$ in (\ref{Mfdef2}) crosses the branch cut $[1, \infty)$ for $x = 0$. Therefore, to extend the formula to $x \leq 0$, we need to find the analytic continuation of ${}_2F_1$. This can be achieved as follows. 
Using the general identities 
$${}_2F_1(a,b,c; w) = {}_2F_1(b,a,c; w)$$
and (see \cite[Eq.~15.8.13]{DLMF})
$${}_2F_1(a,b,2b; w) = \Big(1-\frac{w}{2}\Big)^{-a} {}_2F_1\Big(\frac{a}{2},\frac{a+1}{2},b +\frac{1}{2}; \Big(\frac{w}{2-w}\Big)^2\Big),$$
we can write the hypergeometric function in (\ref{Mfdef2}) as
\begin{align}\label{hyperhyper}
{}_2F_1\Big(\alpha, \alpha + 1, 2\alpha;1 - \frac{\bar{z}}{z}\Big)
= \Big(\frac{x}{z}\Big)^{-\alpha-1} {}_2F_1\Big(\frac{\alpha + 1}{2},\frac{\alpha}{2} + 1,\alpha +\frac{1}{2}; -t^{-2}\Big), \qquad x > 0,
\end{align}
where $t = x/y$. Using the identity (see \cite[Eq.~15.8.2]{DLMF})
\begin{align*}
\frac{\sin(\pi(b-a))}{\pi} & \frac{{}_2F_1(a,b,c;w)}{\Gamma(c)}
= \frac{(-w)^{-a}}{\Gamma(b)\Gamma(c-a)}\frac{{}_2F_1(a,a-c+1,a-b+1;\frac{1}{w})}{\Gamma(a-b+1)}
	\\
& - \frac{(-w)^{-b}}{\Gamma(a)\Gamma(c-b)}\frac{{}_2F_1(b,b-c+1,b-a+1;\frac{1}{w})}{\Gamma(b-a+1)}, \qquad w \in \C \smallsetminus [0, \infty),
\end{align*}
with $w = -t^{-2}$ to rewrite the right-hand side of (\ref{hyperhyper}), and substituting the resulting expression into (\ref{Mfdef2}), we find after simplification
\begin{align}\label{MfSt}
\mathcal{M}_f(z) = 
& -\frac{i \sqrt{\pi } 2^{\alpha -1} \bar{z} \Gamma \left(\frac{\alpha -1}{2}\right)}{y^2 \Gamma \left(\frac{\alpha
   }{2}\right)} S(t).
\end{align}
We have derived (\ref{MfSt}) under the assumption that $x > 0$, but since the hypergeometric functions in the definition of $S(t)$ are evaluated at the point $-t^2$ which avoids the branch cut for $z \in \mathbb{H}$, equation (\ref{MfSt}) is valid also for $x \leq 0$. Equation (\ref{SMrelation}) is the real part of (\ref{MfSt}).
\end{proof}

We obtain Schramm's formula for multiple SLE in the fusion limit as a corollary.

\begin{cor}[Schramm's formula for two fused SLEs]\label{schramm2slecor}
Let $0 < \kappa \le  4$. Consider two fused multiple SLE$_\kappa$ paths in $\mathbb{H}$ started from $0$ and growing toward infinity. Then the probability $P_f(z)$ that a given point $z = x + iy \in \mathbb{H}$ lies to the left of both curves is given by (\ref{Pfdef}).
\end{cor}

\begin{rem}\upshape
We remark that the method adopted in \cite{GC2006} was based on exploiting so-called fusion rules, which produces a third order ODE for $P_f$ which can then be solved in order to give the prediction in (\ref{PGamsaCardy}). However, even given the prediction (\ref{PGamsaCardy}) for $P_f$ it is not clear how to proceed to give a proof that it is correct. As soon as the evolution starts, the tips of the curves are separated and the system leaves the fused state, so it seems difficult to apply a stopping time argument in this case. However, as was pointed out to us by one of the referees, if one knows \emph{a priori} that the three probabilities that the point is to the left of, right of, or between the two curves satisfy this third order ODE, and if one is given three linearly independent solutions adding up to $1$ of this ODE with the correct boundary behaviors, then it is possible that this information is enough to deduce the statement of Corollary~\ref{schramm2slecor}. (See \cite{D2015b}.)
\end{rem}

%

\subsection{Green's function}\label{sec:fused-green}
In this subsection, we derive an expression for the Green's function for SLE$_\kappa(2)$ started from $(0, 0+)$. 
Let $\alpha = 8/\kappa$. 
For $\alpha \in (1, \infty) \smallsetminus \Z$, we define the `fused' function $h_f(\theta) \equiv h_f(\theta; \alpha)$ for $0 < \theta < \pi$ by
\begin{align}\label{hfdef}
h_f(\theta) = \frac{\pi  2^{\alpha +1}}{\hat{c}}
 \sin \left(\frac{\pi  \alpha }{2}\right) \sin^{2 \alpha-2}(\theta) \Re\left[e^{-\frac{1}{2} i \pi  \alpha } \, _2F_1\left(1-\alpha,\alpha, 1;\frac{1}{2} (1-i \cot (\theta))\right)\right], 
 \end{align}
where the constant $\hat{c} \equiv \hat{c}(\kappa)$ is defined in (\ref{hatcdef}). This definition is motivated by Lemma \ref{hboundarylemma}, which shows that $h_f(\theta)$ is the limiting value of $h(\theta^1, \theta^2)$ in the fusion limit $(\theta^1, \theta^2) \to (\theta,\theta)$. 
The next lemma shows that this definition of $h_f$ can be extended by continuity to all $\alpha > 1$. 

Given $A \in [0,1]$ and $\epsilon > 0$ small, we let $L_A^j \equiv L_A^j(\epsilon)$, $j = 1, \dots, 4$, denote the contours 
\begin{align}\nonumber
& L_A^1 = [A, i\epsilon] \cup [i\epsilon, -\epsilon], && 
L_A^2 = [-\epsilon, -i\epsilon] \cup [-i\epsilon, A],
	\\ \label{Lcontoursdef}
& L_A^3 = [A, 1-i\epsilon] \cup [1-i\epsilon, 1+\epsilon], && L_A^4 = [1+\epsilon, 1+ i\epsilon] \cup [1+i\epsilon, A],
\end{align}
oriented so that $\sum_1^4 L_A^j$ is a counterclockwise contour enclosing $0$ and $1$, see Figure \ref{Lcontours.pdf}.

\begin{figure}
\medskip
\begin{center}
      \begin{overpic}[width=.7\textwidth]{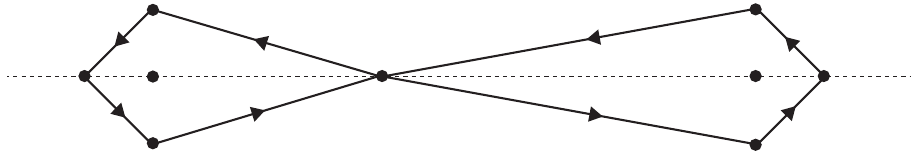}
     \put(27,15){\small $L_A^1$}
     \put(27,0){\small $L_A^2$}
     \put(63,0){\small $L_A^3$}
     \put(63,15.5){\small $L_A^4$}
     \put(102,7.5){\small $\re v$}
     \put(6,5.5){\small $-\epsilon$}
     \put(16,5){\small $0$}
     \put(14,-2.5){\small $-i\epsilon$}
     \put(15.5,17.5){\small $i\epsilon$}
     \put(40.5,5){\small $A$}
     \put(82,5){\small $1$}
     \put(79,-2.5){\small $1-i\epsilon$}
     \put(79,18){\small $1+i\epsilon$}
     \put(91,5.5){\small $1+\epsilon$}
     \end{overpic}
      \begin{figuretext}\label{Lcontours.pdf}
       The contours $L_A^j$, $j = 1, \dots, 4$.
       \end{figuretext}
     \end{center}
\end{figure}

\begin{lemma}
For each $\theta \in (0,\pi)$, the function $h_f(\theta; \alpha)$ defined in (\ref{hfdef}) extends to a continuous function of $\alpha \in (1, \infty)$ satisfying
\begin{align}\label{hfalphaevenodd}
h_f(\theta; n) 
& = 2^{n-3} h_n \sin^{2n-2}\theta^1
\times \begin{cases}
\re \big[2Y_2 - i\pi Y_1\big], & n = 2,4, \dots,
	\\ 
\frac{2}{\pi} \re \big[2i Y_2 + \pi Y_1 \big], & n = 3,5, \dots,
   \end{cases}
\end{align}
where the constant $h_n \in \C$ is defined in (\ref{hndef}) and the coefficients $Y_j \equiv Y_j(\theta; n)$, $j = 1,2$, are defined as follows:
Introduce $y_j \equiv y_j(v,\theta; n)$, $j = 0,1$, by
\begin{align*}
& y_0 = v^{n-1} (1-vz)^{n-1}(1-v)^{-n},
	\\
& y_1 = v^{n-1} (1-vz)^{n-1}(1-v)^{-n}\big(\ln v + \ln(1-vz) - \ln(1-v)\big),
\end{align*}
where $z = \frac{1 - i\cot\theta}{2}$.
Then
\begin{align}\label{fusedY1def}
Y_1 = (2\pi i)^2 \underset{v = 1}{\res} y_0(v,\theta; n)
\end{align} 
and
\begin{align}\label{fusedY2def}
Y_2 = 2\pi i \int_{L_A^1 + L_A^2 + L_A^3 + L_A^4} y_1 dv 
+ 2\pi^2 \int_{L_A^1 - L_A^2 - L_A^3 + L_A^4} y_0 dv, 
\end{align}
where $1/z$ lies exterior to the contours and the principal branch is used for the complex powers throughout all integrations. 
\end{lemma}
\begin{proof}
Let $n \geq 2$ be an integer.  The standard hypergeometric function ${}_2F_1$ is defined by (see \cite[Eq.~15.6.5]{DLMF})
\begin{align}\nonumber
{}_2F_1(a,b,c;z) = & \; \frac{e^{-c\pi i}\Gamma(c)\Gamma(1-b)\Gamma(1+b-c)}{4\pi^2} 
	\\ \label{2F1def}
&\times \int_A^{(0+,1+,0-,1-)} v^{b-1} (1 - vz)^{-a} (1-v)^{c-b-1} dv,
\end{align}
where $A \in (0,1)$, $z \in \C \smallsetminus [1, \infty)$, $b, c-b \neq 1,2,3, \dots$, and $1/z$ lies exterior to the contour. 
Hence, for $\alpha \notin \Z$ and $z \in \C \smallsetminus [1,\infty)$, 
\begin{align}\label{2F1Y}
{}_2F_1(1-\alpha,\alpha, 1; z) = -\frac{1}{4\pi \sin(\pi \alpha)} Y(z; \alpha).
\end{align}
where 
$$Y(z; \alpha) = \int_A^{(0+,1+,0-,1-)} v^{\alpha-1} (1 - vz)^{\alpha -1} (1-v)^{-\alpha} dv.$$
We first show that the function $Y$ admits the expansion
\begin{align}\label{Yalphan}
Y(\theta; \alpha) = (\alpha - n)Y_1 +  (\alpha - n)^2Y_2 + O((\alpha - n)^3), \qquad \alpha \to n,
\end{align}
where $Y_1$ and $Y_2$ are given by (\ref{fusedY1def}) and (\ref{fusedY2def}). 
Let $A \in (0,1)$.
Then 
\begin{align*}
Y(z) = & \bigg\{ (1-e^{-2\pi i \alpha}) \int_{L_A^1}
+ e^{2\pi i(\alpha-1)} (1-e^{-2\pi i \alpha}) \int_{L_A^2} 
+ (e^{2\pi i(\alpha-1)} -1) \int_{L_A^3} 
	\\
& +e^{-2\pi i \alpha} (e^{2\pi i(\alpha-1)} -1) \int_{L_A^4} \bigg\} v^{\alpha-1} (1-vz)^{\alpha-1}(1-v)^{-\alpha} dv.
\end{align*}
Expansion around $\alpha = n$ gives (cf. the proof of (\ref{Falphan})) equation (\ref{Yalphan}) with $Y_2$ given by (\ref{fusedY2def}) and
$$Y_1 = 2\pi i \int_{L_A^1 + L_A^2 + L_A^3 + L_A^4} y_0 dv.$$ 
Since $y_0$ is analytic at $v = 0$ and has a pole at $v = 1$, we see that $Y_1$ can be expressed as in (\ref{fusedY1def}). This proves (\ref{Yalphan}).

Equations (\ref{hfdef}) and (\ref{2F1Y}) give
\begin{align}\label{hfdef2}
h_f(\theta) = -\frac{\pi  2^{\alpha +1}}{\hat{c}}
 \sin \left(\frac{\pi  \alpha }{2}\right) \sin^{2 \alpha-2}(\theta)
 \Re\left[\frac{e^{-\frac{\pi i\alpha}{2}}}{4\pi \sin(\pi \alpha)}Y(z; \alpha)\right].
 \end{align}
 
As $\alpha \to n$, we have
\begin{align*}
& \hat{c}^{-1} = \begin{cases}-\frac{h_n}{(\alpha-n)^2} + \frac{a_n}{\alpha - n} + O(1), & n = 2,4, \dots, \\
 -\frac{h_n}{\alpha-n} + b_n + O(\alpha-n), & n=3,5, \dots,
 \end{cases}
\end{align*}
where $a_n, b_n \in \R$ are real constants. We also have
\begin{align*}
& 2^{\alpha +1} = 2^{n+1}(1 + (\alpha -n)\ln 2 + O((\alpha-n)^2),
	\\
& \sin \left(\frac{\pi  \alpha }{2}\right) = \begin{cases} \frac{(-1)^{\frac{n}{2}}\pi}{2}(\alpha-n) + O((\alpha-n)^3), & n = 2,4, \dots,\\
(-1)^{\frac{n-1}{2}} + O((\alpha-n)^2), & n=3,5, \dots,
\end{cases}
	\\
& \sin^{2\alpha-2}\theta^1 = \sin^{2n-2}(\theta^1)\big(1 + 2\ln(\sin\theta^1)(\alpha - n) + O((\alpha-n)^2)\big),
	\\
& e^{-\frac{\pi i\alpha}{2}} = e^{-\frac{\pi i n}{2}}\Big(1 - \frac{\pi i}{2}(\alpha -n) + O((\alpha-n)^2)\Big),
	\\
& \frac{1}{4\pi\sin(\pi \alpha)} =  \frac{(-1)^{n}}{4\pi^2(\alpha-n)} + O(\alpha-n),
\end{align*}
Substituting the above expansions into (\ref{hfdef2}) and using (\ref{Yalphan}), we obtain, if $n \geq 2$ is even,
\begin{align} \nonumber
h_f(\theta) =&\; \frac{2^{n-2}h_n\sin^{2n-2}(\theta^1) \re Y_1}{\alpha - n} 
+ 2^{n-3} h_n \sin^{2n-2}(\theta^1)
	\\
&\times \re \bigg[2 Y_2 -i\pi  Y_1 + 2 \Big( \ln 2 - \frac{a_n}{h_n}+ 2\ln(\sin\theta^1)\Big)Y_1\bigg]
+ O(\alpha-n),
 \end{align}
while, if $n \geq 2$ is odd,
\begin{align} \nonumber
h_f(\theta) =& -\frac{2^{n-1}h_n\sin^{2n-2}(\theta^1) \im Y_1}{\pi(\alpha - n)} 
+ \frac{2^{n-2} h_n \sin^{2n-2}(\theta^1)}{\pi}
	\\
&\times \re \bigg[2i Y_2 + \pi Y_1 + 2i \Big( \ln 2 - \frac{b_n}{h_n} + 2\ln(\sin\theta^1)\Big)Y_1\bigg]
+ O(\alpha-n).
 \end{align}
In order to establish (\ref{hfalphaevenodd}), it is therefore enough to show that $\re Y_1 = 0$ for even $n$ and that $\im Y_1 = 0$ for odd $n$. 

Consider the function $J(z)$ defined by
$$J(z) = \int_{|v-1|=\epsilon} v^{n-1} (1-vz)^{n-1}(1-v)^{-n}dv, \qquad z \in \C \smallsetminus \{1\},$$
where $\epsilon > 0$ is so small that $1/z$ lies outside the contour. Then, by (\ref{ccidentity}),
$$\overline{J(z)} = -\int_{|v-1|=\epsilon} v^{n-1} (1-v\bar{z})^{n-1}(1-v)^{-n}dv
= -J(\bar{z}), \qquad z \in \C \smallsetminus  \{1\}.$$
Letting $u = 1-v$, we can express $J(z)$ as
$$J(z) = -\int_{|u|=\epsilon} (1-u)^{n-1} (1-(1-u)z)^{n-1}u^{-n}du.$$
The change of variables $u= \frac{z - 1}{z}\tilde{u}$ then yields
$$J(z) = (-1)^{n} \int_{|\tilde{u}|=\epsilon} (z - z\tilde{u} + \tilde{u})^{n-1} (1-\tilde{u})^{n-1} \tilde{u}^{-n} d\tilde{u}
= (-1)^{n-1} J(1-z), \quad z \in \C \smallsetminus\{0,1\}.$$
Hence, if $\re z = 1/2$, 
$$\overline{J(z)} = -J(\bar{z}) = -J(1-z) = (-1)^n J(z).$$
Since
$$Y_1(\theta; n) = 2\pi i J\Big(\frac{1 - i\cot \theta}{2}\Big),$$
it follows that $\re Y_1 = 0$ ($\im Y_1 = 0$) for even (odd) $n$. This completes the proof of the lemma.
 \end{proof}

Taking $\xi \to 0+$ in Theorem \ref{thm:green}, we obtain the following result for SLE$_\kappa(2)$ in the fusion limit.

\begin{thm}[Green's function for fused SLE$_\kappa(2)$]
Let $0< \kappa \le 4$ and consider chordal SLE$_\kappa(2)$ started from $(0,0+)$. Then, for each $z = x + iy \in \mathbb{H}$,
\begin{equation}\label{fusedgreens-formula}
\lim_{\ee \to 0} \ee^{d-2}\PP^2 \left( \Upsilon_\infty(z) \le \ee \right) = c_* \GG_f(z),
\end{equation}
where $\PP^2$ is the SLE$_\kappa(2)$ measure, the function $\GG_f$ is defined by
\begin{align}\label{fusedGGdef}
\GG_f(z) = (\Im z)^{d-2}h_f(\arg z), \qquad
z \in \mathbb{H},
\end{align}
and the constant $c_* = c_*(\kappa)$ is given by (\ref{cstardef}). 
\end{thm}

For any given integer $n \geq 2$, we can compute the integrals in (\ref{fusedY1def}) and (\ref{fusedY2def}) defining  $Y_1$ and $Y_2$ explicitly by taking the limit $\epsilon \to 0$. For the first few simplest cases $n=2,3,4$, this leads to the expressions for the fused SLE$_\kappa$(2) Green's function presented in the following proposition.

\begin{prop}\label{hf234prop}
For $\alpha = 2,3,4$ (corresponding to $\kappa = 4, 8/3, 2$, respectively), the function $h_f(\theta)$ in (\ref{fusedGGdef}) is given explicitly by
\begin{align*}
h_f(\theta) = \begin{cases} \frac{2}{\pi}(\sin\theta - \theta \cos\theta) \sin \theta, & \alpha = 2, 
	\\  
\frac{8}{15 \pi} (4 \theta-3 \sin{2 \theta} 
+2 \theta \cos{2 \theta})\sin ^2\theta, & \alpha = 3,
	\\ 
\frac{1}{12 \pi}(27 \sin\theta + 11 \sin{3 \theta}
-6 \theta (9 \cos\theta + \cos{3 \theta})) \sin^3\theta, & \alpha = 4,
\end{cases} \ 0 < \theta < \pi.
\end{align*}
\end{prop}
\begin{proof}
The proof relies on long but straightforward computations and is similar to that of Proposition \ref{hexplicit234prop}. 
\end{proof}

\begin{rem}\upshape
The formulas in Proposition \ref{hf234prop} can also be obtained by taking the limit $\theta^2 \downarrow \theta^1$ in the formulas of Proposition \ref{hexplicit234prop}.
\end{rem}

In view of Lemma \ref{2SLElemma}, it follows from Theorem \ref{thm:green} that the Green's function for two fused multiple SLEs started from $0$ is given by the symmetrized expression $\GG_f(z) + \GG_f(-\bar{z})$. We formulate this as a corollary.

\begin{cor}[Green's function for two fused SLEs]
Let $0< \kappa \le 4$. Consider a system of two fused multiple SLE$_\kappa$ paths in $\mathbb{H}$ started from $0$ and growing towards $\infty$. Then, for each $z = x + iy \in \mathbb{H}$,
\begin{align*}
\lim_{\ee \to 0} \ee^{d-2}\PP \left( \Upsilon_\infty(z) \le \ee \right) = c_* (\GG_f(z) + \GG_f(-\bar{z})),
\end{align*}
where $d=1+\kappa/8$, the constant $c_* = c_*(\kappa)$ is given by (\ref{cstardef}), and the function $\GG_f$ is defined by (\ref{fusedGGdef}).
\end{cor}

\appendix

\section{The function $\GG(z,\xi^1, \xi^2)$ when $\alpha$ is an integer}\label{alphaintegersec}
In Section \ref{mainsec}, we defined the function $\GG(z,\xi^1, \xi^2)$ for noninteger values of $\alpha = 8/\kappa > 1$ by equation (\ref{GGdef}). We then claimed that $\GG$ can be extended to integer values of $\alpha$ by continuity. The purpose of this section is to verify this claim and to provide formulas for $\GG(z,\xi^1, \xi^2)$ in the case when $\alpha$ is an integer. In particular, we will prove Proposition \ref{hexplicit234prop}.

\subsection{A representation for $h$}
Equations (\ref{Idef}) and (\ref{GGdef}) express $\mathcal{G}$ in terms of an integral with a Pochhammer contour enclosing the variable points $\xi^2$ and $z$. It is convenient to express $\mathcal{G}$ in terms of an integral whose contour encloses the fixed points $0$ and $1$. 
Moreover, instead of considering $\GG$ directly, it is convenient to work with the associated scale invariant function  $h(\theta^1, \theta^2)$ defined in (\ref{hdef}).

\begin{lemma}[Representation for $h$]
Define the function $F(w_1, w_2)$ by
\begin{align}\nonumber
F(w_1, w_2) 
& = \int_A^{(0+,1+,0-,1-)} v^{\alpha -1}(v-w_1)^{\alpha -1} (v-w_2)^{-\frac{\alpha}{2}} (1-v)^{-\frac{\alpha}{2}} dv, 
	\\ \label{Fdefalpha}
& \hspace{8cm} w_1, w_2 \in \C \smallsetminus [0,\infty),
\end{align}
where $A \in (0,1)$ is a basepoint and $w_1, w_2$ are assumed to lie outside the contour. 
For each noninteger $\alpha > 1$, the function $h$ defined in (\ref{hdef}) admits the representation
\begin{align}\label{hF}
h(\theta^1, \theta^2; \alpha) = \frac{\sin^{\alpha-1}\theta^1}{\hat{c}} \im\Big[\sigma(\theta_2) (-e^{i\theta^2})^{\alpha-1} F(w_1, w_2)\Big], \qquad (\theta^1, \theta^2) \in \Delta,
\end{align}
where $w_1 \equiv w_1(\theta^1, \theta^2)$ and $w_2\equiv w_2(\theta^1, \theta^2)$ are given by
\begin{align}\label{w1w2def}
w_1 := 1 - e^{-2i\theta^2}, \qquad w_2 := \frac{1 - e^{-2i\theta^2}}{1 - e^{-2i\theta^1}} = \frac{\sin\theta^2}{\sin\theta^1}e^{-i(\theta^2 - \theta^1)},
\end{align}
the constant $\hat{c}$ is defined in (\ref{hatcdef}), and
\begin{align}\label{sigmadef}
\sigma(\theta^2) = \begin{cases} e^{-i\pi \alpha}, & \theta^2 \geq \frac{\pi}{2}, \\
e^{i\pi \alpha}, & \theta^2 < \frac{\pi}{2}.
\end{cases}
\end{align}
\end{lemma}
\begin{proof}
See \cite{LV2018B}.
\end{proof}

Let $L_A^j := L_A^j(\epsilon)$, $j = 1, \dots, 4$, be the contours defined in (\ref{Lcontoursdef}).

\begin{lemma}\label{halphaintegerlemma}
For each $(\theta^1, \theta^2) \in \Delta$, $h(\theta^1, \theta^2; \alpha)$ extends to a continuous function of $\alpha \in (1, \infty)$ such that
\begin{align}\label{halphaevenodd}
h(\theta^1, \theta^2; n) 
& = h_n \sin^{\alpha - 1}\theta^1
\times \begin{cases}
\im\big[e^{(n-1)i\theta^2} \big(F_2 + i(\theta^2 - 2\pi 1_{\{\theta^2 \geq \frac{\pi}{2}\}}) F_1\big)\big], & n = 2,4, \dots,
	\\ 
\im\big[e^{(n-1)i\theta^2} F_1\big], & n = 3,5, \dots,
   \end{cases}
\end{align}
where the constant $h_n \in \R$ is defined by
\begin{align}\label{hndef}
h_n = \begin{cases} -\frac{i^n \Gamma(\frac{n}{2}) \Gamma(n)}{2 \pi ^3 \Gamma(\frac{3n}{2}-1)}, \qquad n = 2,4, \dots,
	\\
-\frac{i^{n+1} \Gamma(\frac{n}{2}) \Gamma(n)}{4 \pi^2 \Gamma(\frac{3n}{2}-1)}, \qquad n = 3,5,\dots,	
\end{cases}
\end{align}
and the coefficients $F_j \equiv F_j(\theta^1,\theta^2; n)$, $j = 1,2$, are defined as follows:
Let $w_1$ and $w_2$ be given by (\ref{w1w2def}) and define $f_j \equiv f_j(v,\theta_1, \theta_2; n)$, $j = 0,1$, by
\begin{align*}
& f_0 = v^{n-1} (v-w_1)^{n-1}(v-w_2)^{-\frac{n}{2}} (1-v)^{-\frac{n}{2}},
	\\
& f_1 = v^{n-1} (v-w_1)^{n-1}(v-w_2)^{-\frac{n}{2}} (1-v)^{-\frac{n}{2}}\bigg(\ln v + \ln(v-w_1) - \frac{\ln(v-w_2)}{2} - \frac{\ln(1-v)}{2}\bigg).
\end{align*}
Then
\begin{align}\label{F1def}
F_1 = \begin{cases}
\frac{4\pi^2(-1)^{\frac{n}{2}+1}}{(\frac{n}{2} -1)!} \partial_v^{\frac{n}{2}-1}\big|_{v=1}\big(v^{n-1}(v-w_1)^{n-1}(v-w_2)^{-\frac{n}{2}}\big), & n = 2, 4, \dots,
	\\
2\pi i\int_{L_0^3(\epsilon) - L_0^4(\epsilon)} f_0 dv, & n = 3, 5, \dots.
\end{cases}
\end{align} 
and
\begin{align}\label{F2def}
F_2 = 
-2\pi^2 \int_{L_0^3(\epsilon)} f_0 dv + 2\pi i \int_{|v - 1|=\epsilon} f_1 dv, \qquad n = 2, 4, \dots,
\end{align}
where $\epsilon > 0$ is so small that $w_1, w_2$ lie exterior to the contours and the principal branch is used for all complex powers in the integrals. 
\end{lemma}
\begin{proof}
Let $n \geq 2$ be an integer. We first show that the function $F$ defined in (\ref{Fdefalpha}) admits the expansion
\begin{align}\label{Falphan}
F(\theta^1,\theta^2; \alpha) = (\alpha - n)F_1 + (\alpha - n)^2F_2 + O((\alpha - n)^3), \qquad \alpha \to n,
\end{align}
where $F_1$ and $F_2$ are given by (\ref{F1def}) and (\ref{F2def}). 
Define $f(v) \equiv f(v,w_1,w_2, \alpha)$ by
$$f(v) = v^{\alpha -1}(v-w_1)^{\alpha -1} (v-w_2)^{-\frac{\alpha}{2}} (1-v)^{-\frac{\alpha}{2}}.$$
Let $\epsilon > 0$ be small and fix $A \in (\epsilon,1-\epsilon)$.
Then we can rewrite (\ref{Fdefalpha}) as
\begin{align*}
F(&w_1, w_2)  = \bigg(\int_{L_A^1 - L_A^3} 
+ e^{2\pi i \alpha} \int_{L_A^2 + L_A^3}
+ e^{\pi i \alpha} \int_{L_A^4 - L_A^2}
+ e^{-\pi i \alpha} \int_{-L_A^1 - L_A^4}\bigg) f(v) dv
	\\
& = (1- e^{-\pi i \alpha})\bigg(\int_{L_A^1} + e^{2\pi i\alpha}\int_{L_A^2}\bigg) f(v) dv
+ (e^{\pi i \alpha} - e^{-\pi i \alpha})\bigg(e^{\pi i\alpha}\int_{L_A^3} + \int_{L_A^4}\bigg) f(v) dv,
\end{align*}
where the principal branch is used for all complex powers in all integrals. 
Since the integral of $f$ converges at $v = 0$, we can take $\epsilon \to 0$ in the first term on the right-hand side, which gives
\begin{align}\label{Fw1w2exp}
F(w_1, w_2) = (1- e^{-\pi i \alpha})(e^{2\pi i\alpha} - 1) \int_0^A f(v) dv
+ (e^{\pi i \alpha} - e^{-\pi i \alpha})\bigg(e^{\pi i\alpha}\int_{L_A^3} + \int_{L_A^4}\bigg) f(v) dv.
\end{align}
As $\alpha \to n$, we have
\begin{align*}
(1-e^{-\pi i\alpha})(e^{2\pi i\alpha} -1) 
& = 2 i \pi  (-1)^n \left((-1)^n-1\right) (\alpha -n)
- 2\pi^2 (\alpha - n)^2 + O((\alpha-n)^3),
	\\
e^{\pi i\alpha} - e^{-\pi i\alpha} 
& = 2i(-1)^n\pi(\alpha -n) + O((\alpha-n)^3),
	\\
(e^{\pi i\alpha} - e^{-\pi i\alpha}) e^{\pi i\alpha} 
& = 2 i \pi  (\alpha -n) -2 \pi ^2 (\alpha -n)^2 + O((\alpha-n)^3),
\end{align*}
and
\begin{align*}
f(v) = e^{(\alpha -1)\ln v} e^{(\alpha -1)\ln(v-w_1)} 
e^{-\frac{\alpha}{2} \ln(v-w_2)}
e^{-\frac{\alpha}{2} \ln(1-v)} 
= f_0 + (\alpha - n)f_1 + O((\alpha - n)^2).
\end{align*}
Substituting these expansions into (\ref{Fw1w2exp}), we obtain the expansion (\ref{Falphan}) with $F_1$ and $F_2$ given for $n \geq 2$ by
\begin{align*}
F_1 & = 2\pi i(1 - (-1)^n) \int_0^A f_0dv + 2\pi i \bigg(\int_{L_A^3} + (-1)^n \int_{L_A^4}\bigg) f_0 dv,
\end{align*}
and
$$F_2 = 2\pi i(1 - (-1)^n)\int_0^A f_1 dv 
-2\pi^2\bigg(\int_0^A  + \int_{L_A^3}\bigg)f_0 dv 
+ 2\pi i \bigg(\int_{L_A^3} + (-1)^n \int_{L_A^4} \bigg) f_1 dv.$$
The expression (\ref{F1def}) for $F_1$ follows immediately if $n$ is odd. 
If $n$ is even, then $f_0$ has a pole of order $n/2$ at $v = 1$. Thus, choosing $A = 1-\epsilon$ and using the residue theorem, we find
\begin{align}\nonumber
F_1(w_1, w_2) & = 2\pi i \int_{|v-1|=\epsilon} f_0 dv
	\\ \label{F1even}
& = (2\pi i)^2 \underset{v = 1}{\res} \frac{v^{n-1} (v-w_1)^{n-1}(v-w_2)^{-\frac{n}{2}}}{(1-v)^{\frac{n}{2}}}, \qquad n = 2, 4, \dots,
\end{align}
which yields the expression (\ref{F1def}) for $F_1$ also for even $n$. 
Finally, letting $A = 0$, we find the expression (\ref{F2def}) for $F_2$ for $n$ even. This completes the proof of (\ref{Falphan}).

We next claim that, as $\alpha \to n$,
\begin{align}\nonumber
& \im\big[\sigma(\theta^2) (-e^{i\theta^2})^{\alpha-1} F(\theta^1, \theta^2; \alpha)\big]
	\\ \label{imFexpansion}
& = \begin{cases} 
-\im\big[e^{(n-1)i\theta^2} \big(F_2 + i(\theta^2 - 2\pi 1_{\{\theta^2 \geq \frac{\pi}{2}\}}) F_1\big)\big] 
 (\alpha - n)^2 + O((\alpha - n)^3), & n = 2,4, \dots,
	\vspace{.1cm} \\
-\im\big[e^{(n-1)i\theta^2} F_1\big] (\alpha - n) + O((\alpha - n)^2), & n = 3,5,\dots.
\end{cases}
\end{align}
Indeed, the expansion (\ref{Falphan}) yields
\begin{align*}
\im\big[&\sigma(\theta^2) (-e^{i\theta^2})^{\alpha-1} F(\theta^1, \theta^2; \alpha)\big]
= -\im\Big[(1 \mp i\pi (\alpha - n) + \cdots) 
e^{(n-1)i\theta^2} 
	\\
&\times \big(1 + \ln(-e^{i\theta^2})(\alpha - n) + \cdots\big) 
\big(F_1(\alpha - n) + F_2 (\alpha - n)^2 + \cdots\big)\Big]
	\\
=& -\im\big[e^{(n-1)i\theta^2} F_1\big] (\alpha - n)
 -\im\big[e^{(n-1)i\theta^2} \big(F_2 \mp i\pi F_1+ i(\theta^2 - \pi) F_1\big)\big] 
 (\alpha - n)^2
	\\
& + O((\alpha - n)^3).	
\end{align*}
where the upper (lower) sign applies for $\theta^2 \geq \pi/2$ ($\theta^2 < \pi/2$) and we have used that $\ln(-e^{i\theta^2}) = i(\theta^2 - \pi)$ in the last step.
Equation (\ref{imFexpansion}) therefore follows if we can show that 
\begin{align}\label{imF1zero}
\im\Big[e^{(n-1)i\theta^2} F_1\Big] = 0, \qquad n = 2, 4, 6, \dots.
\end{align}

Let $n \geq 2$ be even and define $g(w)$ by 
$$g(w) = 
(1 + e^{-i\theta^2} w)^{n-1}(1 + e^{i\theta^2} w)^{n-1}
(w + \cos\theta^2 - \cot\theta^1 \sin\theta^2)^{-\frac{n}{2}} w^{-\frac{n}{2}}.$$
Then, by (\ref{F1even}),
\begin{align*}
\im\Big[e^{(n-1)i\theta^2} F_1\Big] 
 = &\; \im\Big[e^{(n-1)i\theta^2} 2\pi i \int_{|v-1|=\epsilon} f_0 dv \Big]
	\\
 = &\; \im\Big[e^{(n-2)i\theta^2} 2\pi i \int_{|w|=\epsilon} 
(1 + e^{-i\theta^2} w)^{n-1} (e^{-i\theta^2} w + e^{-2i\theta^2})^{n-1}
	\\
&\times (e^{-i\theta^2}(w + \cos\theta^2 - \cot\theta^1 \sin\theta^2))^{-\frac{n}{2}} (- e^{-i\theta^2} w)^{-\frac{n}{2}} dw \Big]
	\\
=&\; (-1)^{-n/2} \im\Big[2\pi i \int_{|w-1|=\epsilon} g(w) dw \Big]
\end{align*}
where we have used the change of variables $v = 1 + e^{-i\theta^2} w$ and the definitions (\ref{w1w2def}) of $w_1$ and $w_2$ in the second equality.
Since $g(w) = \overline{g(\bar{w})}$, the identity
\begin{align}\label{ccidentity}
\overline{\int_\gamma g(w) dw} = \int_{\bar{\gamma}} \overline{g(\bar{v})} dv,
\end{align}
which is valid for a sufficiently smooth contour $\gamma \subset \C$, implies that $\int_{|w-1|=\epsilon} g(w) dw$ is purely imaginary. 
This proves (\ref{imF1zero}) and hence also (\ref{imFexpansion}).

For each integer $n \geq 2$, we have the following asymptotic behavior of $\hat{c}^{-1}$ as $\alpha \to n$:
\begin{align}\label{hatcexpansion}
\frac{1}{\hat{c}} = \begin{cases}
-\frac{h_n}{(\alpha - n)^2} + O\big(\frac{1}{\alpha - n}\big), & \text{$n$ even},
	  \\ 
- \frac{h_n}{\alpha - n} + O(1), & \text{$n$ odd},
\end{cases}
\quad n = 2,3,4, \dots.
\end{align}
Substituting (\ref{imFexpansion}) and (\ref{hatcexpansion}) into (\ref{hF}), we find (\ref{halphaevenodd}).
\end{proof}

By taking the limit as $\epsilon$ approaches zero in the integrals in (\ref{F1def}) and (\ref{F2def}), it is possible to derive explicit expressions for $F_1$ and $F_2$, and hence also for the function $h$. This leads to a proof of the explicit expressions for the SLE$_\kappa(2)$ Green's function given in Proposition \ref{hexplicit234prop}.

\begin{proof}[Proof of Proposition \ref{hexplicit234prop}]
We give the proof for $\kappa = 4$. The proofs for $\kappa = 8/3$ and $\kappa = 2$ are similar. 
Let $n = 2$. As $\epsilon$ goes to zero, we have
\begin{align*}
\int_{L_0^3(\epsilon)} f_0 dv
& = \int_{L_0^3(\epsilon)} \frac{v (v-w_1)}{(v-w_2)(1-v)} dv 
	\\
& = \frac{w_2 (w_1-w_2) \ln (v-w_2)-(w_1-1) \ln (1-v)+v (1-w_2)}{w_2-1}\bigg|_{v = 0}^{1+\epsilon - i0}
 	\\
& = \frac{-(w_1-1) \ln\epsilon}{w_2-1} 
+ J_1(w_1, w_2) + O(\epsilon),
\end{align*}
where the order one term $J(w_1, w_2)$ is given by
$$J_1(w_1, w_2) = \frac{w_2 (w_1-w_2) (\ln (1-w_2)-\ln (-w_2))-i \pi (w_1-1)-w_2+1}{w_2-1}.$$
On the other hand, since the function 
$$\ln v + \ln(v-w_1) - \frac{\ln(v-w_2)}{2}$$
is analytic at $v = 1$, the residue theorem gives
\begin{align*}
\int_{|v - 1|=\epsilon} f_1 dv
= &\; \int_{|v - 1|=\epsilon} \frac{v (v-w_1)}{(v-w_2)(1-v)}\bigg(\ln v + \ln(v-w_1) - \frac{\ln(v-w_2)}{2} - \frac{\ln(1-v)}{2}\bigg) dv
	\\
= & - 2\pi i \frac{1-w_1}{1-w_2}\bigg(\ln 1 + \ln(1-w_1) - \frac{\ln(1-w_2)}{2}\bigg)
	\\
& - \frac{1}{2}\int_0^{2\pi} \frac{(1 + \epsilon e^{i\varphi}) (1 + \epsilon e^{i\varphi} -w_1)}{(1 + \epsilon e^{i\varphi}-w_2)(-\epsilon e^{i\varphi})} \ln(-\epsilon e^{i\varphi}) i \epsilon e^{i\varphi} d\varphi
	\\
= & - 2\pi i \frac{1-w_1}{1-w_2}\bigg(\ln(1-w_1) - \frac{\ln(1-w_2)}{2}\bigg)
	\\
& + \frac{i }{2}\int_0^{2\pi} \frac{1-w_1}{1-w_2} (\ln \epsilon + i(\varphi - \pi)) d\varphi + O(\epsilon \ln \epsilon)
	\\
= &\;  i\pi \frac{1-w_1}{1-w_2} \ln \epsilon + J_2(w_1, w_2) + O(\epsilon \ln \epsilon),
\end{align*}
where the order one term $J_2(w_1, w_2)$ is given by
$$J_2(w_1, w_2) = - 2\pi i \frac{1-w_1}{1-w_2}\bigg(\ln(1-w_1) - \frac{\ln(1-w_2)}{2}\bigg).$$
Hence, since the singular terms of $O(\ln \epsilon)$ cancel,
\begin{align}\nonumber
F_2 = &\; 
2\pi i \lim_{\epsilon \to 0}\bigg(\pi i \int_{L_0^3} f_0(v) dv + \int_{|v - 1|=\epsilon} f_1(v) dv\bigg)
	\\\label{F2alpha2}
= & \; 2\pi i \big(\pi i J_1(w_1, w_2) + J_2(w_1, w_2)\big), \qquad n = 2.
\end{align}
On the other hand,
\begin{align}\label{F1alpha2}
F_1 & = (2\pi i)^2 \underset{v = 1}{\res} \frac{v (v-w_1)}{(v-w_2)(1-v)} 
= - (2\pi i)^2 \frac{1-w_1}{1-w_2}
= \frac{4\pi^2 e^{-i\theta^2} \sin\theta^1}{\sin(\theta^1 - \theta^2)}, \qquad n = 2.
\end{align}

The terms $J_1$ and $J_2$ involve the logarithms $\ln(1-w_1)$, $\ln(1-w_2)$, and $\ln (-w_2)$. The expressions (\ref{w1w2def}) for $w_1$ and $w_2$ imply that (recall that principal branches are used for all logarithms)
\begin{align}\nonumber
& \ln(1-w_1) = \ln\big(e^{-2i\theta^2}\big) = 2 i \big(\pi 1_{\{\theta^2 \geq \frac{\pi}{2}\}} - \theta^2\big),
	\\\nonumber
& \ln(1-w_2) = \ln\bigg(-e^{-i\theta^2} \frac{\sin(\theta^2 -\theta^1)}{\sin\theta^1}\bigg) 
= \ln\bigg|\frac{\sin(\theta^2 -\theta^1)}{\sin\theta^1}\bigg| + i(\pi - \theta^2),
	\\\label{w1w2logs}
& \ln(-w_2) = \ln\bigg(-e^{-i(\theta^2 - \theta^1)} \frac{\sin \theta^2}{\sin\theta^1}\bigg) 
= \ln\bigg|\frac{\sin \theta^2}{\sin\theta^1}\bigg| + i(\pi + \theta^1 - \theta^2),
\end{align}
for all $(\theta^1, \theta^2) \in \Delta$.
For $n = 2$, equation (\ref{halphaevenodd}) gives 
\begin{align*}
h(\theta^1, \theta^2; 2)  = &\;
\frac{\sin \theta^1}{2\pi^3}\im\Big[e^{i\theta^2} \big(F_2 + i (\theta^2 - 2\pi 1_{\{\theta^2 \geq \frac{\pi}{2}\}}) F_1\big)\Big]. 
\end{align*}
Substituting the expressions (\ref{F2alpha2}) and (\ref{F1alpha2}) into this formula and using (\ref{w1w2logs}), equation (\ref{halpha2}) follows after simplification.
\end{proof}

\bibliographystyle{plain}
\bibliography{is}

\end{document}